\documentclass[11pt,namelimits,sumlimits]{amsart}
\usepackage{amssymb,amsmath}
\usepackage[mathscr]{eucal}
\usepackage{comment} 

\usepackage{psfrag}

\usepackage{graphicx}
\usepackage{amsmath,amsthm}
\usepackage{graphics}
\usepackage{color}
\usepackage{epsfig}
\usepackage{fullpage}
\usepackage{amssymb,amsmath}
\usepackage[mathscr]{eucal}

\def\baselinestretch{1.1187}

\newtheorem{theorem}[equation]{Theorem}
\newtheorem{lemma}[equation]{Lemma}
\newtheorem{prop}[equation]{Proposition}
\newtheorem{cor}[equation]{Corollary}
\newtheorem{corollary}[equation]{Corollary}

\newtheorem{definition}[equation]{Definition}

\theoremstyle{remark}
\newtheorem{remark}[equation]{Remark}
\newtheorem{notation}[equation]{Notation}
\newtheorem{convention}[equation]{Convention}
\newtheorem{assumption}[equation]{Assumption}

\numberwithin{equation}{section}

\newcommand{\tp}{{p,\tau}}

\newcommand{\tpp}{{p,\tau_p}}
\newcommand{\zp}{{p,0}}
\newcommand{\tb}{{\underline{\tau}}}

\newcommand{\vhat}{\widehat{v}}
\newcommand{\uhat}{\widehat{u}}
\newcommand{\hhat}{\widehat{h}}

\newcommand{\tauhat}{\widehat{\tau}}

\newcommand{\gammahat}{{\widehat{\gamma}}} 
\newcommand{\betahat}{{\widehat{\beta}}} 
\newcommand{\varphihat}{\widehat{\varphi}}
\newcommand{\varphinl}{{\varphi_{nl}}}

\newcommand{\phie}{{\phi_{even}}}
\newcommand{\phio}{{\phi_{odd}}}

\newcommand{\gammagl}{{\alpha}}
\newcommand{\eq}{{{eq}}}

\newcommand{\ep}{{{{eq\text{-}pol}}}}
\newcommand{\Lep}{{L_\ep}}

\newcommand{\osc}{{{osc}}}
\newcommand{\ave}{{{avg}}}
\newcommand{\zosc}{{{0,osc}}}
\newcommand{\zave}{{{0,avg}}}
\newcommand{\riza}{{{root}}}

\newcommand{\Xtilde}{{\widetilde{X}}}
\newcommand{\gtilde}{{\widetilde{g}}}
\newcommand{\gtildep}{{\widetilde{g}'}}
\newcommand{\Atilde}{{\widetilde{A}}}

\newcommand{\atilde}{{\widetilde{a}}}
\newcommand{\btilde}{{\widetilde{b}}}
\newcommand{\rtilde}{{\widetilde{r}}}
\newcommand{\Rtilde}{{\widetilde{R}}}
\newcommand{\mutilde}{{\widetilde{\mu}}}

\newcommand{\RRR}{\mathsf{R}}

\newcommand{\Omegatilde}{\widetilde{\Omega}}  
\newcommand{\varphitilde}{\widetilde{\varphi}}

\newcommand{\phihat}{{{\widehat{\phi}}}}  
\newcommand{\skernel}{\mathscr{K}}
\newcommand{\skernelv}{\widehat{\mathscr{K}}}
\newcommand{\val}{\mathscr{V}}

\newcommand{\beq}{\begin{equation}}
\newcommand{\eeq}{\end{equation}}
\newcommand{\bea}{\begin{eqnarray}}
\newcommand{\eea}{\end{eqnarray}}

\newcommand{\R}{\mathbb{R}}

\newcommand{\Z}{\mathbb{Z}}
\newcommand{\N}{\mathbb{N}}

\newcommand{\Sph}{\mathbb{S}}
\newcommand{\Spheq}{\mathbb{S}^2_{eq}}
\newcommand{\Cir}{\mathbb{P}}
\newcommand{\cat}{\mathbb{K}}
\newcommand{\tildecat}{{{\widetilde{\cat}}}}
\newcommand{\PiSph}{\Pi_{\Spheq}}

\newcommand{\tildePiKp}{\widetilde{\Pi}_{\mathbb{K},p}}

\newcommand{\Exp}{\operatorname{Exp}}

\newcommand{\cunder}{\underline{c}\,}

\newcommand{\Stildep}{\widetilde{S}'}
\newcommand{\Shat}{\widehat{S}}

\newcommand{\phiunder}{{\underline{\phi}}}
\newcommand{\Gtau}{{\underline{G}_\tau}}
\newcommand{\Gtaup}{{\underline{G}_{\tp}}}
\newcommand{\Gtpp}{{\underline{G}_{\tpp}}}
\newcommand{\cattp}{{{\cat}_{\tp}}}
\newcommand{\tildecattp}{{\widetilde{\cat}_{\tp}}}
\newcommand{\tildecatmtp}{{\widetilde{\cat}_{p,-\tau}}}
\newcommand{\tildecatztp}{{\widetilde{\cat}_{p,0}}}
\newcommand{\junder}{{\underline{j}}}

\newcommand{\Lcal}{{\mathcal{L}}}
\newcommand{\Lcalp}{{\mathcal{L}'}}
\newcommand{\Lcaltilde}{{\widetilde{\mathcal{L}}}}

\newcommand{\Acal}{{\mathcal{A}}}
\newcommand{\Rcal}{{\mathcal{R}}}
\newcommand{\Jcal}{{\mathcal{J}}}
\newcommand{\Ecal}{{\mathcal{E}}}
\newcommand{\Bcal}{{\mathcal{B}}}
\newcommand{\Ecalinv}{{\mathcal{E}^{-1}_L}}
\newcommand{\Ecalinvep}{{\mathcal{E}^{-1}_{\Lep}}}

\newcommand{\Lgtp}{{\mathcal{L}'_{\gtilde \,}}}

\newcommand{\dbold}{{\mathbf{d}}}

\newcommand{\zetabold}{{\boldsymbol{\zeta}}}
\newcommand{\zetaboldhat}{{\widehat{\boldsymbol{\zeta}}}}
\newcommand{\taubold}{{\boldsymbol{\tau}}}
\newcommand{\tbbold}{{\underline{\taubold}}}
\newcommand{\tauboldhat}{{\widehat{\boldsymbol{\tau}}}}
\newcommand{\mubold}{{\boldsymbol{\mu}}}
\newcommand{\zerobold}{{\boldsymbol{0}}}

\newcommand{\xiw}{{{w}}}
\newcommand{\zw}{{{w}}}
\newcommand{\zwE}{{{w_E}}}
\newcommand{\zwhat}{\widehat{\zw}}
\newcommand{\xiwhat}{\widehat{\xiw}}

\newcommand{\psihat}{\widehat{\psi}}

\newcommand{\Pitilde}{\widetilde{\Pi}}
\newcommand{\avg}{\operatornamewithlimits{avg}}
\newcommand{\sym}{{sym}}
\newcommand{\xx}{{{\ensuremath{\mathrm{x}}}}}
\newcommand{\xxhat}{{{\widehat{\xx}}}}
\newcommand{\xxroot}{{{\xx_\riza}}}
\newcommand{\xxbal}{{{\xx_{balanced}}}}
\newcommand{\fdecay}{{{f_{\Spheq,\xx_1}}}}
\newcommand{\fdecayz}{{{f_{\Spheq,0}}}}
\newcommand{\xxx}{{|\xx|}}
\newcommand{\Acalxx}{{\Acal_{\xx_1}\!}}

\newcommand{\deltaL}{{{\delta_1}}}
\newcommand{\deltaLp}{{{\delta'_1}}}
\newcommand{\deltaz}{{{\delta_0}}}
\newcommand{\deltat}{{{\delta_2}}}

\newcommand{\yy}{\ensuremath{\mathrm{y}}}
\newcommand{\zz}{\ensuremath{\mathrm{z}}}
\newcommand{\rr}{r}

\newcommand{\xxtilde}{{\ensuremath{\widetilde{\mathrm{x}}}}}
\newcommand{\yytilde}{{\ensuremath{\widetilde{\mathrm{y}}}}}

\newcommand{\domTheta}{\text{Dom}_\Theta}
\newcommand{\Pp}{{\widehat{G}}}

\newcommand{\Ptp}{\Phi''}

\newcommand{\Phip}{\Phi'}

\newcommand{\grouptwo}{{\mathscr{G}_{\Spheq,m}}}  
\newcommand{\groupthree}{{\mathscr{G}_{\Sph^3,m}}}

\newcommand{\Xhat}{\widehat{X}}

\newcommand{\Mhat}{\widehat{M}}

\newcommand{\Sp}{S'}
\newcommand{\ghat}{\widehat{g}}

\newcommand{\XXX}{\widehat{\mathsf{X}}}
\newcommand{\YYY}{\widehat{\mathsf{Y}}}
\newcommand{\ZZZ}{\widehat{\mathsf{Z}}}

\newcommand{\xbar}{\underline{\XXX}}
\newcommand{\ybar}{\underline{\YYY}}
\newcommand{\zbar}{\underline{\ZZZ}}
\newcommand{\xbartilde}{{\underline{\mathsf{X}}}}
\newcommand{\ybartilde}{{\underline{\mathsf{Y}}}}
\newcommand{\zbartilde}{{\underline{\mathsf{Z}}}}

\newcommand{\ytilde}{{\mathsf{Y}}}

\newcommand{\arccosh}{\operatorname{arccosh}}

\newcommand{\disjun}{\textstyle\bigsqcup}

\newcommand{\mmer}{{m_{mer}}}
\newcommand{\mpar}{{m_{par}}}

\newcommand{\Lmer}{{L_{mer}}}
\newcommand{\Lpar}{{L_{par}}}

\newcommand{\psicut}{{\psi_{cut}}}
\newcommand{\Psibold}{{\boldsymbol{\Psi}}}

\begin{document}

\title[Doubling]{Minimal surfaces in the round three-sphere \\
by doubling the equatorial two-sphere, I}

\author[N.~Kapouleas]{Nikolaos~Kapouleas}
\address{Department of Mathematics, Brown University, Providence,
RI 02912} \email{nicos@math.brown.edu}

\date{\today}


\begin{abstract}
We construct 
closed embedded minimal surfaces in the round three-sphere
$\mathbb S^3(1)$,
resembling two parallel copies of the equatorial two-sphere $\Spheq$,
joined by small catenoidal bridges 
symmetrically arranged either along two parallel circles of $\Spheq   $,
or along the equatorial circle and the poles.
To carry out these constructions
we refine and reorganize the doubling methodology
in ways which we expect to apply also to further constructions.
In particular we introduce what we call ``linearized doubling'', 
which is an intermediate step where singular solutions 
to the linearized equation are constructed 
subject to appropriate linear and nonlinear conditions.
Linearized doubling provides a systematic approach for dealing 
with the obstructions involved 
and also understanding in detail the regions further away from the catenoidal bridges.
\end{abstract}

\maketitle
\section{Introduction}
\label{Sintro}
\nopagebreak

\subsection*{The general framework}
$\phantom{ab}$
\nopagebreak

This article is an important step in the author's program to develop 
doubling constructions for minimal surfaces by singular perturbation methods.
It is also 
the first article in a series in which we discuss gluing constructions
for closed embedded minimal surfaces in the round three-sphere $\Sph^3(1)$
by doubling the equatorial two-sphere $\Spheq   $.
Doublings of the equatorial two-sphere $\Spheq   $ are important because
their area is close to $8\pi$ (the area of two equatorial two-spheres), 
a feature they share with the celebrated surfaces 
constructed by Lawson in 1970 \cite{L2}.
The classification of the low area 
closed embedded minimal surfaces in the round three-sphere $\Sph^3(1)$, 
especially of those of area close to $8\pi$ 
or less, 
is a natural open question.
This is further motivated by the recent resolutions of 
the Lawson conjecture by Brendle \cite{brendle} 
and the Willmore conjecture by Marques and Neves \cite{neves} 
where they also characterize the Clifford torus and the equatorial sphere
as the only examples of area $\le2\pi^2$. 
We refer to \cite{brendle:survey} for a survey of existence and uniqueness
results for minimal surfaces in the round three-sphere.

The general idea of doubling constructions by gluing methods was proposed and discussed in 
\cite{kapouleas:survey,kapouleas:clifford,alm20}.
The particular kind of gluing methods used relates most
closely to the methods developed in \cite{schoen} and \cite{kapouleas:annals},
especially as they evolved and were systematized in
\cite{kapouleas:wente:announce,kapouleas:wente,kapouleas:imc}.
We refer to \cite{kapouleas:survey} for a general discussion of this gluing methodology 
and to \cite{alm20} for a detailed general discussion of doubling by gluing methods.

Roughly speaking, in such doubling constructions 
one starts with an approximately minimal surface
consisting of two approximately parallel copies of a given minimal surface $\Sigma$
with a number of discs removed and replaced by approximately catenoidal bridges.
The initial surface is then perturbed to minimality by Partial Differential Equations methods.
Understanding such constructions in full generality seems beyond the immediate horizon
at the moment.
In the first such construction \cite{kapouleas:clifford}, 
there is so much symmetry imposed that the
position of the catenoidal bridges is completely fixed and all bridges are identical
modulo the symmetries.
Moreover the bridges are uniformly distributed, 
that is when their number is large enough, 
there are bridges located inside any preassigned domain of $\Sigma$.
Wiygul \cite{wiygul:t,wiygul} has extended that construction to situations 
where the symmetries do not determine the 
vertical (that is perpendicular to $\Sigma$) 
position of the bridges.

In this article for the first time we deal with situations where 
the horizontal position of the bridges is not determined by the symmetries, 
that is the bridges can slide along $\Sigma$,
or there are more than one bridge modulo the symmetries. 
Equally importantly the bridges are not uniformly distributed on $\Sigma$,
that is they stay away from certain fixed domains of $\Sigma$ 
even when the number of the bridges tends to infinity.  
To realize such constructions 
we introduce what we call ``linearized doubling'', 
which is an intermediate step in the construction, 
where singular solutions to the linearized equation are constructed, 
subject to appropriate linear and nonlinear conditions.
Linearized doubling provides a systematic approach for dealing 
with the obstructions involved 
and also provides a detailed understanding of the regions further 
away from the catenoidal bridges.

We expect that linearized doubling will be indispensable in developing further constructions
except in the (rare) cases of exceptionally high symmetry.
Since there is an abundance of such potential constructions, 
linearized doubling will have many further profound applications.
Note for example the potential doubling constructions of free boundary minimal surfaces,  
or of self-shrinkers of the mean curvature flow, 
which we will discuss elsewhere. 

Unlike the case of desingularization constructions,  
doubling constructions generalize to higher dimensions: 
In another article under preparation \cite{kapouleas:high:doubling}, 
we generalize the current results to doubling the equatorial 
$\Sph^{n-1}(1)$ in the round $\Sph^n(1)$
for any $n>3$.
Although the existence of infinitely many 
closed embedded smooth minimal hypersurfaces 
of some simple topological types 
in the round sphere of dimension $n>3$
was established by Hsiang \cite{hsiang1,hsiang2} and of unknown topological type
for $3\le n\le7$ by Marques-Neves \cite{neves:yau}, 
our construction in \cite{kapouleas:high:doubling} provides 
for the first time infinitely many 
topological types of closed embedded smooth minimal hypersurfaces in 
the round sphere of any dimension $n>3$.
Note that the constructions in \cite{kapouleas:high:doubling} 
like the ones in this article are fairly explicit 
with the volume of the hypersurfaces constructed uniformly bounded (depending on the dimension).

We return now to the doublings of the equatorial two-sphere $\Spheq   $ 
constructed in this article and the rest of the series.
All these doublings are symmetric under a group $\groupthree$.
$\groupthree$ is defined (see \ref{Dgroup}) as the group of isometries of $\Sph^3(1)$ which map
$\Lmer$ to itself, where $\Lmer$ (see \ref{ELmer}) is the union of $\mmer$ meridians 
arranged with maximal symmetry.
The centers of the catenoidal bridges we employ in the construction form a set $L$
which we call the configuration of the construction.
$L$ is invariant under $\groupthree$ and therefore we can write $L=\Lmer\cap\Lpar$
where $\Lpar$ is the union of $\mpar$ parallel circles symmetrically arranged with respect
to the equator.
The number of bridges used is therefore $\mmer\mpar$, 
or when the poles (as degenerate circles) are included, 
$\mmer\,(\mpar-2)+2$.
The latitude of the circles in $\Lpar$
(except for the equator and poles if included)
has to be appropriately chosen for the construction to work.
We call this ``horizontal balancing''.
As discussed in \cite{alm20} and later in \ref{rem:eq} 
the construction fails when $L$ lies on an equatorial circle.
We need therefore $\mmer\ge3$ and $\mpar\ge2$.

The perturbation methods we employ require
that the catenoidal bridges are small
so that they do not interact with each other too much.
To ensure this we need the number of catenoidal bridges to be large.
Moreover our current approach relies on a comparison with 
and careful analysis of certain  
rotationally invariant solutions which are controlled by ODEs,  
and this imposes the extra requirement that $\mmer$ is large. 
We only present the two simplest possible cases in this article 
in order to emphasize the fundamental ideas and minimize technical issues: 
In the first case (see theorem \ref{Ttwocir}, also announced and discussed in \cite{alm20}) 
$\mpar=2$ and therefore we have two parallel circles and
the number of catenoidal bridges is $2\mmer$;
in the second case (see theorem \ref{Teq-pol}) 
$\mpar=3$ with 
parallel circles the two poles (which we count as degenerate parallel circles) 
and the equator circle, 
and therefore we have $\mmer+2$ bridges.

The approach here can be extended to apply at least 
to the case when $\mmer$ is large in terms of $\mpar$  \cite{sphere2}. 
The exact limitations of the applicability of this approach are currently under invistigation 
although there certainly exist cases where the ODE model is inadequate,  
as for example when $\mpar$ is large and $\mmer$ small. 
In such cases further ideas will be needed to carry out the construction.

\subsection*{Outline of the approach}
$\phantom{ab}$
\nopagebreak

The constructions in this article and articles in preparation using the same approach are based on the following two main ideas: 
The first idea involves the introduction of an intermediate 
step in the construction, as mentioned earlier, where singular solutions of the linearized 
equation on the given surface being doubled (the equatorial two-sphere in this article) are constructed and analyzed. 
These singular solutions have logarithmic singularities at the points where we plan to place the catenoidal bridges.
The initial surfaces are constructed by gluing the catenoidal bridges to appropriately modified graphs of these singular 
solutions with neighborhoods of the singular points excised.  

More precisely the simplest singular solutions of the linearized equation we consider 
satisfy the linearized equation away from the singularities and 
can be viewed also as multi-Green's functions for the linearized equation. 
We call them 
\emph{linearized doubling} (LD) solutions (see \ref{DLDnoK}).  
If we use an LD solution to construct an initial surface as described above, 
to ensure that the error introduced by the gluing is small, 
the LD solution has to satisfy certain matching conditions. 
Unfortunately the supply of LD solutions which satisfy these matching conditions is inadequate for our purposes. 
This can be remedied however by expanding the class of LD solutions under consideration 
to a larger class of solutions which satisfy the linearized equation 
only modulo a certain space which we call $\skernel[L]$ (see \ref{DLDmodK}) 
which plays also the role of the (extended) substitute kernel used in the linear theory in various earlier constructions 
\cite{alm20,kapouleas:survey,kapouleas:clifford,alm7,haskins:kapouleas:invent,kapouleas:finite,kapouleas:imc,kapouleas:wente,kapouleas:wente:announce,kapouleas:cmp,kapouleas:jdg,kapouleas:annals,kapouleas:bulletin}. 
We call those of the solutions in the expanded class that satisfy the desired matching conditions 
\emph{matched linearized doubling} (MLD) solutions (see \ref{DMLD}). 
MLD solutions are in sufficient supply because by an easy technical step it is possible to convert any LD solution 
(even if it does not satisfy the matching conditions) to a corresponding MLD solution. 
In doing so we trade the failure to satisfy the matching conditions for the failure to satisfy the precise linearized equation. 

It is rather difficult to estimate the LD and MLD solutions carefully so that we have satisfactory control of the construction.  
In particular we need to construct families of MLD solutions which satisfy the balancing and unbalancing conditions as required by 
the general approach (see \cite{kapouleas:survey,alm20} for a discussion of the general approach). 
The second main idea of this article allows us in certain cases to achieve the required control by comparing 
the LD and MLD solutions to certain ODE solutions which can be well understood. 
In particular the study of balancing and unbalancing questions is reduced to the ODE framework. 
The implementation of this idea relies on the rotational invariance of the original surface (the equatorial two-sphere in this article)  
and the largeness of $\mmer$. 
If these conditions are not satisfied the questions involving the LD and MLD solutions (and the corresponding doubling constructions) 
are still open.  

At a more technical level we remark that in this article we experimented with constructing the initial surfaces carefully so that 
they are exactly minimal away from the gluing regions and the support of the functions in $\skernel[L]$. 
This reduces the error terms we have to deal with later, 
at the expense of complicating the construction of the initial surfaces. 
We also note that we organized the presentation so that the results using standard or earlier methodology 
(sections 2, 3, 4 and 7) are separated from the more innovative steps of constructing and analyzing the LD and MLD solutions (sections 5 and 6).

\subsection*{Organization of the presentation}
$\phantom{ab}$
\nopagebreak

The main body of this article consists of three parts.
The first part consists of sections 2, 3, and 4, where we present a general 
construction of initial approximate minimal surfaces based on 
LD solutions and MLD solutions. 
The second part of the paper consists of sections 5 and 6 where we construct and study
in detail the LD and MLD solutions needed for the constructions carried out in this paper.
Finally in the last part which consists of section 7 only we combine the earlier 
results to prove the main results of this paper.

In more detail now, 
in section 2 we review the elementary geometry of the geometric objects we are interested in, 
and we establish the corresponding notation. 
In particular we study aspects of the geometry of the round three-sphere and its equator,
the symmetries we impose,
and the catenoidal bridges we will be using later. 
In section 3 we discuss in detail linearized doubling, 
the LD and MLD solutions, 
and we construct the initial surfaces by gluing MLD solutions and catenoidal bridges.
We also discuss geometric aspects of the initial surfaces 
needed later in understanding their perturbations. 
In section 4 we develop the perturbation theory on the initial surfaces constructed in section 3:  
We solve the linearized equation on the initial surfaces and 
we also estimate the solutions and the corresponding nonlinear terms.
Note that the theory in sections 3 and 4 is developed with a general setting in mind 
(see also \ref{remarkone}) 
and is not restricted to the cases we actually pursue in this article.  

In section 5 we carefully study and estimate the LD and MLD solutions needed for the construction of doublings
where the catenoidal bridges are distributed on two parallel circles. 
In section 6 we do the same in the case where the catenoidal bridges are distributed on the equatorial
circle with two more bridges, one at each pole.
Finally in section 7 we use the MLD solutions constructed in sections 5 and 6 to construct our 
minimal surfaces by using the results of sections 3 and 4.

\subsection*{General notation and conventions}
$\phantom{ab}$
\nopagebreak

In comparing equivalent norms we will find the following notation useful. 

\begin{definition}
\label{Dsimc}
If $a,b>0$ and $c>1$ we write $a\sim_c b$ to mean that the inequalities $a\le cb$
and $b\le ca$ hold.
\end{definition}

We discuss now the H\"{o}lder norms we use.
We use the standard notation $\|u: C^{k,\beta}(\,\Omega,g\,)\,\|$ 
to denote the standard $C^{k,\beta}$-norm of a function or more generally
tensor field $u$ on a domain $\Omega$ equipped with a Riemannian metric $g$.
Actually the definition is completely standard only when $\beta=0$
because then we just use the covariant derivatives and take a supremum
norm when they are measured by $g$.
When $\beta\ne0$ we have to use parallel transport along geodesic segments 
connecting any two points of small enough distance in order to define the H\"{o}lder  
seminorms and this could lead to complications in some cases. 
In this paper we take care to avoid situations where such complications 
may arise and so we will not discuss this issue further.

In this paper we use also weighted H\"{o}lder norms.
The definition we use is somewhat more flexible than the one used in some earlier work
(for example in \cite{kapouleas:finite,kapouleas:annals,kapouleas:wente,kapouleas:clifford,haskins:kapouleas:invent}):
\begin{definition}
\label{D:newweightedHolder}
Assuming that $\Omega$ is a domain inside a manifold,
$g$ is a Riemannian metric on the manifold, 
$\rho,f:\Omega\to(0,\infty)$ are given functions, 
$k\in \N_0$, 
$\beta\in[0,1)$, 
$u\in C^{k,\beta}_{loc}(\Omega)$ 
or more generally $u$ is a $C^{k,\beta}_{loc}$ tensor field 
(section of a vector bundle) on $\Omega$, 
and that the injectivity radius in the manifold around each point $x$ in the metric $\rho^{-2}(x)\,g$
is at least $1/10$,
we define
$$
\|u: C^{k,\beta} ( \Omega,\rho,g,f)\|:=
\sup_{x\in\Omega}\frac{\,\|u:C^{k,\beta}(\Omega\cap B_x, \rho^{-2}(x)\,g)\|\,}{f(x) },
$$
where $B_x$ is a geodesic ball centered at $x$ and of radius $1/100$ in the metric $\rho^{-2}(x)\,g$.
For simplicity we may omit any of $\beta$, $\rho$, or $f$, 
when $\beta=0$, $\rho\equiv1$, or $f\equiv1$, respectively.
\end{definition}

$f$ can be thought of as a ``weight'' function because $f(x)$ controls the size of $u$ in the vicinity of
the point $x$.
$\rho$ can be thought of as a function which determines the ``natural scale'' $\rho(x)$
at the vicinity of each point $x$.
Note that if $u$ scales nontrivially we can modify appropriately $f$ by multiplying by the appropriate 
power of $\rho$. 
Note that from the definition follows that we always have
\begin{equation}
\label{E:norm:derivative}
\| \, \nabla u: C^{k-1,\beta}(\Omega,\rho,g,\rho^{-1}f)\|
\le
\|u: C^{k,\beta}(\Omega,\rho,g,f)\|, 
\end{equation}
and the multiplicative property 
\begin{equation}
\label{E:norm:mult}
\| \, u_1 u_2 \, : C^{k,\beta}(\Omega,\rho,g,\, f_1 f_2 \, )\|
\le
C(k)\, 
\| \, u_1 \, : C^{k,\beta}(\Omega,\rho,g,\, f_1 \, )\|
\,\,
\| \, u_2 \, : C^{k,\beta}(\Omega,\rho,g,\, f_2 \, )\|.
\end{equation}

We will be using extensively cut-off functions,
and for this reason we adopt the following.

\begin{definition}
\label{DPsi} 
We fix a smooth function $\Psi:\R\to[0,1]$ with the following properties:
\newline
(i).
$\Psi$ is nondecreasing.
\newline
(ii).
$\Psi\equiv1$ on $[1,\infty]$ and $\Psi\equiv0$ on $(-\infty,-1]$.
\newline
(iii).
$\Psi-\frac12$ is an odd function.
\end{definition}

Given now $a,b\in \R$ with $a\ne b$,
we define smooth functions
$\psicut[a,b]:\R\to[0,1]$
by
\begin{equation}
\label{Epsiab}
\psicut[a,b]:=\Psi\circ L_{a,b},
\end{equation}
where $L_{a,b}:\R\to\R$ is the linear function defined by the requirements $L(a)=-3$ and $L(b)=3$.

Clearly then $\psicut[a,b]$ has the following properties:
\newline
(i).
$\psicut[a,b]$ is weakly monotone.
\newline
(ii).
$\psicut[a,b]=1$ on a neighborhood of $b$ and 
$\psicut[a,b]=0$ on a neighborhood of $a$.
\newline
(iii).
$\psicut[a,b]+\psicut[b,a]=1$ on $\R$.

Suppose now we have two sections $f_0,f_1$ of some vector bundle over some domain $\Omega$.
(A special case is when the vector bundle is trivial and $f_0,f_1$ real-valued functions).
Suppose we also have some real-valued function $d$ defined on $\Omega$.
We define a new section 
\begin{equation}
\label{EPsibold}
\Psibold[a,b;d](f_0,f_1):=
\psicut[a,b]\circ d \, f_1
+
\psicut[b,a]\circ d \, f_0.
\end{equation}
Note that
$\Psibold[a,b;d](f_0,f_1)$
is then a section which depends linearly on the pair $(f_0,f_1)$
and transits from $f_0$
on $\Omega_a$ to $f_1$ on $\Omega_b$,
where $\Omega_a$ and $\Omega_b$ are subsets of $\Omega$ which contain
$d^{-1}(a)$ and $d^{-1}(b)$ respectively,
and are defined by
$$
\Omega_a=d^{-1}((-\infty,a+\frac13(b-a))),
\qquad
\Omega_b=d^{-1}((b-\frac13(b-a),\infty)),
$$
when $a<b$, and 
$$
\Omega_a=d^{-1}((a-\frac13(a-b),\infty)),
\qquad
\Omega_b=d^{-1}((-\infty,b+\frac13(a-b))),
$$
when $b<a$.
Clearly if $f_0,f_1,d$ are smooth then
$\Psibold[a,b;d](f_0,f_1)$
is also smooth.

\subsection*{Acknowledgments}
The author would like to thank Richard Schoen for his continuous support and interest in the results of this article. 
He would like also to thank the Mathematics Department and the MRC at Stanford University
for providing a stimulating mathematical environment and generous financial support during Fall 2011 and Winter 2012.
The author was also partially supported by NSF grant DMS-1105371. 

\section{Elementary geometry and notation}
\label{AppendixA}
\nopagebreak

\subsection*{The parametrization $\Theta$ and the coordinates $\xx\yy\zz$}
$\phantom{ab}$
\nopagebreak

We consider now the unit three-sphere $\Sph^3(1)\subset\R^4$.
We denote by $(x_1,x_2,x_3,x_4)$ the standard coordinates of $\R^4$
and we define by
\begin{equation}
\label{Eeqtwo}
\Spheq:=\Sph^3(1)\cap\{x_4=0\}
\end{equation}
an equatorial two-sphere in $\Sph^3(1)$.
To facilitate the discussion we fix spherical coordinates $(\xx,\yy,\zz)$ on $\Sph^3(1)$
(see \ref{LThetasymmetries})
by defining a map $\Theta:\R^3\to\Sph^3(1)$ by
\begin{equation}
\label{ETheta}
  \Theta(\xx,\yy,\zz) = 
(\cos\xx \cos\yy \cos\zz, \cos\xx \sin\yy \cos\zz, \sin\xx\cos\zz,\sin\zz).
\end{equation}
Note that in the above notation we can think of $\xx$ as the geographic latitude on $\Spheq$
and of $\yy$ as the geographic longitude.
We will also refer to
\begin{equation}
\label{Eeqone}
\begin{gathered}
\Cir_0:= \Spheq\cap \{x_3=0\}=
\Theta(\{\xx=\zz=0\}),
\\
p_N:=(0,0,1,0)=\Theta(\pi/2,\yy,0),
\\
p_S:=(0,0,-1,0)=\Theta(-\pi/2,\yy,0),
\end{gathered}
\end{equation}
as the equator circle, the North pole, and the South pole of $\Spheq$ respectively.
More generally to facilitate reference to circles of latitude we introduce the notation
for $x\in[-1,1]$
\begin{equation}
\label{Ecilat}
\Cir_{x}
:= \Spheq\cap \{x_3=x\}.
\end{equation}
We have then that $\Cir_{\sin \xx}$ is the circle (or pole) of latitude $\xx$
(which is consistent with the definition of the equator circle $\Cir_0$ above),
$\Cir_{-1}=\{p_S\}$,
and 
$\Cir_{1}=\{p_N\}$.

Clearly the standard metric of $\Sph^3(1)$ is given in the coordinates of \ref{ETheta}
by
\begin{equation}
\label{EThetag}
    \Theta^* g = \cos^2\zz\,( \, d\xx^2 +\cos^2\xx \,d\yy^2 \,)\, + d\zz^2.
\end{equation}
Finally we define a nearest-point projection 
$\PiSph:\Sph^3(1)\setminus\{\,( 0,0,0,\pm1 )\,\}\to\Spheq$
by
\begin{equation}
\label{EPiSph}
\PiSph(x_1,x_2,x_3,x_4)=
\frac1{|(x_1,x_2,x_3,0)|}
(x_1,x_2,x_3,0).
\end{equation}
Clearly we have
\begin{equation}
\label{EPiSphTheta}
\PiSph\circ \Theta(\xx,\yy,\zz)
=
\Theta(\xx,\yy,0).
\end{equation}

We introduce now some convenient notation.

\begin{notation}
\label{NT}
For $X$ a subset of $\Spheq$ we will write $\dbold_X$ for the distance function from $X$,
that is
$\dbold_X(p)$ denotes the distance in $\Spheq$ of 
some $p\in\Spheq$
from $X$ with respect to the standard metric.
Moreover for $\delta>0$ we define a tubular neighborhood of $X$ by
$$
D_X(\delta):=\{p\in\Spheq:\dbold_X(p)<\delta\}.
$$
If $X$ is finite we just enumerate its points in both cases,
for example $\dbold_q(p)$ is the geodesic distance between $p$ and $q$ and
$D_q(\delta)$ is the geodesic disc in $\Spheq$ of center $q$ and radius $\delta$.
\hfill $\square$
\end{notation}

\subsection*{Symmetries of $\Theta$ and symmetries of the construction}
$\phantom{ab}$
\nopagebreak

We first define reflections 
$\xbar$, $\ybar_c$, $\ybar:=\ybar_0$, and $\zbar$ 
in $\R^3$,  
and translations 
$\YYY_c$ 
in $\R^3$,  
where $c\in\R$, by 
\begin{equation}
\label{Edomsym}
\begin{aligned}
\xbar(\xx,\yy,\zz):=(-\xx,\yy,\zz),\qquad
\ybar_c(\xx,\yy,\zz)&:=(\xx,2c-\yy,\zz)\qquad,
\zbar(\xx,\yy,\zz):=(\xx,\yy,-\zz),
\\
\YYY_c(\xx,\yy,\zz)&:=(\xx,\yy+c,\zz).
\end{aligned}
\end{equation}
All these clearly preserve 
\begin{equation}
\label{E:domTheta}
\domTheta:=
\left(-\frac{\pi}{2},\frac{\pi}{2}\right) \times \R \times \left(-\frac{\pi}{2},\frac{\pi}{2}\right). 
\end{equation}
We also define corresponding reflections 
$\xbartilde$, $\ybartilde_c$, $\ybartilde:=\ybartilde_0$, and $\zbartilde$ 
in $\R^4$,  
and rotations $\ytilde_c$
in $\R^4$,  
all of which preserve 
$\Sph^3(1) \subset \R^4$, 
by
\begin{equation}
\label{Esphsym}
\begin{aligned}
\xbartilde(x_1,x_2,x_3,x_4):=&(x_1,x_2,-x_3,x_4),
\\
\ybartilde(x_1,x_2,x_3,x_4):=&(x_1,-x_2,x_3,x_4),
\\
\zbartilde(x_1,x_2,x_3,x_4):=&(x_1,x_2,x_3,-x_4),
\\
\ybartilde_c(x_1,x_2,x_3,x_4):=&( x_1 \cos2c + x_2 \sin2c \, ,\, x_1 \sin2c - x_2 \cos2c \, , \, x_3 \, , \, x_4),
\\
\ytilde_c(x_1,x_2,x_3,x_4):=&(x_1 \cos c -x_2 \sin c \, ,\, x_1 \sin c + x_2 \cos c \, , \, x_3 \, ,\,  x_4).
\end{aligned}
\end{equation}
Note that $\xbartilde$, $\ybartilde$, $\zbartilde$, and $\ybartilde_c$
are reflections with respect to the $3$-planes
$\{x_3=0\}$, $\{x_2=0\}$, $\{x_4=0\}$, and $\ytilde_c(\{x_2=0\})$,
respectively.
$\zbartilde$ fixes $\Spheq$ pointwise and exchanges its two sides in $\Sph^3(1)$.
Clearly $\ytilde_{2\pi}$ is the identity map.
We record the symmetries of $\Theta$ in the following lemma:

\begin{lemma}
\label{LThetasymmetries}
$\Theta$ restricted to
$
\domTheta
$
is a covering map onto 
$\Sph^3(1)\setminus\{x_1=x_2=0\}$.
Moreover the following hold:
\newline
(i).
The group of covering transformations is generated by
$\YYY_{2\pi}$.
\newline
(ii).
$\xbartilde\circ\Theta=\Theta\circ\xbar$,
$\ybartilde_c\circ\Theta=\Theta\circ\ybar_c$,
$\zbartilde\circ\Theta=\Theta\circ\zbar$,
and
$\ytilde_c\circ\Theta=\Theta\circ\YYY_c$.
\end{lemma}

The symmetry group of our constructions depends on a large number $m\in\N$
which we assume now fixed.
We define $\Lmer=\Lmer[m]\subset\Spheq$ to be the union of $m$ meridians symmetrically arranged: 
\begin{equation}
\label{ELmer}
\Lmer=\Lmer[m]:=\Theta(\{(\xx,\yy,0):\xx\in[-\pi/2,\pi/2],\yy=2\pi i/m, i\in\Z\}).
\end{equation}

\begin{definition}
\label{Dgroup}
We denote by $\groupthree$ and $\grouptwo$ the groups of 
isometries of $\Sph^3(1)$ and $\Spheq$ respectively
which fix $\Lmer[m]$ as a set.
\end{definition}

Clearly $\groupthree$ is a finite group
and is generated by the reflections
$\xbartilde$, $\ybartilde$, $\zbartilde$ and $\ybartilde_{\pi/m}$.
$\grouptwo$ can be identified with the subgroup of $\groupthree$
which is generated by
$\xbartilde$, $\ybartilde$, and $\ybartilde_{\pi/m}$.

\subsection*{The linearized equation and rotationally invariant solutions}
$\phantom{ab}$
\nopagebreak

It will be easier later to state some of our estimates if we use a scaled metric on $\Spheq$
and scaled coordinates
$(\xxtilde,\yytilde)$
defined by
\begin{equation}
\label{Egtilde}
\gtilde:=m^2 g_{\Spheq},
\qquad
\xxtilde=m\xx,
\qquad
\yytilde=m\yy.
\end{equation}
To simplify the notation we also define linear operators acting on twice differentiable functions
on domains of $\Spheq$ by
\begin{equation}
\label{ELcalp}
\Lcalp:=\Delta+2,
\qquad\qquad
\Lgtp:=\Delta_{\gtilde}+2m^{-2}
=m^{-2}\Lcalp .
\end{equation}
$\Lcalp$ is of course the linearized operator for the mean curvature on $\Spheq$.

By a rotationally invariant function we mean a function on a domain of $\Spheq$ which depends only on the
latitude $\xx$.
The linearized equation $\Lcalp\phi=0$ amounts to an ODE when the solution $\phi$ is rotationally invariant.
Motivated by this we introduce some notation
to simplify the presentation.

\begin{notation}
\label{Nsym}
Consider a function space $X$ consisting of functions defined on a domain
$\Omega\subset\Spheq$.
If $\Omega$ is invariant under the action of $\grouptwo$
we use a subscript ``$\sym$'' to denote the subspace $X_\sym\subset X$
consisting of those functions in $X$ which are invariant under the action of $\grouptwo$. 
If $\Omega$ is a union of parallel circles we use a subscript ``$\xx$'' to denote 
the subspace of functions $X_{\xx}$ consisting of 
rotationally invariant functions which therefore depend only on $\xx$.
If moreover $\Omega$ is invariant under reflection with respect to the equator of $\Spheq$
we use a subscript ``$|\xx|$'' to denote 
the subspace of functions $X_{\xxx}= X_\xx \cap X_\sym$
consisting of those functions which depend only on $|\xx|$.
\hfill $\square$
\end{notation}
For example we have
$C^0_\xxx(\Spheq)\subset C^0_\sym(\Spheq) \subset C^0(\Spheq)$
and $C^0_\xxx(\Spheq) \subset C^0_\xx(\Spheq)$,
but $C^0_\xx(\Spheq)$ is not a subset of $C^0_\sym(\Spheq)$.

\begin{definition}
\label{Dphieo}
We define rotationally invariant functions
$\phio \in C^\infty_\xx(\Spheq)$ and 
$\phie \in C^\infty_\xxx(\Spheq\setminus\{p_N,p_S\})$ 
by
$$
\phio=\sin\xx,
\qquad
\phie=
1- \sin\xx \, \log\frac{1+\sin\xx}{\cos\xx}
=
1+ \sin\xx \, \log\frac{1-\sin\xx}{\cos\xx} .
$$
\end{definition}

\begin{lemma}
\label{Lphieo}
$\phie$ and $\phio$ are even and odd in $\xx$ respectively.  
They satisfy $\Lcalp\phie=0$ and $\Lcalp\phio=0$.
Moreover $\phie$ is strictly decreasing on $[0,\pi/2)$
where it has a unique root we will denote by $\xx_\riza$.
\end{lemma}

\begin{proof}
$\phie$ corresponds to a translation and is a first harmonic of the Laplacian on $\Spheq$.
$\phio$ is the pushforward of the scaling of the catenoid by the Gauss map and we can finish the proof using this. 
Alternatively it is straightforward to check by direct calculation.
\end{proof}

We discuss now the Green's function for $\Lcalp$ on $\Spheq$:

\begin{lemma}
\label{LGp}
There is a function 
$G\in C^\infty((0,\pi))$
uniquely characterized by (i) and (ii) 
and moreover satisfying (iii-vii) below.
We denote by $\rr$ the standard coordinate of $\R^+$:
\newline
(i). For small $\rr$ we have $G(\rr)=(1+O(\rr^2))\,\log \rr$.
\newline
(ii). 
For each $p\in\Spheq$
we have $\Lcalp G_p=0$
where
$G_p:=G\circ  \dbold_p\in C^\infty(\Spheq\setminus\{p,-p\})$
(recall \ref{NT}).
\newline
(iii). 
$G_{p_N}= (\log 2-1) \, \phio+\phie$ (recall \ref{Eeqone}).
\newline
(iv).
$G(r)\, = \, 1+ \, \cos r \, (\,-1+\log\frac{2\sin r }{1+\cos r}\,)$. 
\newline
(v).
$\frac{\partial G}{\partial r} (r) \, = \, -\sin r \, \log\frac{2\sin r }{1+\cos r} 
+ \frac1{\sin r} + \frac{\sin r \cos r }{1+\cos r} $.  
\newline
(vi).
$
\|\, 
G - 
\,\cos r \, \log \rr
\,
: C^{k}(\,
(0,1)\,,\,
r, dr^2,r^{2}\,)\,\|
\le 
\, C(k) \, .
$ 
\newline
(vii).
$
\|\, 
G 
\,
: C^{k}(\,
(0,1)\,,\,
r, dr^2,\, |\log r| \,)\,\|
\le 
\, C(k) \, .
$ 
\end{lemma}

\begin{proof}
Since $\dbold_{p_N}=\frac\pi2-\xx$ we have by direct calculation 
using \ref{Dphieo} that
$$
(\log 2-1) \phio+\phie = 
\, 1- \, \cos \circ\dbold_{p_N}   \, + \, \cos \circ\dbold_{p_N}   \log\frac{2\sin \circ\dbold_{p_N}   }{1+\cos \circ\dbold_{p_N}  } =
(\, 1+O(\dbold_{p_N}^2) \,) \, \log\circ\dbold_{p_N}.
$$ 
This clearly implies (i-iv).
(v) follows from (iv) by direct calculation.
(vi) follows from (iv) and (v). 
(vii) follows from (vi).  
\end{proof}

For future reference we define a decomposition of functions on domains
of $\Spheq$ as follows.

\begin{definition}
\label{Dave}
Given a function $\varphi$ on some domain $\Omega\subset\Spheq$ we define
a rotationally invariant function $\varphi_\ave$ on 
the union $\Omega'$ of the parallel circles on which $\varphi$ is integrable
(whether contained in $\Omega$ or not),
by requesting that on each such circle $C$ 
$$
\left.\varphi_\ave\right|_C
:=\avg_C\varphi.
$$ 
We also define $\varphi_\osc$ on $\Omega\cap\Omega'$ by
$
\varphi_\osc:=\varphi-\varphi_\ave.
$
\end{definition}

\subsection*{Catenoidal bridges}
$\phantom{ab}$
\nopagebreak

Recall now that a catenoid of size $\tau$ in Euclidean three-space
can be parametrized conformally on a cylinder $\R\times\Sph^1(1)$ by 
\begin{equation}
\label{Ecatenoid}
\begin{gathered}
\Xhat_{cat}(t,\theta):=
(\tau\cosh t \cos\theta, \tau\cosh t \sin \theta, \tau \, t)
= (\rr(t)\cos\theta,\rr(t)\sin\theta,\zz(t)),
\\
\text{ where }
\quad
\rr(t):=\tau\cosh t,
\quad
\zz(t):=\tau\, t.
\end{gathered}
\end{equation}
Alternatively the part above the waist can be given as a radial graph of
a function $\varphi_{cat}:[\tau,\infty)\to\R$ defined by
\begin{multline}
\label{Evarphicat}
\varphi_{cat}(\rr):=
\tau\arccosh \frac \rr \tau=
\tau\left(\log \rr-\log \tau+\log\left(1+\sqrt{1-{\tau^2}{\rr^{-2}}\,}\right)\right)
= \\
=
\tau\left(\log \frac { 2 \rr } {\tau} 
+
\log\left(\frac12+\frac12\sqrt{1-\frac{\tau^2}{\rr^{2}}\,}\right)\right),
\end{multline}
where we denote by $(x^1,x^2,x^3)$ the standard Cartesian coordinates of $\R^3$ and
$\rr$ is the polar coordinate on the $x^1x^2$-plane defined by
$\rr:=\sqrt{(x^1)^2+(x^2)^2\,}$.
By direct calculation or balancing considerations we have for future
reference that
\begin{equation}
\label{Ecatder}
\frac{\partial\varphi_{cat}}{\partial\rr_{\phantom{cat}}}(\rr) 
=
\frac\tau{\sqrt{r^2-\tau^2\,}}.
\end{equation}

\begin{lemma}
\label{Lcatenoid}
$
\|\, \varphi_{cat}(\rr)
-
\tau \log \frac { 2 \rr } {\tau} 
\,
: C^{k}(\,
(9 \tau,\infty)\,,\,
r, dr^2,r^{-2}\,)\,\|
\le 
\, C(k) \, \tau^3.
$ 
\end{lemma}

\begin{proof}
This follows easily from 
\ref{Evarphicat} 
and 
\ref{Ecatder}. 
\end{proof}

Because of the rotational invariance it simplifies the presentation to
use exactly minimal catenoidal bridges in the construction of the minimal surfaces, 
unlike in \cite{kapouleas:clifford,wiygul:t,wiygul} where the catenoidal bridges 
used are only approximately minimal:

\begin{lemma}[$\Gtau$ and $\Gtaup$] 
\label{LGtaup}
For $\tau>0$ small enough 
there is a function 
$\Gtau \in C^0( \, [\tau ,9^{-2})\, ) \cap C^\infty(\, (\tau ,9^{-2})\, ) $
uniquely characterized by (i) and (ii) 
and moreover satisfying (iii):
\newline
(i). 
The initial conditions $\Gtau(\tau)=0$
and $\displaystyle\frac{\partial\Gtau}{dr}(r)\to\infty$
as $r\to\tau+$ hold.
\newline
(ii). 
For each $p\in\Spheq$
the graph of 
$\Gtaup:= \Gtau \circ  \dbold_p$
is minimal in $\Sph^3(1)$
(recall \ref{NT}).
\newline
(iii). 
If $\tau$ is small enough in terms of given $k\in\N$ and $\alpha\in(0,1/2)$, 
then there is a constant $b_{cat}$
such that 
$$
\| \, \Gtau - \varphi_{cat} - b_{cat} 
: C^{k}(\,
(9 \tau,9\,\tau^\alpha)\,,\,
r, dr^2 , r^2 \,)
\, \|
\le 
\, C(k,\alpha) \, \tau \, |\log \tau| \, , 
$$ 
where $b_{cat}$ depends only on $\tau$ and satisfies $|b_{cat}|<C\tau^2$. 
\end{lemma}

\begin{proof}
We can assume without loss of generality that $p=p_N$.
The graph of $\Gtaup$ can be parametrized on the portion
of a cylinder by
$$
Y(r,\theta)
=
(
\sin r \cos \theta \cos \Gtau(r) ,
\sin r \sin \theta \cos \Gtau(r) ,
\cos r \cos \Gtau(r) ,
\sin \Gtau(r) 
)
\in
\Sph^3(1)
\subset
\R^4,
$$
where clearly 
$\PiSph\circ Y(r,\theta)=
(
\sin r \cos \theta ,
\sin r \sin \theta ,
\cos r 
)
\in\Spheq
$ 
(recall \ref{EPiSph}). 
Straightforward calculation implies then that 
\begin{multline*}
\frac{\partial Y}{\partial r} (r,\theta)
=
(\cos r \cos \Gtau(r) - \frac{\partial \Gtau}{\partial r}(r) \, \sin r \, \sin \Gtau (r) \, )
\,
(\cos\theta,\sin\theta,0,0)
+
\\
+
(0,0, -\sin r\cos \Gtau(r)- \frac{\partial \Gtau}{\partial r}(r) \, \cos r \sin \Gtau (r) \, , \,
\frac{\partial \Gtau}{\partial r}(r) \, \cos \Gtau (r) \, ) 
,
\end{multline*}
which implies further that
$$
\left | 
\frac{\partial Y}{\partial r} (r,\theta)
\right |^2
=
\cos^2 \Gtau (r)
+
\left(
\frac{\partial \Gtau}{\partial r}(r)
\right)^2.
$$
We will apply the standard balancing formula 
(see for example \cite{kapouleas:survey,alm20})
with Killing field $\vec{K}$ given by
$$
\left. \vec{K}
\right|_{(x_1,x_2,x_3,x_4)}
:=(0,0,-x_4,x_3) .
$$
Using \ref{EThetag} we calculate that the length of the circle $Y(\{r\}\times\Sph^1)$ 
is 
$2\pi\sin r \cos \Gtau (r)$ and then we have 
\begin{multline*}
\int_{Y(\{r\}\times\Sph^1)} 
\vec{\eta}\cdot \vec{K}
=
2\pi\sin r \cos \Gtau (r)
\left(
\sin r \sin \Gtau (r) \cos \Gtau (r)
+
\frac{\partial \Gtau}{\partial r}(r) 
\cos r
\right)
\cdot
\\
\cdot
\left(
\cos^2 \Gtau (r)
+
\left(
\frac{\partial \Gtau}{\partial r}(r)
\right)^2
\right)^{-1/2}
\end{multline*}
By the balancing formula this is independent of $r$
and so equals the value at $r=\tau$.
We conclude then
\begin{multline*}
\sin r \cos \Gtau (r)
\left(
\sin r \sin \Gtau (r) \cos \Gtau (r)
+
\frac{\partial \Gtau}{\partial r}(r) 
\cos r
\right)
=
\\
=
\left(
\cos^2 \Gtau (r)
+
\left(
\frac{\partial \Gtau}{\partial r}(r)
\right)^2
\right)^{1/2}
\sin \tau \cos \tau.
\end{multline*}
By squaring both sides, calculating, and
solving for $\frac{\partial}{\partial r}\Gtau(r)$, we obtain
\begin{equation}
\label{EGder}
\begin{aligned}
A(r) \, &
\left( \displaystyle\frac{\partial \Gtau}{\partial r}(r) \right)^2
+
\, 2\, B(r)\,
\displaystyle\frac{\partial \Gtau}{\partial r}(r) 
=
\, C(r),
\\
&\displaystyle\frac{\partial \Gtau}{\partial r} 
=
-\frac B A + \sqrt{\frac C A + \frac{B^2}{A^2}  \, } ,
\qquad\text{ where }
\\
A(r):=& \, 
\sin^2 r\, \cos^2 r \, \cos^2\Gtau(r) -\sin^2\tau \, \cos^2\tau,
\\
B(r):=& \, 
\sin^3 r\, \cos r \, \sin  \Gtau(r) \, \cos^3\Gtau(r) ,
\\
C(r):=& \, 
\sin^2\tau \, \cos^2\tau  \, \cos^2\Gtau(r)   -\,   \sin^4 r\, \, \sin^2\Gtau(r)   \, \cos^4\Gtau(r).   
\end{aligned}
\end{equation}

Let $b_{cat}=\Gtau(9\tau)-\varphi_{cat}(9\tau)$.
Using the smooth dependence of ODE solutions on the coefficients
it is easy to confirm that 
$|b_{cat}|<C\tau^2$. 
Using also \ref{Evarphicat} we conclude 
$(9\tau,10\tau)\subset S'$, where 
$S':= \{r\in ( 9\tau, 9\,\tau^\alpha): \Gtau\le10\,\tau\log\frac r \tau\}$. 
Let $S$ be the connected component of $S'$ containing 
$(9\tau,10\tau)$.
We have then on $S$ that 
$\Gtau \le C\tau|\log\tau|$,
and therefore 
\begin{equation*}
\begin{aligned}
A(r)=& \, 
(r^2-\tau^2)\,(1+O(r^2+\tau^2\log^2 \tau\,)\,)
\\
B(r)=& \, 
O(r^3\tau|\log \tau|)
\\
C(r)=& \, 
\tau^2 \,(1+O(r^2+\tau^2\log^2 \tau\,)\,).
\end{aligned}
\end{equation*}
Using \ref{EGder} and \ref{Ecatder} we obtain that on $S$ 
\begin{equation}
\displaystyle\frac{\partial \Gtau}{\partial r} (r)
=
\frac{\partial\varphi_{cat}}{\partial\rr_{\phantom{cat}}}(\rr) 
+O(\tau \,r |\log\tau|\,).
\end{equation}
By integrating we conclude that on $S$ 
$$
\Gtau (r) 
=
\varphi_{cat}(\rr)
+
b_{cat}
+
O(\tau \,r^2 |\log\tau|\,).
$$
We conclude then that on $S$ we have 
$\Gtau\le 8\,\tau\log\frac r \tau$ 
and hence 
$S= (9\tau,9\,\tau^\alpha)$.
Finally using \ref{EGder} we can estimate the higher order derivatives and conclude (iii). 
\end{proof}

\begin{corollary}
\label{Cone}
For $\tau$ small enough 
in terms of given $k\in\N$ and $\alpha\in(0,1/2)$ the following holds.
$$
\|\, 
\Gtau 
-
\tau \log ( 2 \rr / {\tau} )
\,
: C^{k}(\,
(9 \tau,9\,\tau^\alpha)\,,\,
r, dr^2 ,  \tau^{2\alpha} \, |\log  \tau| + \tau^2 r^{-2} \,)
\, \| \, \le \,
C(k,\alpha) \, \tau \, , 
$$ 
\end{corollary}

\begin{proof}
This follows by combining \ref{Lcatenoid} and \ref{LGtaup}.iii
and using that 
$\tau+r^2 |\log  \tau| \le 2 \tau^{2\alpha} \, |\log  \tau| $
on the interval under consideration.
\end{proof}

\begin{corollary}
\label{CGGest}
For $\tau$ small enough 
in terms of given $k\in\N$ and $\alpha\in(0,1/2)$ the following holds.
$$
\| \, \Gtau -\tau G +\tau\log(\tau/2)\,\cos r  
: C^{k}(\,
(\tau^\alpha,9\,\tau^\alpha)\,,\,
\tau^{-2\alpha}dr^2\,)
\, \|
\le 
\, C(k,\alpha) \, \tau^{1+2\alpha} \, |\log \tau| .
$$ 
\end{corollary}

\begin{proof}
We have 
$
G-\cos r \,\log\frac\tau2 - \log\frac{2r}{\tau}
\,=\,
(1-\cos r )\,(1+\log\frac\tau2 -\log r ) +\cos r\,\log\frac{2\sin r}{r (1+\cos r) } \, 
$
by an easy calculation based on \ref{LGp}.iv. 
This implies that 
$$
\|\, 
G-\cos r \,\log(\tau/2) - \log(2r/\tau)
\,
: C^{k}(\,
(9 \tau,9\,\tau^\alpha)\,,\,
r, dr^2 ,  \, r^2 \,)
\, \| \, \le \,
C(k,\alpha) \, |\log \tau| \, . 
$$ 
Combining this with \ref{Cone} we conclude the proof.
\end{proof}

\begin{convention}
\label{con:alpha}
We fix now some small $\alpha>0$ which we will assume as small in absolute terms
as needed.
\hfill $\square$
\end{convention}

\begin{definition}
\label{DExp}
For $\tau\in(0,1)$ and $p\in\Sph^3(1)$ we 
define 
$\Exp_\tp:= \RRR_{p,\tau}\circ \exp_{ g , p }$ 
where 
$\RRR_{p,\tau}:T_p\Sph^3(1)\to T_p\Sph^3(1)$ is defined by $\RRR_{p,\tau}(v)=\tau v$ 
and 
$\exp_{ g , p }$ denotes 
the exponential map of $(\Sph^3(1), g)$ at $p$.  
Let $B_\tp\subset T_p\Sph^3(1)$ be the ball 
such that the restriction of $\Exp_\tp$ to $B_\tp$ is a diffeomorphism onto $\Sph^3(1)\setminus \{-p\}$.
We define a metric on $B_\tp$ 
by 
$\gtilde_\tp \, := \, \Exp^*_\tp \, ( \tau^{-2}\, g )$. 
\end{definition}

\begin{lemma}
\label{LExp}
The estimate 
$\| \, \gtilde_\tp - h_p \, : \, C^k  ( \, B_0( 9^2 \tau^{\alpha-1} ) \setminus \{0\} \, , \, \Rtilde \, , \, h_p \, , \, \Rtilde^4\, ) \,\|
\, \le \, C(k) \, \tau^2 $ 
holds, 
where $B_0( 9^2 \tau^{\alpha-1} ) \, \subset T_p\Sph^3(1) $ is the ball centered at the origin and 
of radius $9^2 \tau^{\alpha-1}$ with respect to $h_p$,  
$\Rtilde$ denotes the distance from the origin in the $h_p$ metric, 
and $h_p$ is the Euclidean metric on $T_p\Sph^3(1)$ defined by $h_p := \left. g \right |_p = \left. \gtilde_\tp \right |_p$.  
\end{lemma}

\begin{proof}
Clearly 
$
h_p = d\Rtilde^2 \, + \, \Rtilde^2 \, g_{\Sph^2(1) }
$
and
$
\gtilde_\tp = d\Rtilde^2 \, + \, \tau^{-2} \, \sin^2 ( \tau \Rtilde ) \,\, g_{\Sph^2(1) }. 
$
By calculating and using the definitions we obtain
$\| \, \Rtilde^2 \, g_{\Sph^2(1) }
 \, : \, C^k  ( \, B_\tp \setminus \{0\} \, , \, \Rtilde \, , \, h_p \, , \, \Rtilde^2\, ) \,\|
\, \le \, C(k) $ 
and 
\\ 
${ \| \, \tau^{-2} \, \Rtilde^{-2} \, \sin^2 ( \tau \Rtilde ) \, - \, 1 
 \, : \, C^k  ( \, B_0( 9^2 \tau^{\alpha-1} ) \setminus \{0\} \, , \, \Rtilde \, , \, h_p \, , \, \Rtilde^2\, ) \,\|
\, \le \, C(k) \, \tau^2 } $. 
Using \ref{E:norm:mult} we complete the proof.
\end{proof}

\begin{definition}
\label{Dtildecat}
We define the catenoidal bridge $\cattp$ centered at $p\in\Spheq$ and of waist size $\tau$
to be the union of the graphs of $\pm \Gtaup$ restricted to 
$D_p(9\tau^\alpha)\setminus D_p(\tau)$ where 
$\alpha$ is as in \ref{con:alpha}.
For $\tau>0$ small enough 
we define (recall \ref{DExp}) 
$
\tildecattp := \tildecatmtp:= 
\Exp^{-1}_\tp (\cattp). 
$
Finally we define $\tildecat_{\zp}$
to be the standard catenoid in the Euclidean space 
$( T_p \Sph^3(1) \, , \, h_p )$, 
appropriately placed so that $\tildecattp$ depends smoothly on $\tau$
for $|\tau|$ small enough.
\end{definition}

Note that the last statement above 
applies since 
$\tildecattp$ is controlled by 
an ODE with initial conditions at the waist independent of $\tau$ 
and coefficients smoothly depending on $\tau$.

\begin{definition}
\label{DtildePiK}
For $p\in\Sph^2$ 
we define $\tildePiKp$ to be the nearest point projection
from an appropriate neighborhood of $\tildecat_{\zp}$ in 
$( T_p \Sph^3(1) \, , \, h_p )$ 
to $\tildecat_{\zp}$.
We also define $\rtilde:T_p \Sph^3(1)\to\R$ 
to be the distance from the axis of $\widetilde{\cat}_{\zp}$ 
in $T_p \Sph^3(1)$ with respect to 
$h_p$. 
\end{definition}

\begin{lemma}
\label{LPiK}
For $\tau$ small enough the restriction of $\tildePiKp$ to $\tildecattp$ 
is well defined and is moreover a smooth diffeomorphism onto a domain 
$\Omegatilde_\tau\subset \widetilde{\cat}_{\zp}$.
Moreover $\tildecattp$ is the graph in the Euclidean space 
$( T_p \Sph^3(1) \, , \, h_p )$ 
over 
$\Omegatilde_\tau \subset \widetilde{\cat}_{\zp}$ 
of a function $\varphitilde_\tau$ which satisfies 
$$
\| \, \varphitilde_\tau
: C^{k}(\,
\Omegatilde_\tau \,,\,
\rtilde \, ,  \, \gtilde_0 \, , 
\, \tau+\,\tau^2\,\rtilde^2 \, |\log\tau| \,)
\, \|
\le 
\, C(k,\alpha) \, , 
$$ 
where $\gtilde_0$ is the metric on $\widetilde{\cat}_{\zp}$ 
induced by the Euclidean metric 
$h_p$ 
on 
$T_p \Sph^3(1)$ 
and $\rtilde$ is as in \ref{DtildePiK}.
\end{lemma}

\begin{proof}
We assume without loss of generality that $p=p_N$.
We identify then $T_p\Sph^3(1)$ with $\R^3$ so that for 
$\vec{u}= (u_1,u_2,u_3)\in\R^3$ 
we have
$$
\Exp_\tp (u_1,u_2,u_3) = 
\cos\tau |\vec{u}| \,\,(0,0,1,0)
\,\,+\,\,
\frac {\sin\tau |\vec{u}|} { |\vec{u}|} \,\,(u_1,u_2,0,u_3).
$$
where $|\vec{u}|=(u_1^2+u_2^2+u_3^2\,)^{1/2}$. 
The upper half of 
$\cattp$ 
can be parametrized by 
$X_\tp:[\tau,9\tau^\alpha)\times \Sph^1\to\Sph^3(1)$ 
defined by (recall \ref{ETheta})  
$$
X_\tp(r,\theta) 
=
\cos\Gtau(r)\,\, (\sin r \cos \theta, \sin r \sin \theta, \cos r, 0)
\,\,+\,\,
\sin\Gtau(r)\,\, (0,0,0,1).
$$
The upper half of 
$ \widetilde{\cat}_{\zp} $
can be parametrized by 
$\Xtilde_{\zp} :[\tau,\infty) \times \Sph^1\to T_p\Sph^3(1)$
defined by (recall \ref{Evarphicat})  
$$
\Xtilde_{\zp} (r,\theta) 
=
\tau^{-1} \,\,
(\, r\cos\theta\, ,\, r\sin\theta\, ,\, \varphi_{cat}(r)\,), 
$$
and therefore
\begin{multline}
\Exp_\tp \circ 
\Xtilde_{\zp} (r,\theta) 
=
\cos\sqrt{r^2+\varphi_{cat}^2(r)\,} \,\,(0,0,1,0)
\,\,+\,\,
\\
\,\,+\,\,
\frac 
{\sin\sqrt{r^2+\varphi_{cat}^2(r)\,} } 
{\sqrt{r^2+\varphi_{cat}^2(r)\,} } 
(\, r\cos\theta\, ,\, r\sin\theta\, ,\,0\,,\, \varphi_{cat}(r)\,). 
\end{multline}
Using \ref{LGtaup}.iii and \ref{Lcatenoid} we conclude 
$$
\| \, 
X_\tp 
-
\Exp_\tp \circ \Xtilde_{\zp} 
: C^{k}(\,
(9 \tau,9\,\tau^\alpha)\times\Sph^1 \,,\,
r\, , 
\,(\Exp_\tp \circ \Xtilde_{\zp})^* g \, ,\, \tau+r^2 |\log\tau| \,)
\, \|
\le 
\, C(k,\alpha) \, \tau \, . 
$$ 
This implies that the points of 
$\tildecattp = \Exp^{-1}_\tp (\cattp)$
are within distance 
$C\,\tau^{2\alpha}|\log\tau|$
from 
$\widetilde{\cat}_{\zp}$, 
where for the region $\{\rtilde<10\}$
we use the smooth dependence on $\tau$.
The restriction hence of $\tildePiKp$ to $\tildecattp$ is well defined.
Magnifying and using the implicit function theorem we conclude the proof.
\end{proof}

Since by \ref{LPiK} $\tildecattp$ is a small perturbation of a domain  
$\Omegatilde_\tau\subset \widetilde{\cat}_{\zp}$,
its first and second fundamental forms induced by $h_p$ 
are a small perturbation of those of 
$\Omegatilde_\tau$. 
However we are really interested in the first and second fundamental forms 
$\gtilde$ and $\Atilde$ of $\tildecattp$ induced by 
$\gtilde_\tp$ (defined in \ref{DExp}), 
or equivalently 
\begin{equation} 
\label{EgAtilde}
\gtilde=\tau^{-2}  \, \Exp^{*}_\tp g,
\qquad\quad
\Atilde=\tau^{-1}  \, \Exp^{*}_\tp A, 
\end{equation} 
where $g$ and $A$ denote the first and second fundamental forms of $\cattp\subset \Sph^3(1)$ 
induced by the standard metric of $\Sph^3(1)$ 
and 
$\Exp^{*}_\tp$ denotes pullback by the restriction of $\Exp_\tp$ to $\tildecattp$.  
The next corollary provides the estimates we will need. 

\begin{corollary}
\label{CPiK}
For 
$\Omegatilde_\tau$,  
$\gtilde_0$,  
and
$\rtilde$ as in \ref{LPiK} 
we have that
$$
\begin{aligned}
\| \, (\tildePiKp)_* \gtilde \, - \, \gtilde_0 \,
&: C^{k}(\,
\Omegatilde_\tau \,,\,
\rtilde \, ,  \, \gtilde_0 \, , 
\, \tau \rtilde +\,\tau^2\,\rtilde^4 \,)
\, \|
\le 
\, C(k,\alpha) \, , 
\\
\| \, (\tildePiKp)_* \Atilde \, - \, \Atilde_0 \,
&: C^{k}(\,
\Omegatilde_\tau \,,\,
\rtilde \, ,  \, \gtilde_0 \, , 
\, \tau +\,\tau^2\,\rtilde^3 \,)
\, \|
\le 
\, C(k,\alpha) \, , 
\end{aligned}
$$
where $(\tildePiKp)_*$ denotes the pushforward by $\tildePiKp$ restricted to $\tildecattp$, 
that is the pullback by its inverse,  
and 
$\gtilde$
and 
$\Atilde$
are as above. 
\end{corollary}

\begin{proof}
Let $q\in \Omegatilde_\tau \subset \widetilde{\cat}_{\zp}\subset T_p\Sph^3(1)$ 
and consider the Euclidean metric 
$\hhat_q:= \rtilde^{-2}(q)\, h_p$ on $T_p\Sph^3(1)$.  
Consider Cartesian orthonormal coordinates on $( T_p\Sph^3(1) , \hhat_q ) $ 
so that $\Xhat_{cat}$ defined as in \ref{Ecatenoid} with $\tau=1$ 
provides a conformal parametrization of $\widetilde{\cat}_{\zp}$. 
We consider the geodesic disc $B'_q \subset \widetilde{\cat}_{\zp}$ 
with center $q$, radius $1/10$, and defined with respect to the metric induced $\hhat_q$. 
It is then easy to check that there is a constant $C(k)$ which depends only on $k$ 
such that on $\Xhat^{-1}_{cat}( \, B'_q )$ we have that the $C^k$ norms of the coordinates of 
$\Xhat_{cat}$ are bounded by $C(k)$ and also 
$g_{cyl} \le  \, C(k)\, \Xhat^*_{cat} \hhat_q$, 
where 
$g_{cyl}$ is the standard metric on the cylinder $\R\times \Sph^1(1)$ 
and $\Xhat^*_{cat} \hhat_q$ is the pullback by $\Xhat_{cat}$ of the metric induced by $\hhat_q$. 

Note that by \ref{LPiK} $\tildecattp$ 
is the graph in the Euclidean space 
$( T_p \Sph^3(1) \, , \, \hhat_q )$ 
over 
$\Omegatilde_\tau $ 
of the function $\frac1{\rtilde(q)} \,\varphitilde_\tau$. 
Since we have uniform bounds for the coordinate functions of $\Xhat_{cat}$
we can combine the estimates in \ref{LExp} and \ref{LPiK} 
to conclude that on $B'_q$ the 
norms of the differences of the fundamental forms induced by 
$\rtilde^{-2}(q)\, \gtilde_\tp$ 
on the graph versus the ones 
induced by $\hhat_q$ on 
$\widetilde{\cat}_{\zp}$, 
are bounded by a constant depending only on $k$ and $\alpha$ times  
\begin{equation}
\label{Esm}
\frac{\, \tau+\,\tau^2\,\rtilde^2(q) \, |\log\tau| \,} {\rtilde(q)} 
\, + \, 
\tau^2  \, \Rtilde^2(q)
\, \le \, C \,
( \,  \frac\tau{\rtilde(q)}
\, + \, 
\tau^2  \, \rtilde^2(q) \, ) \, 
\le \, 
C \, \tau^{2\alpha}, 
\end{equation} 
where for the last inequality we used that $\rtilde\le 9 \tau^{\alpha-1}$ by definition, 
and we also used that linear terms dominate because of the smallness of the last term in \ref{Esm}. 
By scaling and applying \ref{D:newweightedHolder} we conclude the proof.
\end{proof}

\section{Linearized doubling and initial surfaces}
\label{Sgeneral}
\nopagebreak

We expect that the approach developed in this paper,
which consists of finding appropriate linearized doubling (LD) solutions first,
and using them to ``build'' the desired minimal surfaces afterward,
can be modified 
to apply to general situations with little or no symmetry 
(see \ref{remarkone}).
Under this approach the difficulty is shifted 
to finding and understanding the appropriate 
LD solutions.

\subsection*{LD and MLD solutions}  
$\phantom{ab}$
\nopagebreak

We proceed now to describe the LD solutions 
for doubling constructions of $\Spheq$.
Note that our definitions although stated for $\Spheq$ can 
easily be modified to apply to any minimal surface. 
Note also that we can think of an LD solution $\varphi$ as a multi-Green's function,  
since clearly in the distributional sense $\Lcalp \varphi$ is a linear combination 
of delta functions: 

\begin{definition}[LD solutions] 
\label{DLDnoK}
Given a finite set $L\subset \Spheq$ and a function $\tau:L\to\R$,
we define a \emph{linearized doubling (LD) solution of configuration $(L,\tau)$} 
to be a function $\varphi\in C^\infty(\Spheq\setminus L)$ which satisfies the following conditions 
where $\tau_p$ denotes the value of $\tau$ at $p$: 
\newline
(i).
$\Lcalp\varphi=0$ on $\Spheq\setminus L$.
\newline
(ii).
$\forall p\in L$ 
there is 
a smooth extension across $\{p\}$
$$
\varphihat_p\in C^\infty(\{p\}\cup(\Spheq\setminus (L\cup\{-p\})\,)\,)
\quad\text{ such that }\quad
\varphihat_p=\varphi-\tau_p G_p
\quad\text{on}\quad
\Spheq\setminus (L\cup\{-p\}) . 
$$
\end{definition}

The main idea of our current approach is to construct initial surfaces 
by gluing catenoidal bridges centered at the points of $L$ to (appropriately modified) 
graphs of the LD solutions.  
This step requires a satisfactory matching of each LD solution 
to the catenoidal bridge at the annulus where the gluing occurs. 
The matching can be controlled by the first terms of the Taylor expansion 
of each $\varphihat_p$ at $p$. 
It turns out however that well matched LD solutions are not in sufficient supply for our purposes.  
For this reason we have to employ also LD solutions which are not well matched.  
Such solutions need to be modified so that they satisfy the matching conditions 
at the expense of not satisfying the exact linearized equation anymore. 
The solutions in this new class 
will only satisfy the linearized equation modulo 
a space $\skernel[L]$ which will be defined later in \ref{Dskernel}.   
$\skernel[L]$ depends smoothly on $L$ 
and plays also the role of the (extended) substitute kernel in the linear theory (see \ref{Plinear}).   

$\skernel[L]$ will be defined later in \ref{Dskernel} and will depend smoothly on $L$. 

\begin{definition}[LD solutions modulo $\skernel{[L]}$]  
\label{DLDmodK}
Given $L$ and $\tau$ as in \ref{DLDnoK},  
and also a finite dimensional space $\skernel[L]\subset C^\infty(\Spheq)$, 
we define a \emph{linearized doubling (LD) solution modulo $\skernel[L]$ of configuration $(L,\tau,\xiw)$} 
to be a function $\varphi\in C^\infty(\Spheq\setminus L)$ which satisfies the same conditions as in \ref{DLDnoK}, 
except that condition (i) is replaced by the following:  
\newline 
(i${}'$).
$\Lcalp\varphi=\xiw\in \skernel[L] \subset C^\infty(\Spheq)$ on $\Spheq\setminus L$.
\end{definition}

Note that LD solutions in the sense of 
\ref{DLDnoK}
are LD solutions in the sense of 
\ref{DLDmodK} 
with $w=0$. 
We describe now the matching conditions. 

\begin{definition}[Mismatch of LD solutions] 
\label{Dval}
Given $\varphi$ as in 
\ref{DLDnoK}
or 
\ref{DLDmodK} 
we define 
$\val[L]:=
\bigoplus_{p\in L} \val[p]$,
where $\val[p]:=\R\oplus T^*_p\Spheq$, and 
the mismatch of $\varphi$ by 
$$ 
\Bcal_L \varphi \, := 
\oplus_{p\in L} 
\left(
\, \varphihat_p(p)+\tau_p\log(\tau_p/2)\, , \, 
d_p\varphihat_p\, \right) 
\in \val[L].
$$ 
\end{definition}

Among all the LD solutions modulo $\skernel[L]$ we will be mainly interested in the ones 
which are well matched: 

\begin{definition}[MLD solutions] 
\label{DMLD}
We define a \emph{matched linearized doubling (MLD) solution modulo $\skernel[L]$ of configuration $(L,\tau,\xiw)$} 
to be some $\varphi$ as in \ref{DLDmodK}  
which moreover satisfies the conditions  
$\Bcal_L \varphi=0$ and $\tau_p>0$ $\forall p\in L$.   
\end{definition}

\begin{remark}
\label{remarkLD}
Note that given $\varphi$ and $L$ as in \ref{DLDmodK}, 
$\tau$, $\xiw$, each $\varphihat_p$, 
and the second components of $\Bcal_L \varphi$,   
are uniquely determined and depend
linearly on $\varphi$.
The first components of 
$\Bcal_L \varphi$ are not linear in $\varphi$ however and this makes the construction 
harder. 
\hfill $\square$
\end{remark}

\subsection*{The definition of $\skernel[L]$}
$\phantom{ab}$
\nopagebreak

In order to describe the support of the functions in $\skernel[L]$ we have first the following. 

\begin{convention}[The constants $\delta_p$] 
\label{con:L} 
Given $L$ as in \ref{DLDnoK} 
we assume that for each $p\in L$ we have chosen a constant $\delta_p>0$, 
where each $\delta_p$ is small enough so that any two $D_p(9\delta_p)$'s are disjoint for two different points $p\in L$.
\hfill $\square$
\end{convention}

\begin{definition}[The obstruction space ${\skernel[L]}$] 
\label{Dskernel}
Given $L$ and $\delta_p$'s as in \ref{con:L} 
we define 
$\skernel[L]\subset C^\infty(\Spheq)$ 
by 
$\skernel[L]:= \bigoplus_{p\in L} \skernel[p]$,
where $\skernel[p]$ is spanned by the following.
\newline
(i).
$
\Lcalp \Psibold\left[2\delta_p,\delta_p;\dbold_p\right]
(G_p,\log\delta_p \cos\circ\dbold_p\,)=
-
\Lcalp \Psibold\left[2\delta_p,\delta_p;\dbold_p \right]
(\,\log\delta_p \, \cos\circ\dbold_p\,,G_p).
$
\newline
(ii).
$
\Lcalp \Psibold\left[2\delta_p,\delta_p;\dbold_p  \right]
(0\,,u_p)
,
$
where $u_p$ is any first harmonic 
of $\Spheq$ vanishing at $p$.
\end{definition}

Note that the functions 
in $\skernel[L]$ are supported on $\bigsqcup_{p\in L} (D_p(2\delta_p)\setminus D_p(\delta_p) )$.
Clearly $\forall p\in L$ 
we have $\dim \skernel[p]=3$ 
and hence 
$\dim \skernel[L]=3|L|$
where $|L|$ is the number of points in $L$.

\subsection*{Symmetric LD solutions}
$\phantom{ab}$
\nopagebreak

Because of the symmetries imposed on our constructions we concentrate now on
LD solutions which are invariant under the action of $\grouptwo$ 
(recall \ref{Dgroup}).
In such a case we can write
\begin{equation}
\label{ELpar}
L=\Lmer\cap\Lpar,
\end{equation}
where
$\Lpar$ is the union of a finite
number of parallel circles and perhaps $\{p_N,p_S\}$,
symmetrically arranged around the equator so that 
$\grouptwo\Lpar=\Lpar$.
We assume that $\delta_p$'s have been chosen as in \ref{con:L} and so that they are $\grouptwo$-invariant.
We also define (recall \ref{DLDmodK} and \ref{Nsym}) 
\begin{equation}
\label{Eskernelsym}
\skernel_\sym[L]:=\skernel[L]\cap C^\infty_\sym(\Spheq),
\qquad
\skernelv_\sym[L]:=\{v\in C^\infty_\sym(\Spheq):\Lcalp v\in\skernel_\sym[L] \}.
\end{equation}
Note that because of the symmetries $\Lcalp$ has no kernel and therefore 
$\Lcalp$ restricted to $\skernelv_\sym[L]$ provides an isomorphism onto $\skernel_\sym[L]$.
The dimension of 
$\skernel_\sym[L]$ 
and $\skernelv_\sym[L]$ 
is clearly 
$k_{eq}+k_{poles}+2k_{other}$ where $k_{eq}=1$ if the equatorial circle is included in $\Lpar$ and $0$ otherwise, 
$k_{poles}=1$ if the poles are included and $0$ otherwise, 
and $\mmer= k_{eq}+k_{poles}+2k_{other}$.  
Note now that the symmetries imposed ensure that the configuration of an LD solution uniquely 
determines the LD solution as in the next lemma:

\begin{lemma}[Symmetric LD solutions] 
\label{LLD}
Given a finite $\grouptwo$-invariant set $L\subset\Spheq$,
a $\grouptwo$-invariant function $\tau:L\to\R$,
and $\xiw\in\skernel_\sym[L]$,
there is a unique 
$\grouptwo$-invariant 
LD solution modulo $\skernel[L]$ 
$\varphi=\varphi[L,\tau,\xiw]$ of configuration $(L,\tau,\xiw)$  
(recall \ref{DLDmodK}). 
Moreover the following hold.
\newline
(i).
$\varphi$ and each $\varphihat_p$ 
depend linearly on $(\tau,\xiw)$.
\newline
(ii).
$\varphi_\ave\in C^0(\Spheq\setminus(L\cap\{p_N,p_S\})\,)$
(recall \ref{Dave}) 
and 
$\varphi_\ave$ is smooth on $\Spheq\setminus \Lpar$ where
it satisfies the ODE $\Lcalp \varphi_\ave=\xiw_\ave$.

If $w=0$ then we also write $\varphi=\varphi[L,\tau]$ and $\varphi$ is 
the unique $\grouptwo$-invariant LD solution 
of configuration $(L,\tau)$ as in \ref{DLDnoK}. 
\end{lemma}

\begin{proof}
We define 
$\varphi_1\in C^\infty_\sym(\Spheq\setminus L)$
by requesting that it is supported on $\bigsqcup_{p\in L} (D_p(2\delta_p))$
and 
$
\varphi_1=
\Psibold\left[\delta_p,2\delta_p;\dbold_p \right]
(G_p,0)
$
on
$
D_p(2\delta_p)
$
for each $p\in L$.
Note that $\Lcalp \varphi_1\in C^\infty_\sym(\Spheq)$
(by assigning $0$ values on $L$) and it is supported on 
$\bigsqcup_{p\in L} (D_p(2\delta_p)\setminus D_p(\delta_p) )$.
Because the symmetries do not allow the first harmonics of the Laplacian on 
$\Spheq$,
there is $\varphi_2\in C^\infty_\sym(\Spheq)$ such that
$\Lcalp \varphi_2=-\Lcalp \varphi_1+\xiw$.
We can define then $\varphi:=\varphi_1+\varphi_2$.
Uniqueness and (i) follow then immediately.
To prove (ii) we need to check that $\varphi$ is integrable on each circle contained
in $\Lpar$ and that $\varphi_\ave$ is continuous there also.
But these follow easily by the logarithmic behavior of $G_p$ (recall \ref{LGp}). 
Since the case $w=0$ is clearly a special case of the general case the proof is complete.
\end{proof}

Next we will need the following.
\begin{definition}[The map ${\Ecal_L}$] 
\label{DEcal}
We define $\val_\sym[L]$ to be the subspace of $\val[L]$ (recall \ref{Dval})  
consisting of those elements which are invariant under 
the obvious action of $\grouptwo$.  
We define then a linear map $\Ecal_L : \skernelv_\sym[L] \to \val_\sym[L]$ by 
$\Ecal_L (v):= (v(p),d_pv)_{p\in L}\in\val_\sym[L]$ 
for $v\in \skernelv_\sym[L]$ (recall \ref{Eskernelsym}). 
\end{definition}

The following assumption is crucial for the construction and will be checked later. 
Note that besides being used in the linear theory later it also allows us to convert 
any LD solution $\varphi$ in the sense of \ref{DLDnoK} to an MLD in the sense of \ref{DMLD} 
by subtracting from it $\Ecal_L^{-1}\,\Bcal_L\varphi$: 

\begin{assumption}
\label{AEcal}
We assume that the map $\Ecal_L : \skernelv_\sym[L] \to \val_\sym[L]$
is a linear isomorphism.  
\end{assumption}

\begin{definition}
\label{DEcalnorm}
We denote by $\|\Ecalinv \|$ the operator norm of 
$\Ecalinv:\val_\sym[L] \to \skernelv_\sym[L]$ 
with respect to 
the $C^{2,\beta}(\Spheq,g)$ norm on the target 
and the maximum norm on the domain subject to the standard metric $g$ of $\Spheq$.
\end{definition}

\subsection*{Initial surfaces from $\grouptwo$-symmetric MLD solutions}
$\phantom{ab}$
\nopagebreak

In this subsection we construct the initial surfaces 
by gluing catenoidal bridges to appropriately modified graphs of MLD solutions.
More precisely we start by 
assuming given a $\grouptwo$-symmetric MLD solution,  
$\varphi=\varphi[L,\tau,\xiw]$ in the notation of \ref{LLD}. 
The first step in the construction is to modify the MLD solution so that its graph 
on an appropriate domain (corresponding to the complement of the catenoidal
bridges) is minimal except on the support of the elements of $\skernel_\sym[L]$. 
We then attach the catenoidal bridges and this way we obtain
an initial surface where the unwelcome mean curvature is supported on
small annuli where the gluing occurs.
We choose now the scale of the gluing annuli: 
\begin{definition}
\label{Edeltapp}
For each $p\in L$ we define
$\delta_p':=\tau_p^\gammagl$ 
where 
$\alpha$ is as in \ref{con:alpha}.
We will also use the notation
$\delta_{min}:=\min_{p\in L}\delta_p$, 
$\tau_{min}:=\min_{p\in L}\tau_p$, 
$\tau_{max}:=\max_{p\in L}\tau_p$, 
and 
$\delta_{min}':=\min_{p\in L}\delta_p'=\tau_{min}^\gammagl$. 
\end{definition}
To simplify the presentation and the construction 
it is convenient to assume the following 
which we will confirm later for the actual constructions we carry out (see \ref{L:h}.vii and \ref{L:heq}.v). 
\begin{convention}
\label{con:one}
We assume from now on that the following hold.  
\newline
(i). 
\ref{con:L} holds 
and $\tau_{max}$ is small enough in absolute terms as needed.
\newline
(ii).
$\forall p\in L$ we have $9 \delta_p' < \tau_p^{\alpha/9}< \delta_p \, $. 
\newline
(iii).
$\tau_{max}\le \tau_{min}^{1-\gammagl/9}$.  
\newline
(iv).
$\forall p\in L$ we have 
$(\delta_p)^{-2} \| \, \varphihat_p 
: C^{2,\beta}(\, \partial D_p(\delta_p)    ,\, (\delta_p)^{-2} g\,)\,\| 
\le
\tau_p^{1-\gammagl/9}$.
\newline
(v).
$\| \varphi:C^{3,\beta}_\sym ( \, \Spheq\setminus\disjun_{q\in L}D_q(\delta_q')    \, , \, 
g \, ) \, \|
\le
\tau_{min}^{8/9} \, $.
\newline
(vi).
On $\Spheq\setminus\disjun_{q\in L}D_q(\delta_q') $ we have
$\tau_{max}^{1+\alpha/5} \le \varphi$.  
\hfill $\square$
\end{convention}

\begin{remark}
\label{remark:emb}
Note that condition \ref{con:one}.vi is only needed to ensure embeddedness. 
For constructions of immersed surfaces which may not be embedded we could 
drop \ref{con:one}.vi.  
For such constructions we could also allow negative $\tau_p$'s 
by replacing $\tau_p>0$ in \ref{DMLD} with ``$\tau_p\ne0$'' 
and the first term on the right in \ref{Dval} with 
``$\varphihat_p(p)+\tau_p\log|\tau_p/2|=0$''.
Note that in such a case if \ref{con:one}.vi holds the positivity of $\tau_p$ is implied anyway. 
\hfill $\square$
\end{remark}

In order now to modify $\varphi$ 
which by definition satisfies the linearized condition \ref{DLDmodK}.i${}'$, 
to another function $\varphinl$ 
which satisfies the nonlinear condition \ref{Lnl}.i, 
we first define a cutoff function 
$\psi'' \in C^\infty_\sym(\Spheq)$ 
by $\psi''=1$ on
$\Spheq\setminus\disjun_{p\in L} D_p(2\delta_p')   $,
and on $D_p(2\delta_p')$ (for each $p\in L$)
$$
\psi''= 
\Psibold\left[\delta_p',2\delta_p';\dbold_p \right]
(0,1).
$$
We then define inductively sequences  
$\{u_n\}_{n=1}^\infty
\subset
C^{3,\beta}_\sym(\Spheq)$
and
$\{\phi_n\}_{n=-1}^\infty
\subset
C^{3,\beta}_\sym(\Spheq)$
by $\phi_{-1}=0$, $\phi_0:=\varphi$,
and for $n>0$
\begin{equation}
\label{Eun}
\phi_n = \phi_{n-1} + u_n, 
\qquad\quad
\Lcalp \, u_n \,
= \, 
\psi''\,( Q_{\phi_{n-2}} - Q_{\phi_{n-1}} ),
\end{equation}
where we define $Q_{\phi_k}$ to vanish on 
$\disjun_{p\in L} D_p( \delta_p')   $
and to satisfy 
$H_{ \phi_k }\, =\, \Lcalp  \phi_k      + Q_{ \phi_k     }$ 
on
$\Spheq\setminus\disjun_{p\in L} D_p( \delta_p')  \,$  , 
where $H_{ \phi_k    }$ is the mean curvature of the graph
of $ \phi_k    $ in $\Sph^3$ pushed forward to $\Spheq$ by 
the projection $\PiSph$ (recall \ref{EPiSph}).

\begin{lemma}
\label{Lnl}
Given a $\grouptwo$-symmetric MLD solution $\varphi=\varphi[L,\tau,\xiw]$ 
which is as in \ref{LLD} and \ref{DMLD} 
and where \ref{con:one} is satisfied, 
we can define
$\varphinl=\varphinl[L,\tau,\xiw]    \in C^\infty_\sym(\Spheq\setminus L)$
as the limit of the sequence 
$\phi_n$ defined above. 
Moreover the following hold.
\newline
(i).
$H_{\varphinl} = \Lcalp \varphi =\xiw$ on 
$\Spheq\setminus\bigsqcup_{p\in L}D_p(2\delta_p')$,
where 
$H_{\varphinl}$ 
is the mean curvature of the graph
of $\varphinl$ in $\Sph^3$ pushed forward to $\Spheq\setminus L$ by 
the projection $\PiSph$ (recall \ref{EPiSph}).
\newline
(ii).
$\varphinl-\varphi$ can be extended to a smooth function
on $\Spheq$ which satisfies
$$
\|\varphinl-\varphi:C^{3,\beta}_\sym(\Spheq,g) \|
\le
\, C \,
(\delta_{min}')^{-2}
\,
\| \varphi:C^{3,\beta}_\sym(\Spheq\setminus\textstyle\bigsqcup_{p\in L}D_p(\delta_p') ,
g)\|^2
\le
\tau_{min}^{3/2} .
$$
\end{lemma}

\begin{proof}
By standard linear theory, \ref{Eun},
and the triviality of the kernel of $\Lcalp$ on $\Spheq$ modulo the symmetries, 
we conclude that for $n\ge1$ we have
$$
\|u_n:C^{3,\beta}(\Spheq,g) \|
\le
\, C \,
\|\psi'':C^{1,\beta}(\Spheq,g)\|
\,\,
\|Q_{\phi_{n-1} } - Q_{\phi_{n-2} }:
C^{1,\beta }(\Omega,g)\|,
$$
where $\Omega:= \Spheq\setminus\bigsqcup_{p\in L}D_p(\delta_p')\supset \text{supp}\, \psi'' $. 
Since the quadratic (and higher) terms $Q_{\phi_k}$ can be expressed as an algebraic
expression involving geometric invariants of $\Spheq$, $\phi_k$, and the derivatives
of $\phi_k$, we have
$$
\|Q_{\phi_{n-1} } - Q_{\phi_{n-2} }:
C^{1,\beta }(\Omega ,g)\|
\le
\left\{
\begin{aligned}
&
\, C \,
\|\varphi:
C^{3,\beta }(\Omega, g)\|^2
\,\,\,
(n=1),
\\
&
\, C \,
\|\phi_{n-1}-\phi_{n-2}:
C^{3,\beta }(\Omega, g)\|
\,\,
\|\phi_{n-2}:
C^{3,\beta }(\Omega, g)\| 
\,\,\,
(n\ge2).
\end{aligned}
\right.
$$
Combining the last two estimates and substituting $u_{n-1}$ for 
$ \phi_{n-1}-\phi_{n-2} $ 
we conclude that 
$$
\|u_n:C^{3,\beta}(\Spheq,g) \|
\le
\left\{
\begin{aligned}
&
\, C \,
(\delta_{min}')^{-2}
\,
\|\varphi:
C^{3,\beta }(\Omega, g)\|^2
\quad
(n=1),
\\
&
\, C \,
(\delta_{min}')^{-2}
\,
\|u_{n-1}:
C^{3,\beta }(\Omega, g)\|
\,\,
\|\phi_{n-2}:
C^{3,\beta }(\Omega, g)\|
\quad
(n\ge2).
\end{aligned}
\right.
$$
Since $2\frac89-2\alpha>\frac32$ by \ref{con:alpha},  
we conclude inductively 
using \ref{con:one}.v that for $n\ge1$
$$
\|u_n:C^{3,\beta}(\Spheq,g) \|
\le
\,
2^{-n}
\, C \,
(\delta_{min}')^{-2}
\,
\| \varphi:C^{3,\beta}(\Omega,
g)\|^2
\le
2^{-n} \, 
\tau_{min}^{3/2} .
$$
Taking limits and sums and using standard regularity theory for the smoothness
we conclude the proof.
\end{proof}

\begin{definition}
\label{Dvarphiinit}
Given $\varphi=\varphi[L,\tau,\xiw] $ as above
we define 
a function
$$
\varphi_{init}
=
\varphi_{init}[L,\tau,\xiw] 
:\Spheq\setminus\textstyle\bigsqcup_{p\in L} D_p(\tau_p)\to[0,\infty)
$$
as follows:
\newline
(i). On $\Spheq\setminus\bigsqcup_{p\in L} D_p(3\delta'_p)$ we have
$\varphi_{init}:=\varphinl[L,\tau,\xiw]    $.
\newline
(ii).
For each $p\in L$ we have on $D_p(3\delta'_p) \setminus D_p(\tau_p)$
(recall \ref{EPsibold})
$$
\varphi_{init}:=
\Psibold\left[2\delta'_p,3\delta'_p;\dbold_p \right]
( \, \Gtpp\, , \,\varphinl[L,\tau,\xiw] \, ).
$$
\end{definition}

\begin{definition}
\label{Dinit}
Given an LD solution $\varphi$ as 
above we define the initial 
smooth surface $M[L,\tau,\xiw]    $
to be the union over
$\Spheq\setminus\bigsqcup_{p\in L} D_p(\tau_p)$ 
of the graphs of $ \pm \varphi_{init}[L,\tau,\xiw]    $.  
\end{definition}

\begin{remark}
\label{remarkone}
The approach developed so far is quite general and can be easily modified
to apply to doublings of minimal surfaces where the following hold:
\newline
(i).  
A reflection exists exchanging the two sides of the given surface.
\newline
(ii).
The linearized operator has no kernel on the given surface.

If those conditions are not satisfied the approach still
applies with further modifications we will describe elsewhere.
\hfill $\square$
\end{remark}

\subsection*{The regions of the initial surfaces}
$\phantom{ab}$
\nopagebreak

\begin{lemma}[The gluing region]
\label{LH} 
For $M=M[L,\tau,\xiw]    $ defined as in \ref{Dinit} 
and $\forall p\in L$ the following hold. 
\newline
(i). 
$
\|\,
\varphi_{init} - \Gtpp  \, 
: \, 
C^{3,\beta}(\, D_p(4\delta'_p)\,\setminus D_p( \delta'_p)\, , \, (\delta'_p)^{-2}\, g\,)\,\| 
\le
\, \tau_p^{1+\frac{15}8\gammagl} \,
.
$
\newline
(ii). 
$
\|\,
\varphi_{init} \, 
: \, 
C^{3,\beta}(\, D_p(4\delta'_p)\,\setminus D_p( \delta'_p)\, , \, (\delta'_p)^{-2}\, g\,)\,\| 
\le
\, C \, \tau_p \, |\log\tau_p| \, 
.
$
\newline
(iii). 
$
\|\,
(\delta'_p)^{2}\, H' \, 
: \,
C^{0,\beta}(\, D_p(3\delta'_p)\,\setminus D_p(2\delta'_p)\, , \, (\delta'_p)^{-2}\, g\,)\,\| 
\le
\, \tau_p^{1+\frac{15}8\gammagl} \,
,
$
where $H'$ denotes the pushforward of $H$ to $\Spheq$ by $\PiSph$ 
and $H$ the mean curvature of the initial surface $M$.
\end{lemma}

\begin{proof}
By the definitions we have for each $p\in L$
$$
\varphi_{init}=
\tau_p G_p - \tau_p\log\frac{\tau_p}2 \, \cos\circ\dbold_p + 
\Psibold\left[2\delta'_p,3\delta'_p;\dbold_p \right]
(\varphi_-,\varphi_+)
$$
on 
$\Omega_p:=D_p(4\delta'_p)\,\setminus D_p( \delta'_p)$, 
where 
\begin{equation*}
\begin{aligned}
\varphi_-:=& \,\,\,
\Gtpp  -\tau_p G_p +\tau_p\log(\tau_p/2)\,\cos \circ\dbold_p,
\\   
\varphi_+:=& \,\,\,
\varphihat_p + \tau_p\log(\tau_p/2)\,\cos \circ\dbold_p
+\varphinl-\varphi.
\end{aligned}
\end{equation*}
By scaling now the ambient metric to 
$\gtildep:= (\delta'_p)^{-2}\, g$
and expanding in linear and higher order terms we have
$$
(\delta'_p)^{2}\, H'
=
(\Delta_{\gtildep}+2 (\delta'_p)^{2} \,)
\varphi_{init}
+
\delta'_p \widetilde{Q}_{(\delta'_p)^{-1} \varphi_{init} }.
$$
Note that on $\Omega_p$ we have 
$$
\begin{aligned}
\varphi_{init} - \Gtpp  =& 
\Psibold\left[2\delta'_p,3\delta'_p;\dbold_p \right]( 0 , \varphi_+ - \varphi_- )\,,  
\\
\Lcalp{\varphi_{init} } =& 
\Lcalp \Psibold\left[2\delta'_p,3\delta'_p;\dbold_p \right](\varphi_-,\varphi_+)\,.  
\end{aligned}
$$
Using these (for the second and the third inequality below) 
and also \ref{LGp}.vii 
we clearly have 
\begin{equation*}
\begin{aligned}
\| {\varphi_{init} } \,\| 
\le 
&
\,C\, 
( \,
\tau_p|\log\tau_p|
+
\|\varphi_-\,\| 
+
\|\varphi_+\,\| 
),
\\
\| {\varphi_{init} } - \Gtpp  \,\| 
\le 
&
\,C\, 
( \,
\|\varphi_-\,\| 
+
\|\varphi_+\,\| 
),
\\
\|
(\Delta_{\gtildep}+2 (\delta'_p)^{2} \,)
\varphi_{init}
:C^{0,\beta}(\, \Omega_p ,\, (\delta'_p)^{-2} g\,)\,\| 
\le 
&
\,C\, 
( \,
\|\varphi_-\,\| 
+
\|\varphi_+\,\| 
),
\\
\| 
\delta'_p \widetilde{Q}_{(\delta'_p)^{-1} \varphi_{init} }
:C^{0,\beta}(\, \Omega_p ,\, (\delta'_p)^{-2} g\,)\,\| 
\le 
&
\,C\, (\delta'_p)^{-1} \,
\| {\varphi_{init} } \,\|^2, 
\end{aligned}
\end{equation*}
where in this proof when we do not specify the norm we mean the
$C^{3,\beta}(\, \Omega_p ,\, (\delta'_p)^{-2} g\,)$ norm.
We conclude that if $\|\varphi_\pm\|\le\delta_p'$ (to control the quadratic terms), 
then we have
$$
\|
(\delta'_p)^{2}\, H'
:C^{0,\beta}(\, \Omega_p ,\, (\delta'_p)^{-2} g\,)\,\| 
\le 
\,C\, 
( \,
(\delta_p')^{-1}\tau_p^2|\log\tau_p|^2
+
\|\varphi_-\,\| 
+
\|\varphi_+\,\| 
).
$$

By \ref{CGGest} we have
$$
\|\varphi_- \| 
\le
C\,\tau_p^{1+2\gammagl}\,|\log\tau_p|.
$$
By the definition of $\varphi_+$ we have 
$$
\|\varphi_+
\,\| 
\le \, 
\| \, \varphihat_p + \tau_p\log(\tau_p/2)\,\cos \circ\dbold_p 
\,\| 
\, + \, 
\|\varphinl-\varphi:C^{3,\beta}_\sym(\Spheq,g) \|
$$
By standard theory (with interior regularity for the gain of derivative) 
and separation of variables 
the matching condition in \ref{DMLD} implies that
$$
\| \, \varphihat_p + \tau_p\log(\tau_p/2)\,\cos \circ\dbold_p 
\,\| 
\le 
C\, 
(\delta'_p/\delta_p)^2\,
\| \, \varphihat_p 
: C^{2,\beta}(\, \partial D_p(\delta_p) ,\, (\delta_p)^{-2} g\,)\,\| .
$$
Using \ref{con:one}.iv, \ref{Lnl}.ii, and \ref{con:alpha}, we conclude that
$$
\|\varphi_+
\,\| 
\le C\, 
(\delta'_p)^2 \, 
\tau_p^{1-\gammagl/9}
+
\tau_{min}^{3/2}
\le 
\,
\tau_p^{1+\frac{17}9\gammagl}
\,.
$$
Combining the above we complete the proof.
\end{proof}

\begin{lemma}
\label{LMemb}
If \ref{con:one} holds then $M$ is embedded.
Moreover the following estimates hold.
\newline
(i).
On $\Spheq\setminus\disjun_{p\in L}D_p(\delta_p') $ we have
$\frac89\tau_{max}^{1+\alpha/5} \le \varphi_{init}$. 
\newline
(ii).
$\| \varphi_{init}:C^{3,\beta}_\sym ( \, \Spheq\setminus\disjun_{p\in L}D_p(\delta_p')    \, , \, 
g \, ) \, \|
\le
\frac98\tau_{min}^{8/9} \, $.
\newline
(iii).
$\forall p\in L$ we have 
$$
\|\, 
\varphi_{init}
-
\tau_p \log ( 2 \dbold_p / {\tau_p} )
\,
: C^{3,\beta}(\,
D_p( 4\tau_p^\alpha )\setminus D_p( 9\tau_p ) \, ,\,
\dbold_p \, , \, g \, , \,  
\tau_p^{\frac{15}8\gammagl} \,
+ \tau_p^2 \dbold_p^{-2} \,)
\, \| \, \le \,
C \, \tau_p \, . 
$$ 
\end{lemma}

\begin{proof}
We first prove the estimates (i-iii):
(i) 
on $\Spheq\setminus\disjun_{p\in L}D_p(3\delta_p') $ 
follows from \ref{con:one}.vi, \ref{Lnl}.ii, and \ref{Dvarphiinit}.i, 
and 
on $D_p(4\delta'_p)\,\setminus D_p( \delta'_p)\,$ for $p\in L$ 
from \ref{LH}.i, \ref{Cone}, and \ref{con:one}.iii. 
(ii) 
on $\Spheq\setminus\disjun_{p\in L}D_p(3\delta_p') $ 
follows from \ref{con:one}.v, \ref{Lnl}.ii, and \ref{Dvarphiinit}.i, 
and 
on $D_p(4\delta'_p)\,\setminus D_p( \delta'_p)\,$ for $p\in L$ 
from \ref{LH}.ii, \ref{Cone}, and \ref{con:one}.iii. 
\ref{Dvarphiinit}.ii and \ref{LH}.i allow us to replace 
$\varphi_{init}$ with $\Gtpp$ in (iii). 
(iii) follows then from \ref{Cone}.
Finally the embeddedness of $M$ follows from (i) 
and by comparing the rest of $M$ with standard catenoids using \ref{LPiK}.
\end{proof}

Our general methodology requires that we subdivide the initial surfaces into various regions
\cite{kapouleas:clifford,haskins:kapouleas:invent,kapouleas:finite,kapouleas:wente,kapouleas:cmp,kapouleas:jdg,kapouleas:annals}. 
Because of the modified approach we 
only need some of the regions.
Because of the linearized doubling approach we also need 
to define the projections of some regions by $\PiSph$ (recall \ref{EPiSph}):

\begin{definition}
\label{D:regions}
We define the following 
for $x\in[0,4]$.
\begin{subequations}
\label{E:regions}
\begin{align} 
\label{ESp}
\Sp_x &:= \Spheq\setminus \disjun_{p\in L} D_p( 2 \delta_p /(1+x) ),
\\
\label{EStildep}
\Stildep_x &:= \Spheq\setminus \disjun_{p\in L} D_p( b \tau_p (1+x) )
\\
S_x[p] &:= M \cap \PiSph^{-1}\left( \overline{D_p( b \tau_p (1+x) )}\,\right) \qquad \forall p\in L \, ,
\\ 
S_x[L] &:= \disjun_{p\in L} S_x[p], 
\\
\label{Shatp} 
\Shat_x[p] &:= M \cap \PiSph^{-1}\left( \overline{D_p(2\delta'_p /(1+x) )}\, \right)\subset \cat_{\tpp} \qquad \forall p\in L \, ,
\\ 
\Shat_x[L] &:= \disjun_{p\in L}\Shat_x[p], 
\end{align}
\end{subequations}
where $b$ is a large constant independent of the $\tau$ parameters
which is to be chosen appropriately later. 
When $x=0$ we may omit the subscript.
\end{definition}

We define now precise Euclidean catenoids approximating the 
appropriate\-ly sca\-led ca\-te\-no\-i\-dal regions of the initial surface $M$,
and also auxiliary notation for future reference.  

\begin{definition}
\label{DPiK}
We define a map (recall \ref{Dtildecat})
$\Pi_{\cat,p} := \tildePiKp \circ \Exp^{-1}_{\tpp}
: \Shat[p] \to \widetilde{\cat}_{\zp}$.  
We also define 
$\tildecat_L := \disjun_{p\in L} \tildecat_{\zp}$  
and 
$\Pi_\cat : 
\Shat[L]\to\
\tildecat_L$ 
by taking the restriction of $\Pi_\cat$ to each
$\Shat[p]$ to be $\Pi_{\cat,p}$. 
\end{definition}

Clearly by \ref{LPiK} $\Pi_{\cat,p}$ is a diffeomorphism from
$\Shat[p]$ to a domain of $\widetilde{\cat}_{\zp}$. 
$\Pi_\cat$ is also a diffeomorphism from 
$\Shat[L]$ to a domain of $\tildecat_L$.
To incorporate now the symmetries into the discussion observe that
we can clearly define uniquely an action of $\groupthree$ on 
$\disjun_{p\in L}T_p\Spheq \supset 
\widetilde{\cat}_L$,
so that $\Pi_\cat$ is equivariant under the actions of the $\groupthree$.
Because of the importance of the scaling we will need the following.

\begin{definition}
\label{DtauK}
We define $\tau:\tildecat_L\to \R$ by $\tau=\tau_p$ on $\tildecatztp$.
\end{definition}

We extend now the notation in \ref{Nsym} to apply to functions
on domains of $M$ or $\tildecat_L$ as follows.

\begin{notation}
\label{NsymM}
Suppose $X$ is a function space consisting of functions defined on 
a domain $\Omega\subset M$ 
or $\Omega\subset\tildecat_L$.
If $\Omega$ is invariant under the action of $\groupthree$ acting on $M$ or  $\tildecat_L$
(recall \ref{Dgroup}),
then  we use a subscript ``$\sym$'' to denote the subspace $X_\sym\subset X$
consisting of those functions in $X$ which are invariant under the action of $\groupthree$. 
\hfill $\square$
\end{notation}

\section{The linearized equation and the nonlinear terms on the initial surfaces}
\label{Slinear}
\nopagebreak

\subsection*{The definition of $\Rcal_{M,appr}$}
$\phantom{ab}$
\nopagebreak

In this section we state and prove proposition \ref{Plinear} and lemma \ref{Lquad}. 
In \ref{Plinear} we solve with estimates the linearized equation on an initial surface 
$M=M[L,\tau,\xiw] $ defined as in \ref{Dinit},
where $\varphi[L,\tau,\xiw] $ 
is a $\grouptwo$-symmetric MLD solution 
of configuration
$(L,\tau,\xiw)$
defined as in \ref{DMLD}. 
In \ref{Lquad} we estimate the nonlinear terms on the initial surface $M$. 
To streamline the presentation we have the following.

\begin{convention}
\label{conventionb}
From now on we assume that $b$ (recall \ref{D:regions}) 
is as large as needed in absolute terms.
We also fix some $\beta\in(0,1)$ and $\gamma\in (1,2)$ 
satisfying $1-\frac\gamma2>2\alpha$ and $(1-\alpha)\,(\gamma-1)>2\alpha$,
for example $\gamma=\frac32$.
We will
suppress the dependence of various constants on $\beta$ and $\gamma$.
\hfill $\square$
\end{convention}

We construct now 
a linear map (recall \ref{NsymM})
\begin{equation}
\label{ERcalMappr}
\Rcal_{M,appr} : 
C^{0,\beta}_\sym(M) 
\to 
C^{2,\beta}_\sym(M) 
\oplus
\skernel_\sym[L] 
\oplus
C^{0,\beta}_\sym(M), 
\end{equation}
where if $E\in C^{0,\beta}_\sym(M)$ 
and 
$\Rcal_{M,appr} E =(u_1,\zw_{E,1},E_1)$,
then $u_1$ is an approximate solution to the linearized equation modulo
the ``extended substitute kernel'',
that is the equation
\begin{equation}
\label{ELcal}
\Lcal u=E+\zwE\circ\PiSph \qquad \text{where} \qquad \zwE\in \skernel_\sym[L],
\qquad
\Lcal:= \Delta+|A|^2+2,
\end{equation}
$\zw_{E,1}$ is the $\skernel_\sym[L]$ term,
and $E_1$ is the approximation error defined by
\begin{equation}
\label{EEone}
E_1:= \Lcal u_1 - E - \zw_{E,1} \circ \PiSph .
\end{equation}
The approximate solution $u_1$ 
is constructed by combining semi-local approximate solutions.
Before we proceed with the construction 
we define some cut-off functions we will need.

\begin{definition}
\label{DpsiM}
We define 
$\psi'\in C^{\infty}_\sym(\Spheq)$ 
and
$\widetilde{\psi}, \psihat \in C^{\infty}_\sym(M)$ 
by requesting the following.
\newline
(i). $\widetilde{\psi}= (1-\psi')\circ\PiSph$ on $M$.
\newline
(ii). 
$\widetilde{\psi}$ is supported on
$S_1[L]$,
$\psihat$ is supported on
$\Shat[L]$,
and $\psi'$ on 
$\Stildep$ (recall \ref{D:regions}).
\newline
(iii).
$\psi'=1$ on $\Stildep_1$
and for each $p\in L$ we have 
\begin{equation*}
\begin{aligned}
\psi'=&\, 
\Psibold\left[b \tau_p  , 2 b \tau_p ; \dbold_p  \right]
(0,1)
\quad\text{on}\quad
D_p(\,2b\,\tau_p\,),
\\
\psihat=&\, 
\Psibold\left[2\delta'_p , \delta'_p ; \, \dbold_p \circ \PiSph \,   \right]
(0,1)
\quad\text{on}\quad
\Shat[p].
\end{aligned}
\end{equation*}
\end{definition}

Given now 
$E\in C^{0,\beta}_\sym(M)$, 
we define 
$E'\in C^{0,\beta}_\sym(\Spheq)$ 
by requiring that it is supported on $\Stildep$, 
and that on $M$ we have the decomposition
\begin{equation}
\label{Edecom}
E=\widetilde{\psi}\,E\,+E'\circ\PiSph.
\end{equation}

Because of \ref{AEcal} there are unique 
$u'\in C^{2,\beta}_\sym(\Spheq)$ 
and
$\zw_{E,1}\in \skernel_\sym[L]$ 
such that
\begin{equation}
\label{Eup}
\Lcalp u'= E'+\zw_{E,1} \text{ on } \Spheq
\quad
\text{and} 
\quad
\forall p\in L \quad
u'(p)=0,
\quad
d_pu'=0.
\end{equation}
We define now 
$\widetilde{E}\in C^{0,\beta}_\sym(\tildecat_L)$, 
supported on 
$\Pi_\cat ( S_1[L] ) $, 
by
\begin{equation}
\label{EEpp}
\widetilde{E} \circ \Pi_\cat = 
\, \widetilde{\psi}\,E+
\,\{\, [\psi',\Lcalp]\, u'  \, + \,(1-\psi')\, E' \,\}\,
\circ\PiSph \,  
\quad\text{on } S_1[L]. 
\end{equation}
We introduce a decomposition
\begin{equation}
\label{EEtildedecom}
\widetilde{E}
=
\widetilde{E}_{low}
+
\widetilde{E}_{high},
\end{equation}
where
$ \widetilde{E}_{low} \in C^{0,\beta}_{\sym,low}(\tildecat_L)$ 
and
$ \widetilde{E}_{high} \in C^{0,\beta}_{\sym,high} (\tildecat_L)$ 
are supported on $\Pi_\cat ( S_1[L] )$. 
Note that here we use subscripts ``$low$'' and ``$high$'' to
denote subspaces of functions which satisfy the condition that their
restrictions to a meridian of a $\widetilde{\cat}_{\zp}$  
belong or are orthogonal respectively 
to the the span of the constants and the first harmonics on the meridian.
Let 
$\Lcal_\tildecat $ denote the linearized operator on $\tildecat_L$,   
and let 
$ \widetilde{u}_{low} \in C^{2,\beta}_{\sym,low}(\tildecat_L)$ 
and
$ \widetilde{u}_{high} \in C^{2,\beta}_{\sym,high} (\tildecat_L)$ 
be solutions of (recall \ref{DtauK})
\begin{equation}
\label{Etildeueq}
\Lcal_\tildecat  \,  \widetilde{u}_{low}  = \tau^2\, \widetilde{E}_{low},
\qquad 
\Lcal_\tildecat  \,  \widetilde{u}_{high}  = \tau^2\, \widetilde{E}_{high},
\end{equation}
determined uniquely as follows.
By separating variables the first equation amounts to 
uncoupled ODE equations which are solved uniquely
by assuming vanishing initial data on the waist of the catenoids.
For the second equation we can as usual change the metric conformally
to $h=\frac12|A|^2g=\nu^*g_{\Spheq}$, 
and then we can solve uniquely because the inhomogeneous term is clearly
orthogonal to the kernel.
We conclude now the definition of 
$\Rcal_{M,appr}$:
  
\begin{definition}
\label{DRMappr}
We define $\Rcal_{M,appr}$
as in \ref{ERcalMappr}
by taking 
$\Rcal_{M,appr} E =(u_1,\zw_{E,1},E_1)$,
where $\zw_{E,1}$ was defined in \ref{Eup},
$E_1$ in \ref{EEone}, 
and
$u_1
:=
\psihat\, \widetilde{u} \circ
\Pi_\cat
\, + \,
( \, \psi' \, u' \, ) \circ \PiSph
\in
C^{2,\beta}_\sym(M)$,
where
$
\widetilde{u}  :=  \widetilde{u}_{low} + \widetilde{u}_{high}  
\in C^{2,\beta}_{\sym}(\tildecat_L)
$. 
\end{definition}

\subsection*{Norms and approximations}
$\phantom{ab}$
\nopagebreak

We introduce now some abbreviated notation 
for the norms we will be using.
\begin{definition}
\label{D:norm}
For $k\in\N$, $\betahat\in(0,1)$, 
$\gammahat\in\R   $,
and $\Omega$ a domain in $\Spheq$, 
$M$, or $\tildecat_L$ 
(recall \ref{DPiK}), 
we define
$\|u\|_{k,\betahat,\gammahat;\Omega}
:=
\|u:C^{k,\betahat}(\Omega ,\rho,g,\rho^\gammahat)\|$,
where 
$\rho:= \dbold_L $ 
and $g$ is the standard metric on $\Spheq$ 
when $\Omega\subset \Spheq$, 
$\rho:= \dbold_L \circ \PiSph$ 
and $g$ is 
the metric induced on $M$ by the standard metric on $\Sph^3(1)$  
when $\Omega\subset M$, 
and $\rho=\rtilde$ (recall \ref{DtildePiK})
and $g$ is 
the metric induced by the Euclidean metric $h_p$ on $T_p\Sph^3(1)$ as in \ref{DExp} 
when $\Omega\subset \tildecat_L$. 
\end{definition}
Note that these definitions are equivalent to more popular definitions 
but we find these definitions more intuitive. 
We compare now norms on some nearby surfaces.

\begin{lemma}
\label{L:norms}
(i). If $\tau_{max}$ is small enough in terms of given $\epsilon>0$, 
$\Omegatilde$ is a domain in $\Pi_\cat(\Shat[L])$, 
$\Omega:=\Pi_\cat^{-1}(\Omegatilde)\subset \Shat[L] \subset M$, 
$k=0,2$, $\gammahat\in\R$,
and $f\in C^{k,\beta}(\Omegatilde)$, 
then we have 
(recall \ref{Dsimc} and \ref{DtauK}):
$
\| \, f\circ\Pi_\cat \, \|_{k,\beta,\gammahat;\Omega}
\, \sim_{1+\epsilon}
\| \, \tau^{-\gammahat}\, f \, \|_{k,\beta,\gammahat;\Omegatilde} \,$.
\newline
(ii). If $b$ is large enough in terms of given $\epsilon>0$, 
$\tau_{max}$ is small enough in terms of $\epsilon$ and $b$,
$\Omega'$ is a domain in $\Stildep = \Spheq\setminus \disjun_{p\in L} D_p( b \tau_p )$ 
(recall \ref{EStildep}),  
$\Omega:=\PiSph^{-1}(\Omega')\cap M$, 
$k=0,2$, $\gammahat\in\R$,
and $f\in C^{k,\beta}(\Omega')$, 
then 
$
\| \, f\circ\PiSph \, \|_{k,\beta,\gammahat;\Omega}
\, \sim_{1+\epsilon} \,
\|f\|_{k,\beta,\gammahat;\Omega'} \, $.
\end{lemma}

\begin{proof}
Note that by assuming $\tau_{max}$ small enough we can ensure that 
$9 \, C(k,\alpha)\, \tau_{max}^{2\alpha}\le \epsilon$.  
(i) follows then from the definitions, \ref{CPiK}, and \ref{Esm}. 
To prove (ii) let $q\in \Stildep$ and consider the metric 
$\ghat_q:= \, (\,  \dbold_L(q)\, )^{-2} \, g $ on $\Sph^3(1)$, 
where $g$ is the standard metric on $\Sph^3(1)$. 
In this metric $M$ is the union of the graphs of $\pm  \varphi_{:q}$ 
where $\varphi_{:q}:= \, (\,  \dbold_L(q)\, )^{-1} \, \varphi_{init}$. 
Let $B'_q$ be the geodesic disc in $(\Spheq\, , \, \ghat_q )$ of center $q$ and radius $1/10$. 
Note that 
$$
\|\, 
\log(2r/\tau) 
\,
: C^{k}(\,
(9 \tau,9\,\tau^\alpha)\,,\,
r, dr^2 ,  \, \log (r /\tau) \,)
\, \| \, \le \,
C(k) \, . 
$$
By \ref{LMemb} we have then that 
\begin{equation}
\label{Efw}
\|\,\varphi_{:q} \, : \, 
C^{3,\beta}(B'_q, \ghat_q)\,\|
\, \le \, C \, f_{weight}(q) 
\, \le \, C \, b^{-1} \log b , 
\end{equation} 
where $f_{weight}(q)=  \frac{\log(\dbold_p(q)/\tau_p\,) }{\dbold_p(q)/\tau_p}  $ if 
$q\in D_p( 3 \delta'_p)$ for some $p\in L$ 
(where we used that if $b>10$ then 
$ \tau_p^{1+\frac{15}8\gammagl} \dbold_p^{-1}(q) \, + \tau_p^3 \dbold_p^{-3}  \,
\le \, \frac{\log(\dbold_p(q)/\tau_p\,) }{\dbold_p(q)/\tau_p}  $ )  
and 
$f_{weight}(q)=  2 \tau_{min}^{8/9}$ 
otherwise (where we used  \ref{Lnl}.ii and \ref{con:one}.v).  
By comparing the metrics and using the definitions we complete the proof. 
\end{proof}

We reformulate now the estimate for the mean curvature from \ref{LH} to
an estimate stated in terms of the global norm we just defined. 

\begin{lemma}
\label{LHM}
The function 
$
H-\xiw \circ \PiSph
$
on the initial surface $M=M[L,\tau,\xiw]    $ 
is supported on 
$\PiSph^{-1}\left(\bigsqcup_{p\in L}(\,D_p(3\delta'_p)\,\setminus D_p(2\delta'_p) \,)\right)$.
Moreover it satisfies the estimate 
$$
\|\, H - \xiw \circ \PiSph\, \|_{0,\beta,\gamma-2;M} 
\le
\, \tau_{max}^{1 + {\alpha}/3} \,
.
$$
\end{lemma}

\begin{proof}
The statement on the support follows from \ref{LGtaup}.ii, 
\ref{Lnl}.i, and the definitions.
Combining now \ref{LH}.iii, \ref{D:norm}, and \ref{L:norms}.ii we complete the proof. 
\end{proof}

\begin{lemma}
\label{L:appr}
(i). 
If $\tau_{max}$ is small enough 
and $f\in C^{2,\beta}( \, \Pi_\cat(\Shat[L]) \, )$, 
then we have 
$$ 
\| \,
\Lcal  
\, (\, f \,  \circ
\Pi_\cat \, )\, 
-
\tau^{-2}\,
\, (\, \Lcal_\tildecat  
f \, )  
\circ \Pi_\cat 
\, \|_{0,\beta,\gammahat-2; \, \Shat[L] } 
\, \le \,
C\, \tau_{max}^{2\alpha}\,
\|\, \tau^{-\gammahat}\,f 
\, \|_{2,\beta,\gammahat; \, \Pi_\cat( \Shat[L] ) } 
\, . 
$$ 
(ii). 
If $\tau_{max}$ is small enough 
and $f\in C^{2,\beta}(\Stildep\,)$,
then for $\epsilon_1\in[0,1/2]$
we have 
$$ 
\| \,
\Lcal
\,\{ \, f  \circ \PiSph \,\}\,
-
\{\, \Lcalp \, f \,\}\, \circ \PiSph 
\, \|_{0,\beta,\gammahat-2 ; \, \PiSph^{-1}(\Stildep\,)  } 
\, \le \,
C \,  
b^{\epsilon_1-1}\, \log b \,\, \tau_{max}^{\epsilon_1} \, 
\|\, f 
\, \|_{ 2 , \beta , \gammahat + \epsilon_1 ; \,\Stildep  } 
\, . 
$$ 
\end{lemma}

\begin{proof}
(i). 
In analogy with \ref{DPiK} we define the map 
$\Pitilde_\cat: \disjun_{p\in L} \tildecat_{\tau_p,p} \to \tildecat_L $ 
by requesting that its restriction to $\tildecat_{\tau_p,p}$ for $p\in L$ 
is the restriction of  
$\tildePiKp$ to $\tildecat_{\tau_p,p}$ (recall \ref{LPiK} and \ref{DtildePiK}). 
We also define $\Lcaltilde$ 
to be the linearized operator on $\disjun_{p\in L} \tildecat_{\tau_p,p}$ 
with respect to the ambient metric which $\forall p\in L$ on $B_\tpp\subset T_p\Sph^3(1)$ 
equals $\gtilde_\tpp$ defined as in \ref{DExp}. 
We have then 
$$
\tau^{2-\gammahat}\, \{ \,
\Lcal  
\, (\, f \,  \circ
\Pi_\cat \, )\, \}
\circ \Pi_\cat^{-1} \,
-
\, \tau^{-\gammahat}
\, (\, \Lcal_\tildecat  
f \, )  
\, = \,
\tau^{-\gammahat}\, \, [ \, \{ \,
\Lcaltilde  
\, (\, f \,  \circ
\Pitilde_\cat \, )\, \}
\circ \Pitilde_\cat^{-1} \,
-
\, \Lcal_\tildecat  
f \, \, ] \,.
$$
Using then \ref{L:norms}.i and that $\tau$ is locally constant proving (i) 
reduces to proving 
$$ 
\| \,
\{ \,
\Lcaltilde  
\, (\, f \,  \circ
\Pitilde_\cat \, )\, \}
\circ \Pitilde_\cat^{-1} \,
-
\, \Lcal_\tildecat  
f 
\, \|_{0,\beta,\gammahat-2; \, \Pi_\cat( \Shat[L] ) } 
\, \le \,
C\, \tau_{max}^{2\alpha}\,
\|\, f 
\, \|_{2,\beta,\gammahat; \, \Pi_\cat( \Shat[L] ) } 
\, . 
$$ 
We fix now a $p\in L$ and we apply \ref{CPiK} and 
the notation and the observations in its proof (with $\tau_p$ instead
of $\tau$) including \ref{Esm}: 
We have then that the $C^{0,\beta}$ norm on $B'_q$ with respect to the metric 
induced by $\hhat_q$ of the corresponding difference of linearized operators 
applied on $f$ is bounded by 
$$
C \, \tau_{max}^{2\alpha} \, 
\|\, f : C^{2,\beta}(B'_q, h_p)\,\|.  
$$
Using scaling and the definitions we conclude the proof of (i). 

(ii). In this case we apply the notation and observations in the proof of \ref{L:norms}.ii.  
By \ref{Efw} and 
by using scaling for the left hand side, 
we conclude that for $q\in \Stildep$, 
we have 
$$ 
(\,  \dbold_L(q)\, )^{2} \, 
\| \,
\Lcal
\,\{ \, f  \circ \PiSph \,\}\,
-
\{\, \Lcalp \, f \,\}\, \circ \PiSph 
\, : C^{0,\beta} ( \PiSph^{-1}( B'_q \,),   \, \ghat_p \, ) \, \| 
\, \le \,
C \,  
f_{weight}(q) \, 
\|\, f 
\, : C^{2,\beta} ( B'_q, \, \ghat_p \, ) \, \| 
\, . 
$$ 
By the definitions it is enough then to check that 
$\forall q\in \Stildep$ 
we have 
$$
f_{weight}(q) \, 
(\,  \dbold_L(q)\, )^{\epsilon_1} \, 
\le \, 
C \, b^{\epsilon_1-1}\, \log b \,\, \tau_{max}^{\epsilon_1}. 
$$
This follows from the definition of $f_{weight}$ 
(given in the proof of \ref{L:norms}) 
and the observation that $x^{\epsilon_1-1}\log x$ 
is decreasing in $x$ for $x\ge b$.  
This completes the proof.
\end{proof}

\subsection*{The main Proposition}
$\phantom{ab}$
\nopagebreak

\begin{prop}
\label{Plinear}
Recall that we assume that \ref{con:alpha}, \ref{con:one}, \ref{conventionb}, and \ref{AEcal} hold.
Suppose further that 
\begin{equation}
\label{EEcal}
\delta_{min}^{-4} \,\tau_{max}^\alpha \,
\|\Ecalinv\| \,
\le \,1.
\end{equation}

A linear map 
$
\Rcal_M: C^{0,\beta}_\sym(M) \to C^{2,\beta}_\sym(M) \times \skernel_\sym[L] 
$ 
can be defined then 
by 
$$
\Rcal_M E
:=
(u,\zwE) 
:=
\sum_{n=1}^\infty(u_n,\zw_{E,n})
\in C^{2,\beta}_\sym(M) \times \skernel_\sym[L]
$$
for 
$E\in C^{0,\beta}_\sym(M)$, 
where the sequence 
$\{(u_n,\zw_{E,n},E_n)\}_{n\in \N}$
is defined inductively for $n\in \N$ by 
$$
(u_n,\zw_{E,n},E_n) := - \Rcal_{M,appr} E_{n-1}, 
\qquad\quad
E_0:=-E.
$$
Moreover the following hold.
\newline
(i). 
$\Lcal u= E+ \zwE \circ \PiSph$.
\newline
(ii).
$
\|u\|_{2,\beta,\gamma;M}
\le  \, C(b)  \, 
\delta_{min}^{-2-\beta} \,
\|\Ecalinv\|\,
\|\, E\, \|_{0,\beta,\gamma-2;M}.
$
\newline
(iii).
$
\|\zwE  : C^{0,\beta}(\Spheq,g)\|
\le  \, C \, 
\delta_{min}^{\gamma-2-\beta} \,
\|\Ecalinv\|\,
\|E\|_{0,\beta,\gamma-2;M}.
$
\newline
(iv).
$
\Rcal_M
$ 
depends continuously on the parameters of $\varphi$.
\end{prop}

\begin{proof}
We subdivide the proof into five steps:

\textit{Step 1: Estimates on $u'$ and $\zw_{E,1}$:}
We start by decomposing $E'$ and $u'$  (defined as in \ref{Edecom} and \ref{Eup}) 
into various parts which will be estimated separately. 
We clearly have by the definitions and the equivalence of the norms 
as in \ref{L:norms} that 
$$
\|E'\|_{0,\beta,\gamma-2;\Spheq}
\le \, C \,
\|E\|_{0,\beta,\gamma-2;M}.
$$
We first solve uniquely for each $p\in L$ the equation $\Lcalp u'_p=E'$ 
on $D_p(2\delta_p)$ by requiring that 
$u'_p(p)=0$, $d_pu'_p=0$, and that the restriction of $u_p$
on $\partial D_p(2\delta_p)$ is a linear combination of constants and first harmonics.
Clearly then by standard theory and separation of variables we have
$$
\|u'_p\|_{2,\beta,\gamma;D_p(2\delta_p)}
\le \, C \,
\|E'\|_{0,\beta,\gamma-2;D_p(2\delta_p)}.
$$
We define now 
$u''\in C^{2,\beta}_\sym(\Spheq)$ 
supported on $\disjun_{p\in L} D_p( 2\delta_p )$
by requesting that for each $p\in L$ we have 
$$
u''=\,
\Psibold\left[2 \delta_p , \delta_p ; \dbold_p  \right]
(0,u_p')
\quad\text{on}\quad
D_p( 2\delta_p).
$$
We clearly have then 
$$
\|u''\|_{2,\beta,\gamma;\,\Spheq}
\le \, C \,
\|E\|_{0,\beta,\gamma-2;M}.
$$
$E'-\Lcalp u''$ 
vanishes on 
$\disjun_{p\in L} D_p( \delta_p )$
and therefore it is supported on 
$\Spheq\setminus \disjun_{p\in L} D_p( \delta_p ) = \Sp_1$
(recall \ref{ESp}).
Moreover it satisfies 
$$
\|E'-\Lcalp u''\|_{0,\beta,\gamma-2;\,\Spheq}
\le \, C \,
\|E\|_{0,\beta,\gamma-2;M}.
$$
Using the definition of the norms and 
the restricted support $\Sp_1$ 
we conclude that 
$$
\|E'-\Lcalp u'': C^{0,\beta}(\Spheq,g)\|
\le \, C \, 
\delta_{min}^{\gamma-2-\beta} \,
\|E'-\Lcalp u''\|_{0,\beta,\gamma-2;\,\Spheq}.
$$
The last two estimates and standard linear theory imply that 
the unique by symmetry solution 
$u'''\in C^{2,\beta}_\sym(\Spheq)$ to $\Lcalp u'''=E'-\Lcalp u''$ 
satisfies
$$
\|u''' : C^{2,\beta}(\Spheq,g)\|
\le \, C \, 
\delta_{min}^{\gamma-2-\beta} \,
\|E\|_{0,\beta,\gamma-2;M}. 
$$

By \ref{AEcal} there is a unique 
$v\in \skernelv_\sym[L]$ 
(recall \ref{Eskernelsym})
such that $u'''+v$ and $d(u'''+v)$ vanish at each $p\in L$. 
Moreover by the last estimate and 
\ref{DEcalnorm} 
$v$ satisfies the estimate 
\begin{equation*}
\|v : C^{2,\beta}(\Spheq,g)\|
+
\| \Lcalp v : C^{0,\beta}(\Spheq,g)\|
\le  \, C \, 
\delta_{min}^{\gamma-2-\beta} \,
\|\Ecalinv\| \,
\|E\|_{0,\beta,\gamma-2;M}. 
\end{equation*}
By the definition of $u'''$ 
we conclude that $\Lcalp(u''+u'''+v) = E'+\Lcalp v $. 
By the definitions of $u''$ and $v$ we clearly have 
that $u''+u'''+v$ satisfies also the vanishing conditions in \ref{Eup} 
and hence 
$$
u'=u''+u'''+v
\qquad \text{and} \qquad 
\zw_{E,1} = \Lcalp v .
$$

Note now that 
$\Lcalp u'''=E'-\Lcalp u''$ 
vanishes on 
$\disjun_{p\in L} D_p( \delta_p )$
and 
by \ref{Dskernel} and \ref{Eskernelsym} 
so does $\Lcalp v\in\skernel_\sym[L]$. 
We conclude that for each $p\in L$ 
we have $\Lcalp (u'''+v)=0$ on $D_p( \delta_p )$, 
and since we know already that $u'''+v$ and $d(u'''+v)$ vanish at $p$, 
we can use standard theory and separation of variables to estimate with decay 
$u'''+v$ on $D_p( \delta_p )$ in terms of the Dirichlet data on $\partial D_p( \delta_p )$. 
Combining with the earlier estimates for $u'''$ and $v$ 
we conclude that
$$
\|u'''+v \|_{2,\beta,\gamma';\Spheq}
\le  \, C \, 
\delta_{min}^{\gamma-\gamma'-2-\beta} \,
\|\Ecalinv\| \,
\| \, E \, \|_{0,\beta,\gamma-2;M}, 
$$
where $\gamma'=\frac{\gamma+2}2\in(\gamma,2)$.
We need the stronger decay for estimating $E_1$ later.
A similar estimate holds with $\gamma$ instead of $\gamma'$. 
Note that by \ref{DEcalnorm} $ \|\Ecalinv\| \ge 1 $. 
Combining with the earlier estimate for $u''$ we conclude that 
$$
\|u'\|_{2,\beta,\gamma;\Spheq}
\le  \, C \, 
\delta_{min}^{-2-\beta} \,
\|\Ecalinv\| \,
\| \, E \, \|_{0,\beta,\gamma-2;M}. 
$$

\textit{Step 2: Estimates on $\widetilde{u}$:}
By the definitions and \ref{L:norms} (with $\epsilon=1$) we have that
$$
\| \, \tau^{2-\gamma} \, \widetilde{E} \, \|_{0,\beta,\gamma-2;\tildecat_L}
\, \le \,
C\, 
(\, \|\, {E} \, \|_{0,\beta,\gamma-2;M}
\, + \,
\|\, u' \, \|_{2,\beta,\gamma;\, \Spheq}
\,).
$$
By considering the standard conformal parametrization of the catenoid on a cylinder 
it is easy to conclude that 
$$
\| \, \tau^{-\gamma} \, \widetilde{u}_{low} \, \|_{2,\beta,1;\tildecat_L}
\, \le \,
C(b) \, 
\| \, \tau^{2-\gamma} \, \widetilde{E} \, \|_{0,\beta,\gamma-2;\tildecat_L}.
$$
Similarly by standard linear theory and the obvious $C^0$ bound on $\widetilde{u}_{high}$ 
we conclude  
$$
\| \, \tau^{-\gamma} \, \widetilde{u}_{high} \, \|_{2,\beta,0;\tildecat_L}
\, \le \,
C(b) \, 
\| \, \tau^{2-\gamma} \, \widetilde{E} \, \|_{0,\beta,\gamma-2;\tildecat_L}.
$$
Combining the above we conclude that 
$$
\| \, \tau^{-\gamma} \, \widetilde{u} \, \|_{2,\beta,\gamma;\tildecat_L}
\, \le \,
\| \, \tau^{-\gamma} \, \widetilde{u} \, \|_{2,\beta,1;\tildecat_L}
\, \le \,
C(b) \, 
\delta_{min}^{-2-\beta} \, \|\Ecalinv\| \, \| \, E \, \|_{0,\beta,\gamma-2;M}. 
$$

\textit{Step 3: A decomposition of $E_1$:}
Using \ref{EEone} and \ref{DRMappr}, \ref{Edecom}, \ref{Etildeueq}, and \ref{EEtildedecom}, 
we obtain 
\begin{equation}
\label{EoneA}
E_1=
E_{1,I}
+
\psihat\, 
\Lcal  
\, ( \, \widetilde{u}  \circ
\Pi_\cat \, ) \, 
+
E_{1,III}
+
\{\, \Lcalp \, ( \, \psi' \, u' \, ) \,\}\, \circ \PiSph 
- 
E
-
\zw_{E,1}\circ\PiSph,
\end{equation}
where 
$E_{1,I} , E_{1,III}\in 
C^{0,\beta}_\sym(M)$ 
are supported on 
$\Shat[L]\setminus\Shat_1[L]$
and
$\Stildep$ respectively by \ref{DpsiM}.ii,
and where they are defined by 
\begin{equation}
\label{EEoneI}
\begin{aligned}
E_{1,I} \, :=& \,
[\Lcal,\psihat] 
\,( \, \widetilde{u} \circ
\Pi_\cat \,) \, ,
\\
E_{1,III} \, :=& \,
\Lcal
\,\{ \, ( \, \psi' \, u' \, ) \circ \PiSph \,\}\,
-
\{\, \Lcalp \, ( \, \psi' \, u' \, ) \,\}\, \circ \PiSph. 
\end{aligned}
\end{equation}
Using \ref{Edecom}, \ref{Etildeueq}, \ref{EEtildedecom}, and \ref{EEpp},  
we obtain 
\begin{multline}
\label{EoneB}
E_1=
E_{1,I}+E_{1,II}
+
\, \widetilde{\psi}\,E+
\,\{\, [\psi',\Lcalp]\, u'  \, + \,(1-\psi')\, E' \,\}\,
\circ\PiSph \,  
+
\\
+
E_{1,III}
+
\{\, [\Lcal',\psi']\,u' + \psi' \,  \Lcal' u' \, \} \circ \PiSph 
-
\widetilde{\psi}\,E\,
- 
E'\circ\PiSph
- 
\zw_{E,1}\circ\PiSph,
\end{multline}
where 
$E_{1,II} \in C^{0,\beta}_\sym(M)$ 
is supported 
by \ref{DpsiM}.ii
on 
$\Shat[L]$
where it satisfies 
\begin{equation}
\label{EEoneII}
E_{1,II} =
\psihat\, 
\left(
\Lcal  
\, (\, \widetilde{u} \,  \circ
\Pi_\cat \, )\, 
-
\tau^{-2}\,
\, (\, \Lcal_\tildecat  
\widetilde{u} \, )  
\circ \Pi_\cat 
\right) \, = \, 
\psihat\, 
\Lcal  
\, (\, \widetilde{u} \,  \circ
\Pi_\cat \, )\, 
-
\widetilde{E} \circ \Pi_\cat  \, 
\end{equation}
By \ref{DpsiM}.ii and \ref{Eup} we have 
$\psi' \Lcalp u'= \psi' E'+\zw_{E,1} $.
Using this  
and canceling terms  
we conclude that 
\begin{equation}
\label{EEoneF}
E_1=E_{1,I}+E_{1,II}+E_{1,III}.
\end{equation}

\textit{Step 4: Estimates on $u_1$ and $E_1$:}
Using the definitions, \ref{L:norms} with $\epsilon=1$, 
and the estimates for $u'$ and $\widetilde{u}$ above 
we conclude that 
$$
\| \, u_1 \, \|_{2,\beta,\gamma;M} 
\, \le \,
C(b) \, 
\delta_{min}^{-2-\beta} \, \|\Ecalinv\| \, 
\| \, E \, \|_{0,\beta,\gamma-2;M}. 
$$

By \ref{L:norms} we have 
$\|\, 
\, (\, \tau^{1-\gamma} \widetilde{u} \,) \circ \Pi_\cat \, \|_{2,\beta,1;\,
\Shat[L]\setminus\Shat_1[L]}
\sim_2
\| \, \tau^{-\gamma} \, \widetilde{u} \, \|_{2,\beta,1;\,\Pi_\cat\,(\,
\Shat[L]\setminus\Shat_1[L]\,)}
$.
Using the definitions \ref{D:norm} and \ref{D:regions} we conclude that 
$$
\|\, 
\, \widetilde{u} \circ \Pi_\cat \, \|_{2,\beta,\gamma;\,
\Shat[L]\setminus\Shat_1[L]}
\, \le \,
C\, 
\tau_{max}^{(1-\alpha)(\gamma-1)} 
\, \,
\| \, \tau^{-\gamma} \, \widetilde{u} \, \|_{2,\beta,1;\tildecat_L}, 
$$
and therefore we have by the definition of $E_{1,I}$ that 
$$
\| \, E_{1,I} \, \|_{0,\beta,\gamma-2;\,M}
\, \le \,
C\, 
\tau_{max}^{(1-\alpha)(\gamma-1)} 
\, \,
\| \, \tau^{-\gamma} \, \widetilde{u} \, \|_{2,\beta,1;\tildecat_L}. 
$$
Applying now \ref{L:appr}.i with $f=\widetilde{u}$ and $\gammahat=\gamma$ 
and using the definition of $\psihat$ we conclude that  
$$
\| \, E_{1,II} \, \|_{0,\beta,\gamma-2;\,M}
\, \le \,
C\, \tau_{max}^{2\alpha}\,
\, \,
\| \, \tau^{-\gamma} \, \widetilde{u} \, \|_{2,\beta,\gamma;\tildecat_L}. 
$$

We decompose now $E_{1,III}= E''_{1,III}+E'''_{1,III}$ 
where $E''_{1,III}$ and $E'''_{1,III}$ are defined 
the same way as $E_{1,III}$ but with $u'$ replaced by 
$u''$ and $u'''+v$ respectively. 
Applying \ref{L:appr}.ii with $\epsilon_1=0$, $f=u''$, and $\gammahat=\gamma$, 
we conclude that 
$$
\| \, E''_{1,III} \, \|_{0,\beta,\gamma-2;\,M}
\, \le \,
C\, b^{-1}\, \log b \,\, 
\|u''\|_{2,\beta,\gamma;\Spheq}.
$$
Applying \ref{L:appr}.ii with $\epsilon_1=\gamma'-\gamma$, $f=u''$, and $\gammahat=\gamma$, 
$$
\| \, E'''_{1,III} \, \|_{0,\beta,\gamma-2;\,M}
\, \le \,
C\, b^{\gamma'-\gamma-1}\, \log b \,\, \tau_{max}^{\gamma'-\gamma} \,
\|u'''+v \|_{2,\beta,\gamma';\Spheq}.
$$

Combining the above with the earlier estimates
and using \ref{EEcal}  and \ref{conventionb} we conclude 
that 
$$
\| \, E_{1} \, \|_{0,\beta,\gamma-2;\,M}
\, \le \,
(\, C(b)\, \tau_{\max}^{\alpha/2} 
\, + \,
C\, b^{-1}\, \log b \,\, 
\, + \,
C\, b^{-1/2}\, \log b \,\, \tau_{max}^{\gamma'-\gamma-\alpha} \,
\,)\,
\| \, E \, \|_{0,\beta,\gamma-2;\,M}.
$$

\textit{Step 5: The final iteration:}
By assuming $b$ large enough 
and $\tau_{max}$ small enough in terms of $b$ we conclude using $\gamma'-\gamma-\alpha>0$ and induction 
that
$$
\| \, E_{n} \, \|_{0,\beta,\gamma-2;\,M}
\, \le \,
2^{-n}
\, \| \, E \, \|_{0,\beta,\gamma-2;\,M}.
$$
The proof is then completed by using the earlier estimates.
\end{proof}

\subsection*{The nonlinear terms} 
$\phantom{ab}$
\nopagebreak

If $\phi\in C^1_\sym(M)$ is appropriately small,
we denote by $M_\phi$ the perturbation of $M$ by $\phi$,
defined as the image of $I_\phi:M\to\Sph^3$, 
where $I:M\to\Sph^3(1)$ is the inclusion map of $M$ and 
$I_\phi$ is defined by $I_\phi(x):=\exp_x(\,\phi(x)\,\nu(x)\,)$ 
where $\nu:M\to T\Sph^3(1)$ is the unit normal to $M$. 
Clearly then (recall \ref{NsymM}) 
$M_\phi$ is invariant under the action of $\groupthree$
on the sphere $\Sph^3(1)$.
Using now rescaling 
we prove a global estimate for the nonlinear terms
of the mean curvature of $M_\phi$ as follows
(see \cite[Lemma 5.1]{kapouleas:clifford} for a similar statement):

\begin{lemma}
\label{Lquad}
If $M$ is as in \ref{Plinear} and 
$\phi\in C^{2,\beta}_\sym(M)$ 
satisfies $\|\phi\|_{2,\beta,\gamma;M} \, \le \, \tau_{max}^{1+\alpha/4} $,  
then $M_\phi$ is well defined as above, 
is embedded, 
and if 
$H_\phi$ is the mean curvature of $M_\phi$ pulled back to $M$ by $I_\phi$ 
and $H$ is the mean curvature of $M$, then we have 
$$
\|\, H_\phi-\, H - \Lcal \phi \, \|_{0,\beta,\gamma-2;M}
\, \le \, C \, 
\|\, \phi\, \|_{2,\beta,\gamma;M}^2.
$$
\end{lemma}

\begin{proof}
Note that such a strong bound on $\phi$ is only needed for ensuring  
the embeddedness of $M_\phi$. 
Following the notation in the proof of \ref{L:norms} and by \ref{Efw} we have 
that for 
$q\in \Stildep$ 
the graph $B''_q$ of $\varphi_{:q}$ over $B'_q$ in $(\Spheq\, , \, \ghat_q )$ 
can be described by an immersion $X_{:q}: B'_q \to B''_q= X_{:q}(B'_q)$ 
such that there are coordinates on $B'_q$ and a neighborhood in $\Sph^3(1)$ of $B''_q$ 
which are uniformly bounded and the immersion in these coordinates has 
uniformly bounded $C^{3,\beta}$ norms, the standard Euclidean metric on the domain is
bounded by $C X_{:q}^* \ghat_q$, and the coefficients of $\ghat_q$ in the target coordinates 
have uniformly bounded $C^{3,\beta}$ norms. 
By the definition of the norm and since 
$\|\phi\|_{2,\beta,\gamma;M} \, \le \, \tau_{max}^{1+\alpha/4} $,  
we have that the restriction of $\phi$ on $B''_q$ satisfies 
$$
\|\, \dbold_L^{-1}(q) \, \phi \, : \, C^{2,\beta}(\, B''_q\, , \, \ghat_q \,) \, \|
\, \le \,  C \, \dbold_L^{\gamma-1}(q) 
\, \|\, \phi\, \|_{2,\beta,\gamma;M}. 
$$
Since the right hand side is small in absolute terms we conclude that $I_\phi$ is well defined 
on $B''_q$, its restriction to $B''_q$ is an embedding, and by using scaling for the left hand 
side that 
$$
\|\, \dbold_L(q) \, ( H_\phi-\, H - \Lcal \phi )  \, 
\, : \, C^{0,\beta}(\, B''_q \, , \, \ghat_q \,) \, \|
\, \le \,  C \, \dbold_L^{2\gamma-2}(q) 
\, \|\, \phi\, \|_{2,\beta,\gamma;M}^2 . 
$$
Since $2\gamma-3-(\gamma-2)=\gamma-1>0$ we conclude that 
$$
\dbold_L^{2-\gamma}(q) \, \|\,  H_\phi-\, H - \Lcal \phi   \, 
\, : \, C^{0,\beta}(\, B''_q \, , \, \ghat_q \,) \, \|
\, \le \,  C 
\, \|\, \phi\, \|_{2,\beta,\gamma;M}^2 . 
$$

Note now that $B''_q$ is very close to the geodesic disc of radius $1/10$ in $\ghat_q$ and with 
a center a point of $M$ which projects by $\PiSph$ to $q$. 
It remains to establish similar estimates for such discs $B''_x\subset M$ with centers at 
points $x\in S[L]$. 
Note that the components of $S[L]$ appropriately scaled are small perturbations of a fixed compact 
region of the standard catenoid by smooth dependence on each $\tau_p$. 
This allows us to repeat the arguments above in this case and combining with the earlier estimates 
we conclude by the definitions the estimate in the statement of the lemma. 

It remains to prove the global embeddedness of $M_\phi$. 
Given the local embeddedness we already know, global embeddedness could only fail if there was 
a nontrivial intersection between $M_\phi$ and $\Spheq$. 
Using the estimates in \ref{LMemb} we can exclude this possibility and the proof is complete.
\end{proof}

\section{LD and MLD solutions in the two-circle case}
\label{Stwocircle}
\nopagebreak

\subsection*{Basic definitions}
$\phantom{ab}$
\nopagebreak

We concentrate now to the case where $\Lpar$ consists of only two circles.
Because of the invariance under $\grouptwo$ there exists $\xx_1\in(0,\pi/2)$
such that
\begin{equation}
\label{Etwopar}
\Lpar=
\Lpar[\xx_1]
:=
(\Cir_{\sin\xx_1}\cup\Cir_{-\sin\xx_1})=
\Theta(\{\xx=\pm\xx_1,\zz=0\}).
\end{equation}
We define
\begin{equation}
\label{EtwoL}
L:=
L[\xx_1,m]=
\Lmer[m]\cap\Lpar[\xx_1]
=\grouptwo \, p_1
\quad 
\text{ where } 
\quad 
p_1:=\Theta(\xx_1,0,0).
\end{equation}
Clearly $L$
consists of $2m$ points,
$m$ of them at latitude $\xx_1$,
and the other $m$ at latitude $-\xx_1$.
We define (recall \ref{con:L})
\begin{equation}
\label{Edelta}
\delta_p:=\deltaL:=\frac1{9m} \cos\xx_\riza \quad (p\in L),
\end{equation}
where we assume from now on 
\begin{equation}
\label{Exxone}
\xx_1\in \, (\, \xxbal/2 \, ,\, ( \xx_\riza + \xxbal )/2 \, ),
\end{equation} 
where $\xx_\riza$ was defined in \ref{Lphieo} 
and
$\xxbal$ will be defined in \ref{L:hhat}.
\ref{Exxone} ensures that the condition in \ref{con:L} is satisfied.
Clearly $\skernel_\sym[L]$ is two-dimensional and spanned by
$W,W'\in\skernel_\sym[L]$, both of which are supported on 
$D_L(2\delta_1)\setminus D_L(\delta_1)$, 
and are defined by requesting that on
$D_{p_1}(2\delta_1)$
we have
\begin{equation}
\label{EWW}
\begin{aligned}
W=W[\xx_1,m]
:=&
\Lcalp\,
\Psibold\left[2\delta_1,\delta_1 ; \dbold_{p_1} \right]
(G_{p_1},\,\log\delta_{1} \, \cos\circ\dbold_{p_1}\,),
\\
W'=W'[\xx_1,m]
:=&
\Lcalp\,
\Psibold\left[2\delta_1,\delta_1 ; \dbold_{p_1} \right]
(0,u\,),
\end{aligned}
\end{equation}
where $u$ is the first harmonic on $\Spheq$ 
characterized by $u(p_1)=0$ and $d_{p_1} u = d_{p_1} \xx$.
Because of the symmetries we only consider constant $\tau:L\to \R$.
\begin{definition}
\label{Etwovarphi}
We define an LD solution $\Phi =\Phi[\xx_1,m]:=\varphi[L[\xx_1,m],\,1] \in C^\infty_\sym(\Spheq\setminus L)$ ,
and $V=V[\xx_1,m],V'=V'[\xx_1,m]\in \skernelv[L]$ by $\Lcalp V= W$ and $\Lcalp V'=W'$  
(recall \ref{LLD}, \ref{Eskernelsym}, and \ref{EWW}).  
\end{definition}
Clearly $\skernelv_\sym[L]$ is spanned by $V$ and $V'$.

\subsection*{The rotationally invariant part $\phi:=\Phi_\ave$}
$\phantom{ab}$
\nopagebreak

To help with the presentation we first introduce some notation.

\begin{notation}
\label{Npartial}
If a function $u$ is defined on a neighborhood of $\Lpar$
and has one-sided partial derivatives at $\xx=\xx_1$,
then we 
we use the notation 
(so that if $u$ is $C^1$ then $\partial_{1+\,} u + \partial_{1-\,} u = 0 $)  
$$
\partial_{1+\,} u :=\left. \frac{\partial u  }{\partial \xx}\right|_{\xx=\xx_1+},
\qquad\qquad
\partial_{1-\,} u :=-\left. \frac{\partial u  }{\partial \xx}\right|_{\xx=\xx_1-}.
$$
\hfill $\square$
\end{notation}

\begin{lemma}
\label{Lphiave}
For $\xx_1$ as in \ref{Exxone}
we have 
that $\phi:=\Phi_\ave[\xx_1,m]$ 
is given by (recall \ref{Dphieo} and \ref{Dave})
\begin{equation*}
\begin{gathered}
\phi=\left\{
\begin{gathered}
\frac{\phi_1}{\phie(\xx_1)}\phie
\quad\text{on}\quad
\{|\xx|\le\xx_1\}\subset\Spheq,
\\
\frac{\phi_1}{\phio(\xx_1)}\phio
\quad\text{on}\quad
\{\xx_1\le\xx\}\subset\Spheq,
\end{gathered}
\right.
\quad\text{ where }\quad
\phi_1=\frac{m}{\cos\xx_1\,(h_{1+}+h_{1-})},
\\
h_{1+}:=\frac1{\phio(\xx_1)}\frac{\partial\phio}{\partial\xx}(\xx_1)
=\frac1{\phi_1} \, \partial_{1+\,} \phi
\, > \, 0 \,
,
\\
h_{1-}:=-\frac1{\phie(\xx_1)}\frac{\partial\phie}{\partial\xx}(\xx_1)
=\frac1{\phi_1} \, \partial_{1-\,}\phi 
\, > \, 0 \,
.
\end{gathered}
\end{equation*}
Moreover we have $\phi\ge\phi_1>0$
on $\Spheq$.
\end{lemma}

\begin{proof}
To simplify the notation for this proof we define domains of $\Spheq$,
$\Omega_N:=\{\xx_1\le\xx\}$ and
$\Omega_{eq}:=\{|\xx|\le\xx_1\}$,
which are neighborhoods of the North pole and the equator respectively.
Because of the symmetries it is clear that $\phi=A_+ \phio$ on $\Omega_N$
and $\phi=A_- \phie$ on $\Omega_{eq}$ for some constants $A_+$ and $A_-$.
Because of the continuity of $\phi$ at 
$P:= \Omega_N\cap\Omega_{eq} = \Cir_{\sin \,\xx_1}$
by \ref{LLD}.ii
we have 
$$
A_-\phie(\xx_1)=A_+\phio(\xx_1).
$$
For $0<\epsilon_1<<\epsilon_2$ we
consider now the domain 
$\Omega_{\epsilon_1,\epsilon_2}:=
D_P(\epsilon_2) 
\setminus 
D_L(\epsilon_1) 
$.
By integrating $\Lcalp \Phi=0$ on $\Omega_{\epsilon_1,\epsilon_2}$
and integrating by parts we obtain
$$
\int_{\partial\Omega_{\epsilon_1,\epsilon_2} }
\frac\partial{\partial\eta}\Phi
\, + \,
2 \int_{\Omega_{\epsilon_1,\epsilon_2} }
\Phi
= 0.
$$
By taking the limit as $\epsilon_1\to 0$ first
and then as $\epsilon_2\to 0$ we obtain using the logarithmic
behavior near $L$ that 
$$
2\pi m = 2 \pi \cos \xx_1 \, 
( \, A_+ \, \phio(\xx_1) \, h_{1+}\,  +\,  
 A_- \, \phie(\xx_1) \, h_{1-} \, ).
$$
Solving the system of the two equations for $A_{\pm}$ and using the monotonicity
of $\phie$ and $\phio$ (recall \ref{Lphieo})
we conclude the proof.
\end{proof}

Motivated by the above lemma we have the following definition.
Note for later applications that $\phiunder$ depends linearly on 
$(\atilde,\btilde)\in\R^2$ 
and $\junder$ on $\btilde\in\R$:

\begin{definition}
\label{Dphiunder}
Given $\atilde,\btilde\in\R$ we define
\begin{equation*}
\begin{aligned}
\phiunder=\phiunder[\atilde,\btilde;\xx_1]\in \,\,
&
C^\infty_\xx( \, \{\xx\in[0,\pi/2)\} \, )
\bigcap
C^0_\xxx( \, \Spheq\setminus\{p_N,p_S\} \, ),
\\
\junder=\junder[\btilde;\xx_1]\in \,\,
&
C^\infty_\xx( \{\xx\in[\xx_1,\pi/2)\} \,)
\bigcap 
C^\infty_\xx(  \{\xx\in(0,\xx_1]\} \,)
\bigcap 
C^0_\xxx(\Spheq\setminus\{p_N,p_S\}\,),
\end{aligned}
\end{equation*}
by requesting they satisfy the initial data 
$$
\phiunder(\xx_1)=\atilde,
\qquad
\left. \frac{ \partial\phiunder }{ \partial\xxtilde } \right|_{\xx=\xx_1} =
\frac1m \frac{ \partial\phiunder }{ \partial\xx } (\xx_1) =  \btilde ,
\qquad
\junder(\xx_1)=0,
\qquad
\partial_{1+}\junder=\partial_{1-}\junder=m \btilde,
$$
and the ODEs $\Lcalp\phiunder=0$ on
$\{\xx\in[0,\pi/2)\}$, 
and $\Lcalp\junder=0$ on $\{\xx\in[\xx_1,\pi/2)\}\subset\Spheq$ and on
$\{\xx\in[0,\xx_1]\}\subset\Spheq$.
\end{definition}
To simplify the presentation we define also the function $\hhat:(0,\pi/2)\to\R$ by
\begin{equation}
\label{Ehhat}
\hhat(\xx):=
\frac{1}{2\cos\xx}
\frac
{\frac1{\phio(\xx)}\frac{\partial\phio}{\partial\xx}(\xx)
+
\frac1{\phie(\xx)}\frac{\partial\phie}{\partial\xx}(\xx)}
{\frac1{\phio(\xx)}\frac{\partial\phio}{\partial\xx}(\xx)
-
\frac1{\phie(\xx)}\frac{\partial\phie}{\partial\xx}(\xx)}
.
\end{equation}

\begin{cor}
\label{Cphiave}
On $\{\xx\in(-\xx_1,\pi/2)\}$ 
we have that 
$$
\phi:=\Phi_\ave[\xx_1,m]
=\phiunder[\phi_1,
\hhat(\xx_1)
;\xx_1]
+
\junder[
\textstyle\frac1{2\cos\xx_1}
;\xx_1].
$$
\end{cor}

\begin{proof}
By \ref{Lphiave}
on $\Cir_{\sin \,\xx_1}$ we have 
$$
\phi=\phi_1, 
\qquad
\partial_{1+\,} \phi = \phi_1\,h_{1+},
\qquad
\partial_{1-\,} \phi = \phi_1\,h_{1-}.
$$
By \ref{Dphiunder} the corresponding initial data for the right
hand side are
$$
\phi_1, 
\qquad
m\,(\,\hhat(\xx_1) + \textstyle\frac1{2\cos\xx_1} \,),
\qquad
m\,(\,-\hhat(\xx_1) + \frac1{2\cos\xx_1} \,).
$$
Using the definitions we calculate that 
$$
\hhat(\xx_1)=\textstyle\frac1{2m} (h_{1+}-h_{1-})\phi_1,
\qquad
\frac1{2\cos\xx_1}=\textstyle\frac1{2m} (h_{1+}+h_{1-})\phi_1.
$$
This implies that the initial data are the same for both
sides and therefore the corollary follows by the uniqueness of ODE solutions.
\end{proof}

The following is important for horizontal balancing considerations.
\begin{lemma}
\label{L:hhat}
$\frac{d\hhat}{d\xx} <0$ 
on $(0,\xxroot)$. 
$\hhat$ has a unique root in 
$(0,\xxroot)$
which we will denote by 
$\xxbal$.
\end{lemma}

\begin{proof}
By direct calculation using \ref{Dphieo} and \ref{Lphieo} we have 
$$
\begin{aligned}
2\hhat(\xx)
&=
\cos\xx
-\frac{\sin^2\xx}{\cos\xx}
-\sin2\xx\,\log\frac{1+\sin\xx}{\cos\xx},
\\
\lim_{\xx\to0+} \hhat(\xx) &=1/2
\qquad\quad
\hhat(\xxroot) =
-\frac{\cos\xxroot}2-\frac{\sin^2\xxroot}{2\cos\xxroot}<0.
\\
2\frac{d\hhat}{d\xx} 
&=
-5\sin\xx-\frac{\sin^3\xx}{\cos^2\xx}-2\log\frac{1+\sin\xx}{\cos\xx}+4\sin^2\xx\log\frac{1+\sin\xx}{\cos\xx}.
\end{aligned}
$$
We clearly have $\log\frac{1+\sin\xx}{\cos\xx}>0$
and by \ref{Dphieo} and \ref{Lphieo} 
$\sin\xx\log\frac{1+\sin\xx}{\cos\xx}<1$ on $(0,\xxroot)$.
It follows that on $(0,\xxroot)$
$
2\frac{d\hhat}{d\xx}<-\sin\xx<0 
$
which allows us to complete the proof.
\end{proof}

It will be very important that we often work 
with solutions which are ``almost'' symmetric with respect to
reflection in the $\xx$ coordinate across $\xx_1$ in a sense
made precise later.
In this spirit we define an ``antisymmetrization''
$\Acalxx     $ as follows.

\begin{definition}
\label{DAcal}
We define 
$\Omega_1 = \Omega_1[\xx_1,m] := \, D_{\Lpar[\xx_1]\,}(3/m)$ 
and given $u\in C_\sym^0(\Omega_1[\xx_1,m] )$ 
we define a function
$\Acalxx     u \in C_\sym^0(\Omega_1[\xx_1,m]  \,)$ 
by
$$
\Acalxx     u(\xx_1+\xx',\yy)=
u(\xx_1+\xx',\yy)
-
u(\xx_1-\xx',\yy),
\qquad
\text{ for }\quad
\xx'\in(-3/m \, , \, 3/m),\quad
\yy\in\R.
$$
\end{definition}

\begin{lemma}
\label{Lode}
The following estimates hold.
\newline
(i). 
$\|  \, \phiunder[1, 0 ;\xx_1] -1 \, : \, C_\sym^{2}( \, \Omega_1[\xx_1,m]\,,\gtilde \, )\, \|\, \le \, C \, /m^2 $.
\newline
(ii).
$\|  \, \junder[1;\xx_1] - m \,|\,|\xx|-\xx_1\,|  \, : \, C_\sym^{2}( \, \Omega_1[\xx_1,m] \setminus \Lpar[\xx_1] \,,\gtilde \, ) \, \| \, \le \, C \, /m \, $.
\newline
(iii).
$\| \, \Acalxx  \,   \phiunder[1, 0 ;\xx_1] \, : C_\sym^{2}( \, \Omega_1[\xx_1,m]\, , \gtilde \, ) \, \|\le \,  C \,  /m^2$.
\newline
(iv).
$\|  \, \Acalxx  \,   \junder[1;\xx_1] \, : \, C_\sym^{2}( \, \Omega_1[\xx_1,m] \setminus \Lpar[\xx_1] \,,\gtilde \, ) \, \| \, \le  \, C \, /m $.
\newline
(v). 
$\|  \, \phiunder[0, 1 ;\xx_1] - m \,(\,|\xx|-\xx_1\,)  \, : \, C_\sym^{2}( \, \Omega_1[\xx_1,m]\,,\gtilde \, )\, \|\, \le \, C \, /m $.
\end{lemma}

\begin{proof}
Let 
$\phiunder_1:=\phiunder[1, 0 ;\xx_1]$. 
$\phiunder_1$ satisfies 
the ODE 
$ \Lcalp \phiunder_1 = 0$
which in the notation of \ref{Eghat} 
amounts to 
$$
\partial_{\xxhat}^2 \, \phiunder_1 - m^{-1}\, \tan(\xx_1+m^{-1}\xxhat)\,\partial_{\xxhat}\,  \phiunder_1 +2m^{-2} \,\phiunder_1 
\, = \,
0\, .
$$
Consider $\Omega$, 
the subset of $\Omega_1[\xx_1,m]$ where $|\phiunder_1-1|<1/2$ and  
$|\partial_{\xxhat}\,  \phiunder_1| < 1/m $.
Using the equation we have 
$|\,\partial_{\xxhat}^2 \, \phiunder_1 \,|\,\le C/m^2$ 
on that subset, 
and therefore we obtain a contradiction unless it is the whole $\Omega_1[\xx_1,m]$.  
This proves (i).
The proof of (ii) is similar 
(note that $m \, \partial_{\xxhat}\,  | \, |\xx|-\xx_1\,|  \,=\,\pm 1\,$). 
(iii) follows then from  (i) and (iv) from (ii).
(v) is equivalent to (ii). 
\end{proof}

\subsection*{Estimates on $\Phi=\Phi[\xx_1,m]$}
$\phantom{ab}$
\nopagebreak

Lemma \ref{Lphiave} provides explicit information on 
$\Phi_\ave$.
We need to estimate $\Phi_\osc$ also.
To this end we introduce a new decomposition 
$\Phi=\Pp+\Ptp$ 
as follows.

\begin{definition}
\label{DV}
\label{EV}
We define $\Pp\in C^\infty_\sym(\Spheq\setminus L)$ by 
requesting that 
\begin{equation*}
\Pp:=
\{\,\Psibold[2\deltaL,3\deltaL; I_{\R^+}]
(G,A_1)
\,\}
\circ
\dbold_L,
\quad\text{where}\quad
A_1:=\,\log\deltaL \, .
\end{equation*}
\end{definition}

Observe that by \ref{LGp} for each $p\in L$ we have on $D_p(2\deltaL)$ 
that $\Pp=G_p$.
This, \ref{LLD}, \ref{DLDnoK}.ii, and 
the definition of $\Phi$ in \ref{Etwovarphi},
imply that $\Phi-\Pp$ can be extended smoothly
across $L$. 
This allows us to have the following.

\begin{definition}
\label{DW}
\label{EW}
We define 
$\Ptp,E''\in C^\infty_\sym(\Spheq)$ by requesting that on
$\Spheq\setminus L$ 
$$
\Phi=\Pp+\Ptp,
\qquad
E'':=-\Lcalp \Pp.
$$
\end{definition}

Note that by \ref{EV} and \ref{LGp} $E''$ vanishes on $D_L(2\deltaL)$.
By \ref{Dave} we have 
\begin{equation}
\label{EWdec}
\begin{aligned}
\Ptp_\ave&,\Ptp_\osc,E''_\ave,E''_\osc\in C^\infty_\sym(\Spheq),\\
\text{and on }\Spheq \quad &
\Ptp=\Ptp_\ave+ \Ptp_\osc,
\quad
E''=E''_\ave+ E''_\osc.
\end{aligned}
\end{equation}

Since $\Lcalp\Phi$ vanishes by \ref{LLD}
and \ref{Etwovarphi},
and $\Lcalp$ is rotationally covariant, 
we conclude from \ref{EW}
and \ref{EWdec}
that
\begin{equation}
\label{ELW}
\Lcalp \Ptp = E'',
\qquad
\Lcalp \Ptp_\ave = E''_\ave,
\qquad
\Lcalp \Ptp_\osc = E''_\osc
\quad\text{ on } \Spheq.
\end{equation}
Moreover since 
$\Phi_\ave\in C^0(\Spheq)$
by \ref{LLD} and \ref{Etwovarphi},
and 
$\Ptp_\ave\in C^\infty_\sym(\Spheq)$ by \ref{EWdec},
we conclude by \ref{EW} that 
\begin{equation}
\label{EWosc}
\begin{aligned}
\Pp_\ave\in C^0(\Spheq),
\qquad & 
\Ptp_\ave= \Phi_\ave-\Pp_\ave=\phi-\Pp_\ave
\quad\text{ on }\quad
\Spheq,
\\
\Phi=&\,\phi+\Pp_\osc+\Ptp_\osc
\quad\text{ on }\quad
\Spheq\setminus L.
\end{aligned}
\end{equation}
Using \ref{EWdec} we conclude that
\begin{equation}
\label{EWposc}
\Ptp=\phi-\Pp_\ave + \Ptp_\osc
\quad \text{ on } \quad
\Spheq.
\end{equation}
Note that in this expression although $\phi$ and $\Pp_\ave$
are not smooth because of a derivative jump at $\xx=\xx_1$,
we do have
$\phi-\Pp_\ave =\Ptp_\ave\in C^\infty_\sym(\Spheq)$
because the derivative jumps cancel out.
Note also that neither $\Lcalp \Ptp_\ave$ nor $\Lcalp \Ptp_\osc$ have to vanish
on $D_L(2\deltaL)$ but their sum does.

$\phi$ is known explicitly by \ref{Lphiave}.
We need to estimate $\Pp_\ave$ and $\Ptp_\osc$.
$\Pp$ is almost explicit and $\Ptp_\osc$ can be estimated by estimating $E''_\osc$
and using \ref{ELW}.
We first estimate $E''$ as follows. 
Note that 
$(\,\Acalxx     E''\,)_\osc \, = \Acalxx     (\, E''_\osc \,) \, $
by the definitions. 

\begin{lemma}
\label{LHest}
The following hold (recall \ref{Exxone} and \ref{Egtilde}).
\newline
(i).
$\|\Pp- A_1:C_\sym^{k}( \, \Spheq \setminus D_L(\deltaL)\,,\gtilde \, )\|\le C(k)\,$
and $\Pp-A_1$ vanishes on $\Spheq\setminus D_L(3\deltaL)$.
\newline
(ii).
$\|m^{-2}E'':C_\sym^{k}( \, \Spheq ,\gtilde \, )\|\le C(k)\,$, 
$\|m^{-2}E''_\ave:C_\sym^{k}( \, \Spheq ,\gtilde \, )\|\le C(k)\,$, 
and
$E''$ vanishes on $D_L(2\deltaL)$. 
\newline
(iii).
$\|m^{-2}E''_\osc:C_\sym^{k}( \, \Spheq ,\gtilde \, )\|\le C(k)\,$
and 
$E''_\osc$ is supported on $D_{\Lpar}(3\deltaL)\subset \Omega_1[\xx_1,m]\, $.
\newline
(iv).
$\|m^{-2}\Acalxx     E'' :C_\sym^{k}( \, \Omega_1[\xx_1,m]\, , \gtilde \, )\|\le C(k)\,/m$ 
and
\newline 
$\|m^{-2}\Acalxx     E''_\ave :C_\sym^{k}( \, \Omega_1[\xx_1,m]\, , \gtilde \, )\|\le C(k)\,/m$.  
\newline
(v).
$\|m^{-2}\Acalxx     E''_\osc :C_\sym^{k}( \, \Omega_1[\xx_1,m]\, , \gtilde \, )\|\le C(k)\,/m$
and 
$\Acalxx     E''_\osc$ is supported on $D_{\Lpar}(3\deltaL)$.
\end{lemma}

\begin{proof}
(i) follows from \ref{LGp}.vi and the definitions.
By \ref{ELcalp} and \ref{EW} we have
$$
m^{-2}E''=
-\Lcalp_\gtilde \Pp =
-\Lcalp_\gtilde ( \Pp - A_1 ) - 2 m^{-2} A_1.
$$
As mentioned earlier $E''$ vanishes on $D_L(2\deltaL)$ 
and clearly 
$-1 < m^{-2} A_1 < 0$ 
by \ref{Edelta}. 
(ii) follows then from (i). 
The second part of (iii) follows 
from (i) which implies that 
$\Pp=A_1$ on $\Spheq\setminus D_L(3\deltaL)$.
The first part of (iii) follows then from (ii) 
and the second part. 

Recall now the coordinates defined in \ref{Egtilde} 
and define a new coordinate $\xxhat:=\xxtilde-m\xx_1$.
The metric $\gtilde$ then in coordinates $(\xxhat,\yytilde)$
is equal to the metric
\begin{equation}
\label{Eghat}
\ghat_t:=d\xxhat^2+\cos^2(\xx_1+t\xxhat)\,d\yytilde^2
\end{equation}
with $t=1/m$.
$\ghat_t$ clearly depends smoothly on $t$ for $|t|$ small,
and for $t=0$ is the Euclidean metric with $(\xxhat,\yytilde)$ 
the standard coordinates.
This implies that $\dbold_{(0,0)}$ as a function of $(t,\xxhat,\yytilde)$
is smooth for small $|t|$ and bounded $(\xxhat,\yytilde)$ independently of $m$.
Since 
$\Acalxx     \dbold_{(0,0)}$ clearly vanishes for $t=0$ we conclude that
$\|\Acalxx      \dbold_L:C_\sym^{k}(D_L(3\deltaL)\, \setminus D_L(2\deltaL)\, , \gtilde \, )\|\le C(k)\,/m$.

Note now that 
$\Acalxx     E''$ is supported on 
$D_L(3\deltaL)\, \setminus D_L(2\deltaL)$. 
By \ref{EV} and \ref{EW} $E''$ factors through $\dbold_L$.
(iv) follows then from (ii) (with a loss of one derivative)
by the estimate on 
$\Acalxx      \dbold_L$
above.
By averaging over the circles then and subtracting we obtain the estimate in (v).
The statement on the support follows from the support of $E''_\osc$ in (iii).
\end{proof}

\begin{lemma}
\label{Lsol}
Given 
$E\in C^{0,\beta}_\sym(\Spheq)$ 
with 
$E_\ave\equiv 0$ 
and $E$ supported on 
$\Omega_1[\xx_1,m]$, 
there is a unique 
$v \in C^{2,\beta}_\sym(\Spheq)$
characterized by (i) below 
and satisfying the following.
\newline
(i).
$\Lcalp  v=E$, or equivalently $\Lgtp v=m^{-2}E$.
\newline
(ii). 
$v_\ave=0$.
\newline
(iii).
$\|v:C^{2,\beta}_\sym(\Spheq,\gtilde, \fdecay )\|
\le
C \,
\|m^{-2} E:C_\sym^{0,\beta}( \Omega_1[\xx_1,m]\, ,\gtilde \, )\|
$
(recall \ref{D:newweightedHolder}),
where we have 
$\fdecay:= e^{-c_1 m \min ( \left|\vphantom{A^A} |\xx|-\xx_1\right| , c_2\, ) } $ 
for some absolute constants $c_1,c_2>0$.
\newline
(iv).
$\|\Acalxx      v:C^{2,\beta}_\sym( \Omega_1[\xx_1,m]\, ,\gtilde \, )\|
\le
$
\newline
$
\phantom{\|\Acalxx      v:C^{2,\beta}_\sym}
$ 
\hfill 
$
C \,\left(m^{-1} \|m^{-2} E:C_\sym^{0,\beta }( \Omega_1[\xx_1,m]\, ,\gtilde \, )\| +
\|m^{-2} \Acalxx     E : C_\sym^{0,\beta }( \Omega_1[\xx_1,m]\, ,\gtilde \, )\|\right)$.
\end{lemma}

\begin{proof}
The existence and uniqueness of $v$ is clear by the symmetries.
(ii) follows also because $\Lcalp v_\ave=E_\ave=0$.

For (iii) observe first that 
$\Spheq\setminus \{p_N,p_S\}$
equipped with the metric (recall \ref{EThetag})
$$
\chi:=
m^2\cos^{-2}\xx \, g
\ = \,
m^2\cos^{-2}\xx \, d\xx^2+m^2\,d\yy^2, 
$$
can be isometrically identified with the cylinder
$\R\times \Sph^1$ 
equipped with the metric $\chi=ds^2+d\theta^2$, 
where $(s,\theta)$ (with $\theta$ defined modulo $2\pi m$) 
denotes the standard coordinates of the cylinder.
Under this identification we can assume that $s$ is an odd function of $\xx$
and $\theta=m \yy = \yytilde$.
By (ii) $u(p_N)=u(p_S)=0$ and on the cylinder the equation
(i) is equivalent to 
$$
\left( \Delta_\chi  + 2m^{-2}\cos^2\xx \right) v
\, = \,
m^{-2}\cos^2\xx \, E.
$$
Because of the symmetries we can work with $\theta$ modulo $2\pi$ 
instead of $2\pi m$.
Let $v'$ be the solution on the cylinder of 
$$
\Delta_\chi v'   
\, = \,
m^{-2}\cos^2\xx \, E 
$$
subject to the condition $v'\to0$ as $s\to\pm\infty$.
By standard theory and separation of variables, 
and using also \ref{Exxone} to ensure the uniform equivalence of
$\chi$ and $\gtilde$ on the support of $E$, we have 
$$
\|v':C^{2,\beta}_\sym(\R\times\Sph^1,\chi , 
e^{-  \left|\vphantom{A^A} |s(\xx)|-s(\xx_1) \right| \,  /2 } 
)\|
\le
C \,
\|m^{-2} E:C_\sym^{0,\beta}( \Omega_1[\xx_1,m]\, ,\gtilde \, )\| . 
$$ 
$v-v'$ now satisfies the equation
$$
\left( \Delta_\chi + 2m^{-2}\cos^2\xx \right) (v-v')  
\, = \,
- 2m^{-2}\cos^2\xx \, v'.  
$$
Note that $v'_\ave$ clearly vanishes. 
Using the smallness of the perturbation introduced by the coefficient
$ 2m^{-2}\cos^2\xx $, 
and the estimate on $v'$ above, 
we conclude that
$$
\|v-v':C^{2,\beta}_\sym(\R\times\Sph^1,\chi , 
e^{-  \left|\vphantom{A^A} |s(\xx)|-s(\xx_1) \right| \,  /2 } 
)\|
\le
C \,
m^{-2} \,
\|m^{-2} E:C_\sym^{0,\beta}( \Omega_1[\xx_1,m]\, ,\gtilde \, )\|.
$$ 
We have then 
$$
\|v:C^{2,\beta}_\sym(\R\times\Sph^1,\chi , 
e^{-  \left|\vphantom{A^A} |s(\xx)|-s(\xx_1) \right| \,  /2 } 
)\|
\le
C \,
\|m^{-2} E:C_\sym^{0,\beta}( \Omega_1[\xx_1,m]\, ,\gtilde \, )\|.
$$ 
By choosing now $c_1$ and $c_2$ appropriately (iii) follows easily.

To prove (iv) recall first that we are working on the cylinder 
$\R\times \Sph^1$ 
equipped with the metric $\chi=ds^2+d\theta^2$, 
where $(s,\theta)$ denotes the standard coordinates on the cylinder 
with $\theta$ defined modulo $2\pi$. 
$\Spheq\setminus \{p_N,p_S\}$ is then an $m$ to $1$ covering 
of the cylinder, and the coordinate $\xx$ can be considered as a function 
on the cylinder as well with 
$$ 
\frac{ds}{d\xx}=\frac{m}{\cos\xx}.  
$$ 
We define following \ref{EPsibold} $s_-:\R\to\R$ by 
$$
s_-(\xx) :=
\Psibold \left[3/m \, , \, 6/m \, , \, |\xx-\xx_1| \, \right] \, ( \, s(2\xx_1-\xx) \,  , \, 2 s(\xx_1) - s(\xx) \, ).
$$
For $u\in C^0(\R\times\Sph^1)$ we define 
$\Acal u \in C^0(\Omega_1[\xx_1,m]  \,)$ 
by 
$$
\Acal u(s,\theta) := u (s,\theta)-u(\, s_-(s) \, ,\theta) .
$$ 
Note that $\Acal u$ agrees 
on $\Omega_1\cap\{\xx>0\}$ 
with 
$\Acalxx     u$ defined as in \ref{DAcal}. 
We define now $v'_\pm,E'_\pm \in C^0(\R\times\Sph^1)$ by requesting 
$$
\Delta_\chi v'_\pm    
\, = \,
m^{-2}\cos^2\xx \, E'_\pm,  
\qquad
E=E^+ + E_- , 
\qquad
v=v_+ + v_-,  
$$
where $E=E_+$ on $\{s>0\}$ and $v'\to0$ as $s\to\pm\infty$.
We have then that $v_+(s,\theta)\equiv v_-(-s,\theta)$ 
and 
$$
\Delta_\chi  \Acal v'_+ 
= 
\Acal\left( m^{-2}\cos^2\xx \, E'_+ \right) 
+
\left[ \Delta_\chi  , \Acal \right] v'_+. 
$$ 
Using the definitions it is clear that $\frac{d s_-}{ds_{\phantom{-}}}+1$ and both terms on the right of the equation 
are supported on $\{|\xx-\xx_1|<6/m\}$. 
Moreover by an easy calculation 
$$
\left\| \frac{d s_-}{ds_{\phantom{-}}}+1  \,:\, C^3 \, \right\| 
\le 
\, C/m. 
$$ 
Estimating first $v'_+$ 
and then $\Acal v'_+$ 
(both with exponential decay $e^{-|s-s(\xx_1)|}$),   
we conclude that (iv) is valid with $v$ replaced by $v'_+$. 
Since 
$s(\xxbal)>m\xxbal$ we have 
$e^{-s(\xxbal)\,} << m^{-2} $ 
and therefore we conclude that (iv) is valid with $v$ replaced by $v'$. 
Combining with the earlier estimate for $v-v'$ we conclude the proof. 
\end{proof}

\begin{lemma}
\label{LWoscest}
$\Ptp_\osc$ satisfies the following estimates.
\newline
(i).
$\|\Ptp_\osc:C^{2,\beta}_\sym(\Spheq,\gtilde, \fdecay )\|
\le
C$,
where $\fdecay$ is as in \ref{Lsol}.iii.
\newline
(ii).
$\|\Acalxx      \Ptp_\osc:C^{2,\beta}_\sym(\Omega_1[\xx_1,m]\, ,\gtilde \, )\|
\le
C \,/m$
(recall \ref{DAcal}).
\end{lemma}

\begin{proof}
Since $\Lcalp \Ptp_\osc = E''_\osc$ by \ref{ELW}
we can use the estimates of \ref{LHest} to conclude the proof
by appealing to \ref{Lsol}. 
\end{proof}

It helps with the presentation of our estimates to introduce one more
decomposition which holds in the vicinity of $L$:

\begin{definition}
\label{DPhippp}
\label{EU}
We define 
$\Phip\in C^\infty_\sym(\Omega_1[\xx_1,m])$
by requesting that (recall \ref{Dphiunder})
$$
\Ptp=\Phip+\phiunder[\phi_1- A_1,
\hhat(\xx_1) \,
;\xx_1]
\quad\text{ on }\quad
\Omega_1[\xx_1,m].
$$
\end{definition}

Using \ref{Dphiunder} we have then
for $p\in L\cap\Omega_1[\xx_1,m]$ 
\begin{equation}
\label{Eppp}
\Ptp(p)=
\,\Phip(p)+\phi_1-A_1,
\qquad
d_p\Ptp=
\,
d_p\Phip+
\hhat(\xx_1) \,
d\xxtilde
\, .
\end{equation}

\begin{lemma}
\label{Lppp}
The following estimates hold.
\newline
(i). 
$\|\Phip:C_\sym^{2,\beta}(\Omega_1[\xx_1,m]\,,\gtilde \, )\|\le C$.
\newline
(ii).
$\|\Acalxx     \Phip : C_\sym^{2,\beta}(\Omega_1[\xx_1,m]\, , \gtilde \, )\|\le C/m$.
\end{lemma}

\begin{proof}
By combining 
\ref{EWosc}, \ref{DPhippp}, and \ref{Cphiave}, we conclude
$$
\Phip_\ave=
\phiunder[A_1, 0 ;\xx_1]
+
\junder[
\textstyle\frac1{2\cos\xx_1}
;\xx_1]
- \Pp_\ave ,
\qquad\qquad
\Phip_\osc =\Ptp_\osc.
$$
Note that the discontinuities on the right hand side of the first equation cancel and the left hand side is smooth.
Moreover by \ref{EU} and \ref{ELW} we have 
$ \Lcalp \Phip_\ave = \, E''_\ave $ 
which in the notation of \ref{Eghat} 
amounts to the ODE
$$
\partial_{\xxhat}^2 \, \Phip_\ave - m^{-1}\, \tan(\xx_1+m^{-1}\xxhat)\,\partial_{\xxhat}\,  \Phip_\ave  +2m^{-2} \,\Phip_\ave 
\, = \,
m^{-2} E''_\ave .
$$
Using then 
\ref{Lode} and that $\Pp_\ave = A_1$ on  
$\Spheq\setminus D_{\Lpar}(3\deltaL)$
by \ref{LHest}.i, 
we conclude that at $\partial \Omega_1[\xx_1,m]$ 
we have
$|\, \Phip_\ave \, | \, \le \, C \, $ 
and 
$|\, \partial_{\xxhat}\,  \Phip_\ave \, | \, \le \, C \,$. 
Using these as initial data for the ODE 
and \ref{LHest}.ii 
we estimate $\Phip_\ave$.
By this estimate 
together with 
\ref{LWoscest}.i 
we conclude (i).

To prove (ii) 
it is enough to prove the estimate for 
$
U:=\Acalxx     \Phip_\ave
$
instead of 
$ 
\Acalxx     \Phip
$
because 
$
\Acalxx      \Phip_\osc =
\Acalxx      \Ptp_\osc
$
satisfies the estimate by 
\ref{LWoscest}.ii. 
To estimate $U$ we calculate that it satisfies
\begin{multline*}
\partial_{\xxhat}^2 U- m^{-1}\, \tan(\xx_1- m^{-1} \xxhat)\,\partial_{\xxhat} U +2m^{-2} \,U
+ 
\\
+
m^{-1}\, (\,
\tan(\xx_1-m^{-1} \xxhat)
-
\tan(\xx_1+m^{-1} \xxhat)
\,) \, 
\partial_{\xxhat} \Phip_\ave
\,=\,
m^{-2}\Acalxx     E'' .
\end{multline*}
Using \ref{Lode} we obtain estimates for the initial data 
for $U$ on 
$\partial \Omega_1[\xx_1,m]$. 
These estimates and \ref{LHest}.iv imply the required estimate.  
\end{proof}

\begin{lemma}
\label{LPtp}
(i).
$\|\Ptp :C_\sym^{k,\beta}(\,\Spheq\,,\gtilde \, )\|\le C(k) \, m     $.
\newline
(ii).
$C' \,m\le \frac89 \phi_1\le\Ptp$ on $\Spheq$. 
\end{lemma}

\begin{proof}
By \ref{EU}, \ref{Lppp}.i, and \ref{Lode}, 
we have that 
$C'\,m\le \frac89 \phi_1\le\Ptp\le C\, m\,$ 
on $\Omega_1[\xx_1,m] $. 
By \ref{EWposc} and \ref{LHest}.i we have on $\Spheq\setminus D_L(3\deltaL)$ 
that 
$ \Ptp=\phi- A_1+ \Ptp_\osc $.
By \ref{LWoscest}.i and \ref{Lphiave} we conclude that
$C' \,m\le \frac89 \phi_1\le\Ptp \le C \, m \,$ 
on $ \Spheq\setminus D_L(3\deltaL) $.   
Since $\Omega_1[\xx_1,m]$ and $\Spheq\setminus D_L(3\deltaL)$       
cover $\Spheq$ the proof of (ii) 
is complete and (i) follows by standard interior regularity theory and \ref{LHest}.ii.
\end{proof}

\subsection*{Estimates on $V,V',W,W'$}
$\phantom{ab}$
\nopagebreak

\begin{lemma}
\label{LVVp}
$V,V'\in C^\infty_\sym(\Spheq)$ defined as in \ref{Etwovarphi} 
satisfy the following.
\newline
(i).
On 
$\Spheq\setminus D_L(2\delta_1)$
we have $V=\Phi$ and $V'=0$.
\newline
(ii).
With $p_1$ as in \ref{EtwoL} and $u$ as in \ref{EWW} we have that on 
$D_{p_1}(2\delta_1)$
the following hold. 
$$
\begin{aligned}
V&=
\Psibold\left[2\delta_1,\delta_1; \dbold_{p_1} \right]
(G_{p_1},\,\log \delta_{1} \, \cos\circ\dbold_{p_1}\,)
+\Ptp,
\\
V'&=
\Psibold\left[2\delta_1,\delta_1; \dbold_{p_1} \right]
(0,u\,).
\end{aligned}
$$
(iii).
$V(p_1)=\phi_1+\Phip(p_1) \, \sim_C \, m $,
$\quad\displaystyle\frac{\partial V}{\partial\xxtilde}(p_1)=
\hhat(\xx_1)
+
\displaystyle\frac{\partial \Phip}{\partial\xxtilde}(p_1)\,
\,\sim_C\,1 
$ and 
$\left|\displaystyle\frac{\partial V}{\partial\xxtilde}(p_1)\right| 
\,\le\, C$,
for some absolute constant $C>1$.
We also have 
$\quad V'(p_1)=0$,
$\quad\displaystyle\frac{\partial V'}{\partial\xxtilde}(p_1)=m^{-1}$.
\newline
(iv). 
$\|V: C_\sym^{k}(\,\Spheq\,,\gtilde \, )\| \, \le  \, C(k) \, m     $
and
$\|V': C_\sym^{k}(\,\Spheq\,,\gtilde \, )\| \, \le  \, C(k) / m     $.
\newline
(v). 
\ref{AEcal}  holds and 
$\|\Ecalinv \| \,  \le \, C \, m^{2+\beta} \, $ (recall \ref{DEcalnorm}). 
\newline
(vi). 
For $\mu,\mu'\in\R$ we have that 
$|\mu| m + |\mu'| \, \le \,C \, \|\, \mu W + \mu' W'  \, : \, C_\sym^{0,\beta}(\,\Spheq\,,g \, )\| $.
\end{lemma}

\begin{proof}
Let $V_{new}$ and $V'_{new}$ be defined by the expressions for $V$ and $V'$ in (i) and (ii). 
To establish (i) and (ii) we need to prove that $V=V_{new}$ and $V'=V'_{new}$. 
Note first that the expressions for $V_{new}$ and $V'_{new}$ in (i) and (ii) match 
because by \ref{DV} and \ref{DW}
$V_{new}=\Phi$ and $V'_{new}=0$ on a neighborhood of 
$\partial D_{p_1}(2\delta_1)$.
Since $\Lcalp \Phi=0$ by \ref{Etwovarphi}, 
we conclude that by \ref{EWW} we have 
that 
$\Lcalp V_{new}=W$ and $\Lcalp V'_{new}= W'$,
which characterize 
$V$ and $V'$ defined as in \ref{Etwovarphi}.
This completes the proof of (i) and (ii).
The equalities in (iii) follow from (ii) by using \ref{Eppp}, the definition
of $u$ in \ref{EWW}, and the definition of $A_1$ in \ref{DV}.
Clearly by \ref{Lphiave} and \ref{Exxone} we have $C' \, m < \phi_1<C\, m$.
Because of \ref{Exxone} we have also that $ | \hhat(\xx_1) | < C$.
The proof of (iii) is completed then by using \ref{Lppp}.  

(iv) is implied by (i) and (ii) by using \ref{LPtp} and \ref{LHest}.i.
By direct calculation using (iii) we have  (recall \ref{DEcal})
$$
\begin{aligned}
\Ecalinv(\,(1,0)_{p\in L}\,)
&=
\,(\phi_1+\Phip(p_1)\,)^{-1} \!\!\!
& \left\{ 
V
-
(\,
\hhat(\xx_1)
+
\frac{\partial \Phip}{\partial\xxtilde}(p_1)
\,)
\,m\, 
V'
\right\}
&,
\\
\Ecalinv(\,(0,d_p |\xxtilde|)_{p\in L}\,)
&=
&
m\, V' \phantom{,} &.
\end{aligned}
$$
(v) follows then by using (iii) and (iv).

We have 
$
|\mu| m + |\mu'|  
\,\le\,
\,C \, \|\, \mu V + \mu' V'  \, : \, C_\sym^{2,\beta}(\,\Spheq\,,g \, )\| 
\,\le\,
\,C \, \|\, \mu W + \mu' W'  \, : \, C_\sym^{0,\beta}(\,\Spheq\,,g \, )\|,  
$
where the first inequality follows easily from (iii) and the second inequality 
from the symmetries and $\Lcal\,( \mu V + \mu' V' ) = \mu W + \mu' W' $.  
(vi) follows then and the proof is complete.
\end{proof}

\subsection*{The family of MLD solutions}
$\phantom{ab}$
\nopagebreak

We determine now the family of MLD solutions we need.
The parameters of the family are 
$\zetabold=(\zeta,\zeta')\in\R^2$
and their range is specified by
\begin{equation}
\label{Ezetarange}
|\zeta|,|\zeta'|\le \cunder_1,
\end{equation}
where $\cunder_1>1$ is a 
constant independent of $m$ and $\tau$ which will be specified later.
We want to construct MLD solutions $\varphi[L,\tau,\xiw]$ where
$L = L[\xx_1[\zetabold,m],m]$, 
$\tau=\tau[\zetabold,m]$, 
and $\xiw = \xiw[\zetabold,m]$, 
that is the parameters of $\varphi$ are functions of $\zetabold$ and $m$. 
Clearly then $\varphi=\tau\Phi+\mu V+\mu' V'$ 
where 
$\xiw= \mu W+\mu' W'$ 
(see also \ref{E:varphi} below). 
Note that by 
\ref{DLDnoK}.ii, \ref{EV}, \ref{EW}, and \ref{E:varphi}, 
we have that 
\begin{equation}
\label{EPhip}
\forall p\in L
\qquad\qquad
\varphihat_p\,=\,\tau\,\Ptp+\mu V+\mu' V'  
\qquad
\text{ on } \quad
D_p(2\deltaL)\, .
\end{equation}
The matching condition in \ref{DMLD} 
amounts to a system of two equations with $\xx_1$ and $\tau$ 
as the unknowns where we assume $\mu$ and $\mu'$ 
given in terms of $\zetabold$ and $m$. 
We will write later this system explicitly by using 
\ref{EPhip}, \ref{Eppp}, and  \ref{LVVp}.iii. 
We consider now the following simplified approximate version 
which is obtained by treating $\Phip$ as an error term to be ignored 
and making appropriate simple choices for $\mu$ and $\mu'$ 
(recall also $d\xx=\frac1m d\xxtilde$).
$$
\tau\,(\phi_1- A_1) +\tau\log(\tau/2)=\tau\zeta,
\qquad\quad
\tau\, m \, \hhat(\xx_1) d\xx=\tau\zeta' d\xx. 
$$
By straightforward calculation 
using \ref{Lphiave} 
this is equivalent to 
$\tau = 2 \, e^{\zeta} \, e^{A_1-\phi_1}$
and
$m\,{\hhat(\xx_1)} = \zeta' $.
To ensure a simplified expression we use a modified 
(by replacing $A_1$ with $A'_1$) 
version of these conditions to define $\xx_1$ and $\tau$ in \ref{Eparameters} below. 

\begin{lemma}
\label{L:h}
For $m$ large enough depending on $\cunder_1$,
and $\zetabold=(\zeta,\zeta')$ as in \ref{Ezetarange},
there are unique 
$\xx_1=\xx_1[\zetabold,m]\in(0,\xxroot)$ 
and 
$\tau=\tau[\zetabold,m]>0$ 
satisfying 
\begin{equation}
\label{Eparameters}
\tau = 2 \, e^{\zeta} \, e^{A'_1-\phi_1}
=
\frac1m  \, e^{\zeta} \, e^{ -\frac{m}{\cos\xx_1\,(h_{1+}+h_{1-})} }, 
\qquad\qquad
{\hhat(\xx_1)}
=
\frac{\zeta'}m,
\end{equation}
where $ A'_1:=-\log 2m $ .
Moreover there is a unique
$$
\xiw=\xiw[\zetabold,m] 
:=
\mu[\zetabold,m]\, W[\xx_1,m]
+
\mu'[\zetabold,m]\, W'[\xx_1,m]
$$
such that
(recall \ref{LLD})
\begin{equation}
\label{E:varphi}
\begin{aligned}
\varphi= \varphi[[\zetabold,m]]
:=&
\tau[\zetabold,m] \, \Phi[\xx_1,m] \, +\, \mu[\zetabold,m] \,  V[\xx_1,m] \,  +\,  \mu'[\zetabold,m] \,  V'[\xx_1,m] \, 
\\
=&
\varphi[\, L[\xx_1[\zetabold,m],m]\,,\,\tau[\zetabold,m]\,,\,\xiw[\zetabold,m]\,]
\end{aligned}
\end{equation}
is an MLD solution as in \ref{DMLD}.
Furthermore the following hold.
\newline
(i).
$\xx_1=\xx_1[\zetabold,m]$,
$\tau=\tau[\zetabold,m]$, 
$\mu=\mu[\zetabold,m]$,
and
$\mu'=\mu'[\zetabold,m]$ 
depend continuously on $\zetabold$.
\newline
(ii).
$| \xx_1[\zetabold,m] - \xxbal | \le 
\,C\, |\zeta' | /m
\le
\,C\, \cunder_1 /m$.
\newline
(iii).
In the notation of \ref{Dsimc}  
we have 
$$
m \sim_2 |\log\tau| \cos\xxbal \, \left.  (h_{1+}+h_{1-})  \right|_{\xx=\xxbal},
\qquad\quad 
\tb \, \sim_{C(\cunder_1)} \,   {\tau[\zetabold,m]} ,
$$
where $\tb:=\tau[\, (0,0)\, ,m]$ 
and
${C(\cunder_1)}>1$ 
depends only on $\cunder_1$. 
\newline
(iv).
$ | \, \zeta + \mu \, \phi_1/\tau \, | \le C $
and 
$ | \, \zeta' + \mu' \, / \tau \, | \le C $.
\newline
(v).
$ \| \, \varphi[[\zetabold,m]] : 
C^{3,\beta}_\sym(\Spheq\setminus D_L(\deltaLp)    , g)\|
\le
\tau^{1-4\gammagl}
\le
\tau^{8/9}$.
\newline
(vi).
$c\,m\,\tau\, \le \frac78 \phi_1\,\tau\, \le \, \varphi$ on 
$\Spheq\setminus D_L(\deltaLp)   $ 
for some absolute constant $c>0$.
\newline
(vii).
$\varphi$ satisfies the conditions in \ref{con:one},  
\ref{AEcal}, and \ref{EEcal}.
\end{lemma}

\begin{proof}
The existence and uniqueness of $\xx_1$ and $\tau$, 
their smoothness, and also (ii), 
follow from \ref{Eparameters} and \ref{L:hhat}.
(iii) follows from (ii), \ref{Eparameters}, and \ref{Ezetarange}.
Using 
\ref{EPhip}, \ref{Eppp}, and  \ref{LVVp}.iii, 
we conclude that the matching conditions in \ref{DMLD} 
amount to 
$$
\begin{gathered}
\Phip(p)+\phi_1-A_1
+
\frac{\mu}{\tau} (\phi_1+\Phip(p)\,)
+\log\frac\tau2 
=0,
\\
\left(1+\frac{\mu}{\tau  } \right)
\,
\left( \hhat(\xx_1)
+
\displaystyle\frac{\partial \Phip}{\partial\xxtilde}(p) \,\right)
+\frac{\mu'}{m\tau}
=0.
\end{gathered}
$$
By further calculation and \ref{Eparameters} these conditions are equivalent to
$$
\mu= 
\,  - \tau \, \frac {A'_1-A_1+\zeta+\Phip(p)} { \phi_1+\Phip(p) },
\qquad
\mu'= 
\tau \, \left(  \frac {A'_1-A_1+\zeta+\Phip(p)} { \phi_1+\Phip(p) } -1
\right)
\left(
\zeta'
+
m\,\displaystyle\frac{\partial \Phip}{\partial\xxtilde}(p) \, \right).
$$
$\xiw$ is uniquely determined by these conditions and (i) follows.
Using \ref{Lppp} and the definition of $\phi_1$ in \ref{Lphiave} 
we conclude that
$$
 | \, \zeta + \mu \, \phi_1/\tau \, | \le \, C + C \cunder_1 /m,
\qquad
| \, \zeta' + \mu' \, / \tau \, | \le 
\, C + C \cunder_1 /m.
$$
(iv) then follows.
By \ref{DW} we have that 
$$
\varphi=
\tau \,\Pp
+
\tau\, \Ptp 
+
\mu\,
V
+
\mu'
\, V' .
$$
By (iv) and \ref{LVVp}.iv we have
\begin{equation}
\label{EmuV} 
\| \, \mu V+\mu' V' : C_\sym^{3,\beta}(\,\Spheq\,,\gtilde \, )\|\le C \, \cunder_1 \, \tau \, .      
\end{equation} 
Using \ref{LGp}.vii we obtain
$ \| \, \Pp \, : 
C^{3,\beta}_\sym(\Spheq\setminus D_L(\deltaLp)    , g)\|
\le
\, C \,(\deltaLp)^{-3-\beta}\,|\log\deltaLp|\,$.
Combining the above with \ref{LPtp}.i for $\Ptp$ 
we conclude that 
$$
\|\,\varphi\, : 
C^{3,\beta}_\sym(\Spheq\setminus D_L(\deltaLp)    , g)\|
\le
\, C \,( \, \cunder_1 \, m^{3+\beta} + m^{4+\beta} + 
\,(\deltaLp)^{-3-\beta}\,|\log\deltaLp|\, ) \, \tau\, .
$$
Using (iii) and $\deltaLp=\tau^\alpha$  we conclude (v) by assuming $m$ is large enough
(equivalently $\tau$ is small enough).
By \ref{LGp}.vii we conclude that $|\Pp|\le \,C\,\alpha\,m $ on  
$\Spheq\setminus D_L(\deltaLp)   $.
By \ref{con:alpha} we can assume $\alpha$ small enough 
so that (vi) follows by using \ref{LPtp}.ii and \ref{EmuV}.

Finally we prove (vii): 
\ref{con:one}.i follows from \ref{Edelta}, \ref{Eparameters}, and by choosing $m$ large enough.
\ref{con:one}.ii-iii are obvious. 
\ref{con:one}.iv follows from \ref{EPhip} by using \ref{LPtp}.i and \ref{EmuV}.
\ref{con:one}.v-vi follow from (v) and (vi). 
\ref{AEcal}  and \ref{EEcal} follow from \ref{LVVp}.v.  
\end{proof}

\section{LD and MLD solutions in the equator-poles case}
\label{Sequatorpole}
\nopagebreak

\subsection*{Basic definitions}
$\phantom{ab}$
\nopagebreak

We proceed now to study the LD and MLD solutions we need in the case that the catenoidal
bridges are located on the equatorial circle and the two poles.
The construction of these solutions parallels that of the LD and MLD solutions in the
two-circle case as presented in the previous section.
The main differences are that now we have two different catenoidal
bridges modulo the symmetries. 
On the other hand we have no horizontal forces and no horizontal sliding 
because the symmetries fix the location of the catenoidal bridges completely.
The configuration now
consists of $m+2$ points,
$m$ of which lie on the equator and the other two
are the poles: 
\begin{equation}
\label{EeqL}
\begin{gathered}
\Lep =
\Lep[m]:=
L_0   [m] \cup L_2,  
\qquad \text{ where }
\\
L_0   =
L_0   [m]:=
\Lmer[m]\cap\Cir_0
=\grouptwo \, p_0,
\quad
L_2   :=
\{p_N,p_S\} 
=\grouptwo \, p_2,
\end{gathered}
\end{equation}
where 
$p_0:=\Theta(0,0,0)=(1,0,0,0)$ and $p_2:=p_N$ (recall \ref{Eeqone}).
We define
\begin{equation}
\label{Edeltaz}
\delta_p:=\delta_0:= 1/9m \quad (p\in L_0   ),
\qquad
\delta_p:=\delta_2:= 1/100 \quad (p\in \{p_N,p_S\}).
\end{equation}
Clearly $\skernel_\sym[\Lep]$ is two-dimensional 
and spanned by
$W_j:=W_j[m] \in\skernel_\sym[L]$ 
for $j=0,2$, 
where $W_j$ is defined by requesting 
(recall \ref{Dskernel} and \ref{Eskernelsym})
that it is supported on 
$D_{L_j  }(2\delta_j)\setminus D_{L_j  }(\delta_j)$ and
satisfies on $D_{p_j}(2\delta_j)$ 
\begin{equation}
\label{EWWeq}
W_j:=\Lcalp\,
\Psibold\left[2\delta_j,\delta_j ; \dbold_{p_j} \right]
(G_{p_j},\,\log\delta_{j} \, \cos\circ\dbold_{p_j}\,). 
\end{equation}
Because of the symmetries each $\tau:\Lep\to \R$ we consider 
takes only two values: 
$\tau_0:=\tau_{p_0}$ taken on $L_0  $ 
and $\tau_2:=\tau_{p_N}$ taken on the poles.
In analogy with \ref{Etwovarphi} we have: 
\begin{definition}
\label{Eeqvarphi}
For $j=0,2$ we define 
LD solutions 
$ 
\Phi_j=\Phi_j[m]:=\varphi[\, L_j   [m], \,1,\,0] \in C^\infty_\sym(\Spheq\setminus L_j   ). 
$ 
We also define 
$V_j=V_j[m] \in \skernelv_\sym[\Lep]$ by $\Lcalp V_j=W_j$. 
(recall \ref{LLD}, \ref{Eskernelsym} and \ref{EWW}).  
\end{definition}
Clearly $\skernelv_\sym[\Lep]$ is spanned by $V_0$ and $V_2$.

\subsection*{The rotationally invariant parts}
$\phantom{ab}$
\nopagebreak

\begin{lemma}
\label{Leqphiave}
We have that 
$\phi_{eq}:=(\Phi_0[m])_\ave = \frac{m}2\,\sin|\xx| $
on 
$\Spheq$.
\end{lemma}

\begin{proof}
The proof is similar to the one for \ref{Lphiave} but simpler: 
Because of the smoothness on each hemisphere and the rotational symmetry it 
is clear that $\phi_\eq = A \sin |\xx|$ for some $A\in \R$. 
For $0<\epsilon_1<<\epsilon_2$ 
let 
$\Omega_{\epsilon_1,\epsilon_2}:=
D_{\Cir_0} (\epsilon_2) 
\setminus 
D_{L_0  }(\epsilon_1) 
$.
By integrating $\Lcalp \Phi_0=0$ on $\Omega_{\epsilon_1,\epsilon_2}$, 
integrating by parts, and
taking the limit as $\epsilon_1\to 0$ first
and then as $\epsilon_2\to 0$, 
we obtain using the logarithmic behavior near $L_0  $ that 
$
2\pi m = 4 \pi \, A  
$, 
which implies the lemma. 
\end{proof}

Note that if we extended the notation of \ref{Dphiunder} 
in the obvious way we would have 
$\phi_{eq}=
\junder[
1/2
;0]
$.

\begin{lemma}
\label{Lpphiave}
We have that 
(recall \ref{Dphieo})
\begin{equation*}
\Phi_2[m]= \phie= G_{p_N}+ (1-\log2)\phio
\in
C^\infty_\xxx(\Spheq\setminus\{p_N,p_S\} \,).
\end{equation*}
\end{lemma}

\begin{proof}
The second equality is just \ref{LGp}.iii and it implies 
clearly the first equality by the definitions and \ref{LLD}. 
\end{proof}

\subsection*{Estimates on $\Phi_0=\Phi_0[m]$}
$\phantom{ab}$
\nopagebreak

Since $\Phi_2$ is rotational invariant and well understood by 
\ref{Lpphiave}
and $(\Phi_0[m])_\ave = \phi_{eq}$ is well understood by \ref{Leqphiave},   
the main remaining task is estimating $(\Phi_0[m])_\osc$.
Our approach for this is similar to the one for estimating
$\Phi_\osc$ in the previous section, 
except that the situation is simplified by the extra symmetry.
In analogy with \ref{DV} and \ref{DW} we have now the following.

\begin{definition}
\label{DVeq}
Let $\Pp_0\in C^\infty_\sym(\Spheq\setminus L_0   )$, 
$\Ptp_0,E''_0,
\Ptp_\zave,\Ptp_\zosc,E''_\zave,E''_\zosc
\in C^\infty_\sym(\Spheq)$,  
$\Pp_2 \in C^\infty_\xxx(\Spheq\setminus\{p_N,p_S\} \,)$, 
and $\Ptp_2, E''_2 \in C^\infty_\xxx(\Spheq \,)$ 
be defined 
by 
requesting that for $j=0,2$ 
\begin{equation}
\label{EVeq}
\Pp_j:=
\{\,\Psibold[2\deltaz,3\deltaz; I_{\R^+}]
(G,A_j)
\,\}
\circ
\dbold_{L_j  },
\quad\text{where}\quad
A_j:=\,\log\delta_j ,
\end{equation}
\vspace{-.54cm}
\begin{equation}
\\
\label{EWeq}
\begin{gathered}
\Phi_j=\,\Pp_j+\Ptp_j,
\qquad
E_j'':=-\Lcalp \Pp_j
\qquad \text{ on } 
\Spheq\setminus  L_j,     
\\
\Ptp_\zave:=(\Ptp_0)_\ave, 
\,\,
\Ptp_\zosc:=(\Ptp_0)_\osc,
\,\,
E''_\zave :=(E''_0)_\ave,
\,\,
E''_\zosc :=(E''_0)_\osc,
\quad\text{ on } \Spheq.
\end{gathered}
\end{equation}
\end{definition}

Note that by \ref{EVeq} $E''_0$ vanishes on $D_{L_0   }(2\deltaz)$.
Moreover $E''_0$ is constant on
$\Spheq\setminus D_{\Cir_0}(3\deltaz)$
and therefore
$E''_\zosc$ is supported on $D_{\Cir_0}(3\deltaz)$.
Since $\Lcalp\Phi_0$ vanishes by \ref{LLD}
and \ref{Eeqvarphi},
and $\Lcalp$ is rotationally covariant, 
we conclude from \ref{EWeq}
that
\begin{equation}
\label{ELWeq}
\Lcalp \Ptp_0 = E''_0,
\qquad
\Lcalp \Ptp_\zave = E''_\zave,
\qquad
\Lcalp \Ptp_\zosc = E''_\zosc
\quad\text{ on } \Spheq.
\end{equation}
Moreover since 
$\Phi_\zave\in C^0_\sym(\Spheq)$
by \ref{LLD} and \ref{Eeqvarphi},
and 
$\Ptp_\zave\in C^\infty_\sym(\Spheq)$ by \ref{EWeq},
we conclude by \ref{EWeq} that 
\begin{equation}
\label{EWeqosc}
\begin{aligned}
\Pp_\zave\in C^0(\Spheq),
\qquad & 
\Ptp_\zave= \Phi_\zave-\Pp_\zave=\phi_{eq}-\Pp_\zave
\quad\text{ on }\quad
\Spheq,
\\
\Phi_0=&\,\phi_{eq}+\Pp_\zosc+\Ptp_\zosc
\quad\text{ on }\quad
\Spheq\setminus L_0   .
\end{aligned}
\end{equation}
Using \ref{EWeq} we conclude that
\begin{equation}
\label{EWposceq}
\Ptp_0=\phi_{eq}-\Pp_\zave + \Ptp_\zosc
\quad \text{ on } \quad
\Spheq .
\end{equation}
Note that in this expression although $\phi_{eq}$ and $\Pp_\zave$
are not smooth because of a derivative jump at the equator $\Cir_0$,
we do have
$\phi_{eq}-\Pp_\zave =\Ptp_\zave\in C^\infty_\sym(\Spheq)$
because the derivative jumps cancel out.
Note also that neither $\Lcalp \Ptp_\zave$ nor $\Lcalp \Ptp_\zosc$ 
have to vanish on $D_p(2\deltaL)$ but their sum does.

\begin{lemma}
\label{LHeq}
The following hold where 
$\Omega_{eq}:= D_{\Cir_0\,}(3/m)$.  
\newline
(i).
$\|\Pp_0 - A_0 :C_\sym^{k}(\Omega_{eq}\setminus D_{L_0   }( \deltaz) \,,\gtilde \, )\|\le C(k)\,$ 
and $\Pp_0-A_0$ vanishes on $\Spheq\setminus D_{L_0  }(3\deltaz)$.
\newline
(ii).
$\|m^{-2}E''_0:C_\sym^{k}(\Omega_{eq}\,,\gtilde \, )\|\le C(k)\,$, 
$\|m^{-2}E''_\zave:C_\sym^{k}( \, \Spheq ,\gtilde \, )\|\le C(k)\,$, 
and
$E''_0$ vanishes on $D_{L_0  }(2\deltaz)$. 
\newline
(iii).
$\|m^{-2}E''_\zosc:C_\sym^{k}( \, \Spheq ,\gtilde \, )\|\le C(k)\,$
and 
$E''_\zosc$ is supported on $D_{\Cir_0}(3\deltaz)\subset \Omega_\eq $.
\end{lemma}

\begin{proof}
(i) follows from \ref{LGp}.vi and the definitions.
By \ref{ELcalp} and \ref{EWeq} we have
$$
m^{-2}E''_0=
-\Lcalp_\gtilde \Pp_0 =
-\Lcalp_\gtilde ( \Pp_0 - A_0 ) - 2 m^{-2} A_0.
$$
As mentioned earlier $E''_0$ vanishes on $D_{L_0  }(2\deltaz)$ 
and clearly 
$-1 < m^{-2} A_0 < 0$ 
by \ref{Edeltaz}. 
(ii) follows then from (i). 
The second part of (iii) follows 
from (i) which implies that 
$\Pp_0=A_0$ on $\Spheq\setminus D_L(3\deltaz)$.
The first part of (iii) follows then from (ii) 
and the second part. 
\end{proof}

\begin{lemma}
\label{LWoscesteq}
$\Ptp_\zosc$ satisfies 
$\|\Ptp_\zosc:C^{2,\beta}_\sym(\Spheq,\gtilde, \fdecayz )\|
\le
C  $
where we have 
$\fdecayz:= e^{-c_1 m \min ( \, | \xx | \, ,\,  c_2\, ) } $ 
for some absolute constants $c_1,c_2>0$.
\end{lemma}

\begin{proof}
The proof is similar to the proof of \ref{LWoscest}, 
where we apply an appropriately modified version of \ref{Lsol} 
on \ref{LHeq}. 
Since the modifications are clear 
we omit the details.
\end{proof}

\begin{definition}
\label{DPhipppeq}
We define 
$\Phip_0\in C^\infty_\sym(\Omega_{eq})$
by requesting that 
$
\Ptp_0=\Phip_0 -A_0\phie 
$
on
$\Omega_{eq}$.
\end{definition}
\ref{DPhipppeq} corresponds to 
\ref{DPhippp} with $-A_0$ and $\phie$ corresponding to $\phi_1-A_1$ and $\phiunder[1,0;0]$. 
Using now \ref{EWeq} we obtain
\begin{equation}
\label{EPpppeq}
\Phi_0=
\Pp_0+ \Phip_0 -A_0\phie 
\quad\text{ on }\quad
\Omega_{eq}\setminus L_0   .
\end{equation}

\begin{lemma}
\label{Lpppeq}
$\Phip_0$ satisfies the estimate
$\|\Phip_0:C_\sym^{2,\beta}(\Omega_{eq}\,,\gtilde \, )\|\le \, 
C$.
\end{lemma}

\begin{proof}
By using the definitions we have 
$$
\Phip_\zave=
\phi_\eq
+ 
A_0\phie 
- 
\Pp_\zave ,
\qquad\qquad
\Phip_\zosc =\Ptp_\zosc,
\qquad\qquad
\Lcalp \Phip_\zave = \, E''_\zave . 
$$
Note that the discontinuities on the right hand side of the first equation cancel and the left hand side is smooth.
Using then that $\Pp_\zave = A_0$ on  
$\Spheq\setminus D_{\Cir_0}(3\deltaz)$
by \ref{LHeq}.i
and \ref{Dphieo} 
we conclude that at $\partial \Omega_\eq$ 
we have
$|\, \Phip_\zave \, | \, \le \, C \, $ 
and 
$|\, \partial_{\xxtilde}\,  \Phip_\zave \, | \, \le \, C \,$. 
Using these as initial data for the ODE 
and \ref{LHest}.ii 
we estimate $\Phip_\zave$.
By this estimate 
together with 
\ref{LWoscesteq} 
we conclude the proof. 
\end{proof}

\begin{lemma}
\label{LPlast}
(i).
$\|\, \Phi_0- \frac{m}2\,\sin|\xx| \, :C_\sym^{k,\beta}(\, \Spheq \setminus D_{\Cir_0} ( \delta_0) \, ,\gtilde \, )\|\le C(k) $.
\newline
(ii).
$\|\, \Phi_0 \, :C_\sym^{k,\beta}(\, \Omega_\eq \setminus D_{L_0  } (\delta_0) \, ,\gtilde \, )\|\le C(k) $ 
(recall $\Omega_\eq = D_{\Cir_0} (3/m) \supset D_{\Cir_0} (3\delta_0) $).
\newline
(iii).
$\| \, \Ptp_0 + A_0  \,  :C_\sym^{k,\beta}(\, \Omega_\eq \, , \, \gtilde \, )\|\le C(k) $.
\end{lemma}

\begin{proof}
By \ref{EWeqosc}, \ref{LHeq}.i, and \ref{Leqphiave} we have 
$\Phi_0- \frac{m}2\,\sin|\xx| =\Ptp_\zosc  + \, (\Pp_0 - A_0)_\osc  $ on 
$\Spheq \setminus D_{\Cir_0}( \deltaz)$. 
By \ref{LWoscesteq} \ref{LHeq}.i 
we conclude (i) for $k=2$.

By \ref{EPpppeq} we have 
$\Phi_0= \Pp_0 -A_0 + \Phip_0 -A_0\, (\phie -1)$ 
on $\Omega_{eq}\setminus L_0   $ 
and by \ref{DPhipppeq} 
$\Ptp_0+A_0 =\Phip_0 -A_0 \, (\phie -1)$ 
on
$\Omega_{eq}$.
(ii) and (iii) follow then for $k=2$ 
from \ref{LHeq}.i, \ref{Lpppeq}, and \ref{Dphieo}. 
Using interior regularity and \ref{LHeq}.iii we complete the proof. 
\end{proof}

\begin{corollary}
\label{CPlast}
(i). 
$\|\, \Phi_0 \, :C_\sym^{k,\beta}(\, \Spheq \setminus D_{L_0  } ( \delta_0) \, ,\gtilde \, )\|\le C(k) \, m $.  
\newline 
(ii).
$\|\, \Ptp_0 \, :C_\sym^{k,\beta}(\, \Spheq \, ,\gtilde \, )\|\le C(k) \, m $. 
\newline 
(iii).
$|\, \Phi_0- \frac{m}2\,\sin|\xx| \, | \,\le \, C$ on $\Spheq \setminus D_{L_0  } ( \delta_0) $.  
\end{corollary}

\begin{proof}
(i) follows from \ref{LPlast}.i,ii.
(ii) follows from (i) and \ref{LPlast}.iii 
by using also that on $\Spheq\setminus\Omega_\eq$ we have $\Ptp_0=\Phi_0-A_0$
and $|A_0| < m$. 
(iii) follows from \ref{LPlast}.i,ii.
\end{proof} 

\begin{lemma}
\label{LPtwo}
(i).
$\|\, \Phi_2 \, :C_\sym^{k,\beta}(\, \Spheq \setminus D_{L_2} ( \delta_2) \, , g \, )\|\le C(k) $.
\newline
(ii). 
$\| \, \Ptp_2 \,  :C_\sym^{k,\beta}(\, \Spheq \, , \, \gtilde \, )\|\le C(k) $ 
and moreover $\Ptp_2=\, (1-\log 2) \,\sin|\xx|$ on $D_{L_2} ( \delta_2)$.  
\end{lemma} 

\begin{proof}
This follows easily from the definitions and \ref{Lpphiave}. 
\end{proof}

\subsection*{Estimates on $V_0,V_2,W_0,W_2$}
$\phantom{ab}$
\nopagebreak

\begin{lemma}
\label{LVVpeq}
$V_0,V_2\in C^\infty_\sym(\Spheq)$ satisfy the following.
\newline
(i).
We have 
$V_0=\Phi_0$ 
on 
$\Spheq\setminus D_{L_0   }(2\deltaz)$
and 
$V_2=\Phi_2=\phie$ 
on 
$\Spheq\setminus D_{\{p_N,p_S\}}(2\deltat)$. 
\newline
(ii).
We have
$$
\begin{aligned}
V_0&=
\Psibold\left[2\delta_0,\delta_0 ; \dbold_{p_0} \right]
(G_{p_0},\,\log\delta_{0} \, \cos\circ\dbold_{p_0}\,)
+\Phip_0
-A_0\phie
\quad
\text{ on }
D_{p_0}(2\deltaz),
\\
V_2&=
\Psibold\left[2\delta_2,\delta_2 ; \dbold_{p_N} \right]
(G_{p_N},\,\log\delta_{2} \, \cos\circ\dbold_{p_N}\,)
+(1-\log 2)\phio
\quad
\text{ on }
D_{p_N}(2\deltaz).
\end{aligned}
$$
(iii).
$V_0(p_0)=\Phip_0(p_0)$,
$\quad V_0(p_N)={m}/2$,
$\quad V_2(p_0)=1$,
$\quad V_2(p_N)= 1+\log(\delta_2/2)$.
Moreover 
$|V_0(p_0)| \le \, C$ and 
$|V_2(p_N)| \le \, C$.  
\newline
(iv). 
$\|V_0: C_\sym^{k}(\,    D_{L_0  } ( \delta_0)     \, ,\gtilde \, )\| \, \le  \, C(k) $, 
$\|V_0: C_\sym^{k}(\,\Spheq\,,\gtilde \, )\| \, \le  \, C(k) \, m     $, \hfill 
\newline
and 
$\|V_2: C_\sym^{k}(\,\Spheq\,,\gtilde \, )\| \, \le  \, C(k) $.
\newline
(v). 
\ref{AEcal}  holds and 
$\|\Ecalinvep \| \,  \le \, C \, m^{2+\beta} \, $ (recall \ref{DEcalnorm}). 
\newline
(vi). 
For $\mu_0,\mu_2\in\R$ we have that 
$|\mu_0| m + |\mu_2| \, \le \,C \, \|\, \mu_0 W_0 + \mu_2 W_2  \, : \, C_\sym^{0,\beta}(\,\Spheq\,,g \, )\| $.
\end{lemma}

\begin{proof}
The proof is similar in structure to the one for \ref{LVVp}. 
Let $V_{0,new}$ and $V_{2,new}$ be defined by the expressions for $V_0$ and $V_2$ in (i) and (ii). 
The expressions for $V_{0,new}$ and $V_{2,new}$ in (i) and (ii) match because by 
\ref{DVeq}, \ref{EPpppeq}, and \ref{Lpphiave} we have 
$V_{0,new}=\Phi_0= \Pp_0+ \Phip_0 -A_0\phie $ 
on a neighborhood of $\partial D_{p_0}(2\delta_0)$ 
and 
$V_{2,new}=\Phi_2=\phie = G_{p_N}+ (1-\log2)\phio$ 
on a neighborhood of $\partial D_{p_N}(2\delta_2)$. 
Since $\Lcalp \Phi_0 = \Lcalp \Phi_2 =  0$ by 
\ref{Eeqvarphi}, 
we conclude that by \ref{EWWeq} we have 
that 
$\Lcalp V_{0,new}=W_0$ and $\Lcalp V_{2,new}= W_2$,
which characterize 
by \ref{LLD} $V_0$ and $V_2$.  
This completes the proof of (i) and (ii).

The equalities in (iii) now follow from (i), (ii), \ref{Leqphiave}, and \ref{Dphieo}, 
where we used also $A_0=\log\delta_0$ from \ref{EVeq}. 
The estimates in (iii) follow from \ref{Lpppeq} and \ref{Edeltaz}. 
(iv) is implied by (i), (ii), \ref{CPlast}, \ref{Dphieo}, \ref{LPtwo}, and $A_0=\log\delta_0$.
Note now that if we use $\{V_0,V_2\}$ as the basis for $\skernelv_\sym[L_\ep]$ and 
the standard basis for $\val_\sym[L_\ep]$, then the entries for the matrix of 
$\Ecal_{L_\ep}$, defined as in \ref{DEcal}, are given in (iv), 
and therefore using (iv) we can easily check that (v) holds.  
(vi) follows by the same argument we used for \ref{LVVpeq}.vi. 
\end{proof}

\subsection*{The family of MLD solutions}
$\phantom{ab}$
\nopagebreak

We discuss now the family of MLD solutions which is converted to a family of
initial surfaces by \ref{Dinit}.
The parameters of the family are 
$\zetabold=(\zeta_0,\zeta_2)\in\R^2$
and their range is specified by
\begin{equation}
\label{Ezetarangeeq}
|\zeta_0|,|\zeta_2|\le \cunder_2,
\end{equation}
where $\cunder_2>1$ is a 
constant independent of $m$ and $\tau$
which will be specified later.
Given $(\zeta_0,\zeta_2)$ as in \ref{Ezetarange}
we define $\tau_0,\tau_2$ by
\begin{equation}
\label{Eparameterseq}
\tau_0 = \tau_0[\zetabold,m]
:= m^{-3/4} \, e^{\zeta_0} \, e^{-{\sqrt{m/2}}},
\quad
\tau_2:= \tau_2[\zetabold,m]
=\tau_0\,\left(\zeta_2-\frac14\log m+\sqrt{\frac{m}2\,}\,\,\right).
\end{equation}
This definition is motivated by a straightforward calculation
where various error terms are ignored.
We skip this calculation because it is not needed for the proof
and is similar to a precise calculation we present in the proof of \ref{L:heq}.

\begin{lemma}
\label{L:heq}
For $m$ large enough depending on $\cunder_2$
and
$\zetabold=(\zeta_0,\zeta_2)$ as in \ref{Ezetarangeeq},
we define $\taubold:\Lep\to\R^+$
to take the values $\tau_0$ on $L_0   $
and $\tau_2$ on $\{p_N,p_S\}$,
with $\tau_0,\tau_2$
defined as in 
\ref{Eparameterseq},
and also 
$
\xiw=\xiw[\zetabold,m] 
:=
\mu_0[\zetabold,m]\, W_0[m]
+
\mu_2[\zetabold,m]\, W_2[m]
$
defined uniquely by the requirement that
\begin{equation}
\label{E:heq}
\varphi= \varphi[[\zetabold,m]]
:=\varphi[\, \Lep[m] \,,\,\taubold[\zetabold,m]\,,\,\xiw[\zetabold,m]\,]
= \tau_0\Phi_0+\tau_2\Phi_2+\mu_0 V_0 + \mu_2 V_2
\end{equation}
is an MLD solution as in \ref{DMLD}.
Moreover the following hold.
\newline
(i).
$\taubold=\taubold[\zetabold,m]$ 
and
$\mubold=\mubold[\zetabold,m]:=(\mu_0,\mu_2)$
depend continuously on $\zetabold$.
\newline
(ii).
$ \left | \zeta_0 \, + \, 
\frac{ \mu_0 (m/2)^{ 1/2} + \mu_2 } {2 \, \tau_0} \right | \, \le \, C $ 
and 
$ \left | \zeta_2 \, + \, 
\frac{ - \mu_0 (m/2)^{ 1/2} + \mu_2 } {2 \, \tau_0} \right | \, \le \, C $ . 
\newline
(iii).
$ \| \, \varphi[[\zetabold,m]] : 
C^{3,\beta}_\sym(\Spheq\setminus D' , g)\|
\le 
\tau_0^{1-4\gammagl} 
\le 
\tau_0^{8/9}$ 
where $D' := \disjun_{j=0,2} D_{L_j  }(\delta'_j)   $  
where $\delta'_j:=\tau_j^\alpha$.  
\newline
(iv).
$c\,\tau_2 \, \le \, \varphi$ on 
$\Spheq\setminus D' $ 
for some absolute constant $c>0$.
\newline
(v).
$\varphi$ satisfies the conditions in \ref{con:one},  
\ref{AEcal}, and \ref{EEcal}.
\end{lemma}

\begin{proof}
Note that by 
\ref{DLDnoK}.ii and \ref{DVeq} for $\varphi$ as in \ref{E:heq} we have that 
\begin{equation}
\label{EPhipeq}
\begin{aligned}
\varphihat_{p_0}\, =& \,  \tau_0\Ptp_0+\tau_2\Phi_2+\mu_0 V_0 + \mu_2 V_2,
\qquad
\text{ on } \quad
D_{p_0}(2\deltaz)\, ,
\\
\varphihat_{p_N}\, =& \,  \tau_0\Phi_0+\tau_2\Ptp_2+\mu_0 V_0 + \mu_2 V_2,
\qquad
\text{ on } \quad
D_{p_N}(2\delta_2)\, ,
\end{aligned} 
\end{equation}
where motivated by \ref{Lpphiave} we define 
$\Ptp_2 := (1-\log2)\phio$. 
Using \ref{DPhipppeq}, \ref{Lpphiave} we calculate
\begin{equation}
\Ptp_0(p_0)=
\,\Phip_0(p_0) - A_0, 
\quad
\Phi_2(p_0)=1,
\quad
\Phi_0(p_N)=m/2,
\quad
\Ptp_2(p_N)=1-\log 2.
\end{equation}
By straightforward calculation using 
\ref{Eparameterseq}
we obtain 
\begin{equation}
\label{Elogeq}
\log\tau_0=\zeta_0-\frac34\log m-\sqrt{\frac{m}2\,},
\qquad\qquad
\log\tau_2=\zeta_0-\frac14\log m-\sqrt{\frac{m}2\,}+
O(1),
\end{equation}
where in this proof we use $O(1)$ to denote terms which are uniformly bounded (independently of $\cunder_2$) 
as $m\to\infty$.
Using the above, \ref{LVVpeq}.iii, and  \ref{Lpppeq}, 
we calculate the matching condition in \ref{DMLD} amounts to 
$$
\begin{gathered}
\tau_0 \, (\zeta_0+\zeta_2 +O(1) \, )
+O(1) \, \mu_0 
+\mu_2 
=0,
\\
\tau_0\,(\zeta_0-\zeta_2+O(1)\, ) \, \, \sqrt{\frac{m}2\,} 
+ \, \frac{m}2 \mu_0 \, + 
O(1) \, \mu_2 \,
=0.
\end{gathered}
$$
Solving this linear system for $\mu_0,\mu_2$ we obtain its unique solution given by
$$
\mu_0=
-\tau_0\, (\zeta_0-\zeta_2+O(1)\,) \,\,(m/2)^{-1/2},
\qquad\qquad
\mu_2=
-\tau_0\, (\zeta_0+\zeta_2+O(1)\,), 
$$
where we assumed that $m$ is large enough in terms of $\cunder_2$.

The above clearly imply (i) and (ii).
They also imply that 
\begin{equation}
\label{Emueq} 
|\mu_0| \, \sqrt{m} \, + \, |\mu_2| \, \le \, C\,\cunder_2\,\tau_0,
\end{equation} 
which together with \ref{LVVpeq}.iv implies that 
\begin{equation}
\label{EmuVeq} 
\| \, \mu_0 V_0 + \mu_2 V_2 : C_\sym^{3,\beta}(\,\Spheq\,,\gtilde \, )\|\le C \, \cunder_2 \, \sqrt{m} \, \tau_0 \, .      
\end{equation} 
Using \ref{LGp}.vii we obtain that 
$ \| \, \Pp_j \, : 
C^{3,\beta}_\sym(\Spheq\setminus D' , g)\|
\le
\, C \,(\delta'_j)^{-3-\beta}\,|\log\delta'_j|\,$ 
for $j=0,2$. 
Note that we have on $\Spheq\setminus D'$ that 
$$
\varphi = \tau_0\Pp_0+\tau_2\Pp_2 + \tau_0\Ptp_0+\tau_2\Ptp_2 + \mu_0 V_0 + \mu_2 V_2.
$$
Combining the above with \ref{LPtwo}.ii and \ref{CPlast}.ii 
we conclude (iii) by assuming $m$ large enough. 

To prove (iv) observe that on $D_{L_0}(\delta_0) \setminus D_{L_0}(\delta'_0)$ 
we have  
$
\varphi = \tau_0\Pp_0+  \tau_0\Ptp_0+ ( \tau_2+\mu_2) \Phi_2 + \mu_0 V_0 , 
$
on $D_{L_2}(\delta_2) \setminus D_{L_2}(\delta'_2)$ 
we have  
$
\varphi = (\tau_0 + \mu_0) \Phi_0 + \tau_2\Pp_2 + \tau_2\Ptp_2 + \mu_2 V_2,
$
and on $\Spheq\setminus \disjun_{j=0,2} D_{L_j}(\delta_j)$ 
we have   
$
\varphi = (\tau_0 + \mu_0) \Phi_0 + 
(\tau_2 + \mu_2) \Phi_2 . 
$
Using \ref{Emueq} and \ref{Eparameterseq} 
we obtain bounds for the coefficients.
(iv) on $\Spheq\setminus \disjun_{j=0,2} D_{L_j}(\delta_j)$ follows then 
by using \ref{Lphieo}.  
By \ref{LGp}.vii and \ref{Elogeq} we have 
$| \,  \Pp_j \, | \le \, C \, \alpha \, |\log \tau_j |
\, | \le \, C \, \alpha \, \sqrt{m}$ 
on $D_{L_j}(\delta_j) \setminus D_{L_j}(\delta'_j)$  for $j=0,2$. 
Using \ref{CPlast}.iii, \ref{LPtwo}, and \ref{LVVpeq}.iv 
to estimate the remaining terms,  
we conclude the proof of (iv) 
by assuming $\alpha$ small enough as in \ref{con:alpha}.

Finally we prove (v): 
\ref{con:one}.i follows from \ref{Edeltaz}, \ref{Eparameterseq}, and by choosing $m$ large enough.
\ref{con:one}.ii-iii are obvious from the definitions and choosing $m$ large enough. 
\ref{con:one}.iv follows from \ref{EPhipeq} by using \ref{CPlast}.i,ii, \ref{LPtwo}, and \ref{EmuVeq}.
\ref{con:one}.v-vi follow from (iii) and (iv). 
\ref{AEcal}  and \ref{EEcal} follow from \ref{LVVpeq}.v.  
\end{proof}

\begin{remark}
\label{rem:eq}
Note that if we only had bridges on the equatorial circle then \ref{E:heq} 
would have to be replaced by 
``$
\varphi = \tau_0\Phi_0+\mu_0 V_0 
$''.
Clearly then it would be impossible to satisfy the vertical matching condition and construct an 
MLD solution in this way. 
\hfill $\square$
\end{remark}

\section{Main results}
\label{Smain}
\nopagebreak

\begin{theorem}[The two parallel circles case]
\label{Ttwocir} 
There is an absolute constant $\cunder_1>0$ such that if $m$ is large enough
depending on $\cunder_1$, 
then there is ${\zetaboldhat}=(\widehat{\zeta},\widehat{\zeta}')\in \R^2$ 
satisfying \ref{Ezetarange}  
such that (in the notation of \ref{L:h}) 
$\xxhat_1:=\xx_1[{\zetaboldhat} ,m] $, 
$\tauhat:=\tau[{\zetaboldhat} ,m]$, 
$ \xiwhat :=\xiw[\zetaboldhat,m]$,  
and 
$ \varphi[[\zetaboldhat,m]]$,  
satisfy \ref{L:h}.ii-vii, 
and moreover 
there is $\phihat\in C^\infty (\Mhat)$, 
where $\Mhat:=M[ \, L[\xxhat_1,m]\,,\,\tauhat\,,\, \xiwhat\,]$ in the notation of 
\ref{Dinit}, 
such that 
in the notation of \ref{D:norm} 
$$
\| \, \phihat \, \|_{2,\beta,\gamma;\Mhat} \, \le \, \tauhat^{1+\alpha/4}, 
$$
and furthermore $\Mhat_\phihat$ (in the notation of \ref{Lquad}) 
is a genus $2m-1$ embedded minimal surface 
in $\Sph^3(1)$, 
which is invariant under the action of $\groupthree$ and has area 
$Area( \Mhat_\phihat ) \to 8\pi$ as $m\to\infty$.
\end{theorem}

\begin{proof}
\textit{Step 1: Construction of the diffeomorphisms $F_\zetabold$:}
We fix an $m\in\N$ which we assume as large in terms of $\cunder_1$ as needed.
We will use the notation 
$\zerobold:=(0,0)\in\R^2$,
$\tb:=\tau[\zerobold,m]$, 
and for $\zetabold\in\R^2$ satisfying \ref{Ezetarange} 
$ M[[\zetabold]] := 
M [\, L[ \,\xx_1[\zetabold,m]\, ,m] \,,\,\tau[\zetabold,m]\,,\,\xiw[\zetabold,m]\,] $ 
(recall \ref{L:h} and \ref{Dinit}) 
and 
$ L[[\zetabold]] := L[ \,\xx_1[\zetabold,m]\, ,m] $. 
We define for $\zetabold\in\R^2$ satisfying \ref{Ezetarange} 
a smooth diffeomorphism 
$F_{\zetabold} : M[[ \zerobold ]] \to M[[\zetabold]]$, 
covariant under the action of $\groupthree$, 
as follows. 
We start by constructing 
smooth diffeomorphisms 
${ F'_{\xx_1} : \Spheq \to \Spheq } $
which depend smoothly on $\xx_1$,  
are  covariant under the action of $\grouptwo$, 
and satisfy the following.
\newline
(a). 
${ F'_{\xx_1} ( L[[  \zerobold ]] \, )
\,=\, 
L[\xx_1,m] \, }$
and moreover if $p\in L[[ \zerobold ]] $,
then 
${ F'_{\xx_1} (p)}$ is the nearest point in 
${ L[\xx_1,m] \, }$ to $p$ 
(which amounts to being on the same side of the equator and of the same longitude). 
\newline
(b). 
$\forall p\in L[[ \zerobold ]] $ 
we have on $D_p(4\deltaL)$ that 
$F'_{\xx_1} = R[\xx_1,p]$, 
where $R[\xx_1,p] \in\text{SO}(3)$ is characterized by 
${ R[\xx_1,p] (p) = F'_{\xx_1} (p)}$ 
(as defined in (a) above),  
and 
${ d_p R[\xx_1,p] (\nabla_p\xx) = \nabla_{F'_{\xx_1} (p)}\xx}$. 
\newline
(c). 
If $q=\Theta(\xx,\yy,0) \in 
D_{\Lpar[\xx_1[\zerobold,m]]}(8\deltaL) \setminus D_{L[[ \zerobold ]]}(5\deltaL)$ 
with $\xx\in (0,\pi)$ (recall \ref{ETheta}), 
then 
$F'_{\xx_1}(q) \,=\, \Theta(\xx+\xx_1- \, \xx_1[\zerobold,m]\,,\yy,0)$.
\newline
(d). 
On $\Spheq \setminus D_{L[[ \zerobold ]]}(5\deltaL)$ 
$F'_{\xx_1}$ is rotationally covariant in the sense that 
it maps a point $\Theta(\xx,\yy,0)$ to $\Theta(\,f_{\xx_1}(\xx)\, ,\yy,0)$  for 
a suitably chosen function $f_{\xx_1}$. 
Note this is consistent with (c) where $f_{\xx_1}$ is implicitly specified on a smaller region. 
\newline
(e). 
On $D_{L[[ \zerobold ]]}(5\deltaL) \setminus D_{L[[ \zerobold ]]}(4\deltaL)$ 
we interpolate between the definitions in (b) and (c) by using cut-off functions.

By choosing $f_{\xx_1}$ carefully we can ensure that $F'_{\xx_1}$ depends smoothly 
on $\xx_1$ and is close to the identity in all necessary norms. 
We proceed now to use $F'_{\xx_1[\zetabold,m]}$ to define $F_\zetabold$ 
by requesting the following.
\newline
(f). 
$\forall p \in L[[\zerobold]]$  we define $F_\zetabold$ 
to map 
$\Lambda_\zerobold:= \Shat_1  [ p ] \subset \, M[[\zerobold]]$ 
onto 
$\Lambda_\zetabold:= \Shat_1  [ \, q \, ] \subset \, M[[\zetabold]]$, 
where $q:=F'_{\xx_1[\zetabold,m]}(p)$, 
and to satisfy on $\Lambda_\zerobold$ (recall \ref{DPiK} and \ref{Shatp}) 
$$
\widehat{F}_{\zetabold} \circ Y_\zerobold \circ \Pi_{\cat,p}
=
Y_\zetabold \circ \Pi_{\cat,q} \circ F_\zetabold,
$$
where $Y_\zetabold$ (and similarly for $Y_\zerobold$) is the conformal isometric from 
$\Pi_{\cat,q} ( \Lambda_\zetabold )$ equipped with the induced metric from the Euclidean metric 
$\tau^{-2}[\zetabold,m] \left. g\right|_p $, 
to the cylinder $[-\ell_\zetabold,  \ell_\zetabold ] \times \Sph^1(1)$ 
equipped with the standard flat metric,
and 
\begin{equation} 
\label{EFhat} 
\widehat{F}_{\zetabold}: 
[-\ell_\zerobold,  \ell_\zerobold ] \times \Sph^1(1)
\to 
[-\ell_\zetabold,  \ell_\zetabold ] \times \Sph^1(1),
\end{equation} 
is of the form in standard coordinates 
$$
\widehat{F}_{\zetabold} (t,\theta)
=
( \, \ell_\zetabold \, t \,/ \,\ell_\zerobold , \, \theta\,), 
$$
where the ambiguity due to possibly modifying the $\theta$ coordinate by adding a constant is 
removed by the requirement that $F_\zetabold$ is covariant with respect to the action 
of $\groupthree$. 
\newline
(g). 
We define now the restriction of $F_\zetabold$ 
on 
$M[[\zerobold]]\setminus \Shat[\, L[[\zerobold]] \, ]=
M[[\zerobold]] \cap \PiSph^{-1} (\,\Spheq\setminus D_{L[[\zerobold]]} (\delta'_1)\,)$ 
to be a map onto 
$M[[\zetabold]]\setminus \Shat[\, L[[\zetabold]] \, ]=
M[[\zetabold]] \cap \PiSph^{-1} (\,\Spheq\setminus D_{L[[\zetabold]]} (\delta'_1)\,)$ 
which preserves the sign of the $\zz$ coordinate and satisfies 
$$
\PiSph \circ F_\zetabold 
=
F'_{\xx_1[\zetabold,m]} \circ \PiSph. 
$$
(h). 
On the region 
$\Shat[\, L[[\zerobold]] \, ] \setminus \Shat_1[\, L[[\zerobold]] \, ] \subset \, M[[\zerobold]]$ 
we apply the same definition as in (g) but with $F'_{\xx_1[\zetabold,m]}$ appropriately modified 
by using cut-off functions and $\dbold_{L[[\zetabold]]}$ 
so that the final definition 
provides an interpolation between (f) and (g). 

\textit{Step 2: Equivalence of norms under $F_\zetabold$:} 
Using \ref{L:h}.iii and 
\ref{Ecatenoid} 
it is easy to check that 
$$
\ell_\zetabold \, \sim_{1+C(\cunder_1)\,/m } \ell_\zerobold.
$$
Using this and arguing as in the proof of \ref{L:norms} we conclude that 
for $u\in C^{2,\beta}(\,M[[\zetabold]]\,)$ 
and $E\in C^{0,\beta}(\,M[[\zetabold]]\,)$ 
we have 
$$
\| \, u\circ F_\zetabold \, \|_{2,\beta,\gamma;M[[\zerobold]] }
\, \sim_2 \, 
\| \, u\, \|_{2,\beta,\gamma;M[[\zetabold]] } ,
\qquad
\| \, E\circ F_\zetabold \, \|_{0,\beta,\gamma-2;M[[\zerobold]] }
\, \sim_2 \, 
\| \, E\, \|_{0,\beta,\gamma-2;M[[\zetabold]] } .
$$

\textit{Step 3: The map $\Jcal$:}
We define now
a map $\Jcal : B \to C^{2,\beta}_\sym(\,M[[\zerobold]]\,)\times\R^2$, where 
$$
B:=\{ \, v\in C^{2,\beta}_\sym(M[[\zerobold]]):\|v\|_{2,\beta,\gamma;M[[\zerobold]]} \, \le \, \tb^{1+\alpha}\,\} 
\times [-\cunder_1,\cunder_1]^2 \subset 
C^{2,\beta}_\sym(\,M[[\zerobold]]\,)\times\R^2,
$$
as follows:
We assume $(v,\zetabold)\in B$ given.
By \ref{L:h}.vii we can apply \ref{Plinear} to obtain 
$(u,\zw_H) := -\Rcal_{M[[\zetabold]]}(\, H-\xiw \circ \PiSph \,)$. 
We define then 
$\phi\in C^{2,\beta}(M[[\zetabold]])$ by
$\phi:= v\circ F_\zetabold^{-1} +u$. 
We have then
\newline
(j).
$\Lcal u + H = ( \xiw + \zw_H) \circ \PiSph$.
\newline
(k). By the definition of B, \ref{L:h}.iii, \ref{Plinear}, \ref{LHM}, and \ref{LVVp}.v we obtain 
$$
\|\zw_H  : C^{0,\beta}(\Spheq,g)\|
+ \|\phi\|_{2,\beta,\gamma;M[[\zetabold]]}\, \le \, \tau^{1+\alpha/4}.
$$

Applying 
\ref{Lquad} and 
\ref{Plinear}
we obtain
$(u_Q,\zw_Q) \, := \, -\Rcal_{M [[\zetabold]]} ( H_\phi-H-\Lcal\phi) $
which satisfies the following:
\newline
(l).
$\Lcal u_Q + H_\phi  =  H + \Lcal \phi + \zw_Q  \circ \PiSph$.
\newline
(m).
$
\|\zw_Q  : C^{0,\beta}(\Spheq,g)\|
+ \| \,u_Q\, \|_{2,\beta,\gamma;M[[\zetabold]]}\, \le \, \tau^{2+\alpha/4}.
$
\newline
(n).
$\Lcal ( \, u_Q \, - \, v\circ F_\zetabold^{-1} \, )           + H_\phi \,  = \, ( \xiw + \zw_H + \zw_Q ) \circ \PiSph\,, \quad$
which follows by combining the definition of $\phi$ with (j) and (l).

This motivates us to define
$$
\Jcal(v,\zetabold)=\left(u_Q\circ F_\zetabold\, , \, \zetabold + \frac1{\tau[\zetabold,m]}\,  ( \, \mu_{sum} \, \phi_1\, , \, \mu'_{sum} \, ) 
\right),
$$
where $\mu_{sum} W + \mu'_{sum} W' = \xiw + \zw_H + \zw_Q$.

\textit{Step 4: The fixed point argument:}
By using (k), (m), \ref{LVVp}.vi and \ref{L:h}.iv,
and by choosing $\cunder_1$ large enough in terms of an absolute constant,
it is straightforward to check that $\Jcal(B)\subset B$.
$B$ is clearly a compact convex subset of $C^{2,\beta'}_\sym( \,M[[\zerobold]]\, )\times\R^2$
for $\beta'\in(0,\beta)$,
and it is easy to check that $\Jcal$ 
is a continuous map in the induced topology.
By Schauder's fixed point theorem
\cite[Theorem 11.1]{gilbarg} then,
there is a fixed point $(\vhat,\zetaboldhat)$ of $\Jcal$, 
which therefore satisfies $\vhat=\uhat_Q\circ F_\zetaboldhat$ and  
$\xiwhat + \zwhat_H + \zwhat_Q=0$, 
where we use ``$\widehat{\phantom{\xiwhat}}$'' to denote the various quantities 
for $\zetabold  = \zetaboldhat$ and $v=\vhat$.
By (n) then we conclude the minimality of $\Mhat_\phihat$. 
The smoothness follows from standard regularity theory and the embeddedness 
from \ref{Lquad} and (k).
The genus follows because we are connecting two spheres with $2m$ bridges. 
Finally the limit of the area as $m\to\infty$ follows from the available estimates 
for 
$\varphi_{init} [ \, L[\xxhat_1,m]\,,\,\tauhat\,,\, \xiwhat\,]$ 
and the bound on the norm of $\varphihat$. 
\end{proof}

\begin{theorem}[The equator and poles case]
\label{Teq-pol} 
There is an absolute constant $\cunder_2>0$ such that if $m$ is large enough
depending on $\cunder_2$, 
then there is ${\zetaboldhat}=(\widehat{\zeta}_0,\widehat{\zeta}_2)\in \R^2$ 
satisfying \ref{Ezetarangeeq}  
such that (in the notation of \ref{Eparameterseq} and \ref{L:heq}) 
$\tauhat_j:=\tau_j[{\zetaboldhat} ,m]$ 
for $j=0,2$, 
$\tauboldhat:=\taubold [{\zetaboldhat} ,m]$, 
$ \xiwhat :=\xiw[\zetaboldhat,m]$,  
and 
$ \varphi[[\zetaboldhat,m]]$,  
satisfy \ref{L:h}.ii-v, 
and moreover 
there is $\phihat\in C^\infty (\Mhat)$, 
where $\Mhat:=M[ \, \Lep[m] \,,\,\tauboldhat \,,\, \xiwhat \,] $
in the notation of \ref{Dinit}, 
such that 
in the notation of \ref{D:norm} 
$$
\| \, \phihat \, \|_{2,\beta,\gamma;\Mhat} \, \le \, \tauhat^{1+\alpha/4}, 
$$
and furthermore $\Mhat_\phihat$ (in the notation of \ref{Lquad}) 
is a genus $m+1$ embedded minimal surface 
in $\Sph^3(1)$, 
which is invariant under the action of $\groupthree$ and has area 
$Area( \Mhat_\phihat ) \to 8\pi$ as $m\to\infty$.
\end{theorem}

\begin{proof}
The proof has the same structure as the proof of \ref{Ttwocir} and so we only 
provide a brief outline emphasizing the differences: 

\textit{Step 1: Construction of the diffeomorphisms $F_\zetabold$:}
We fix an $m\in\N$ which we assume as large in terms of $\cunder_2$ as needed.
We will use the notation 
$\zerobold:=(0,0)\in\R^2$,
$\tbbold:=\taubold[\zerobold,m]$ taking the value 
$\tb_j:= \tau_j [\zerobold,m]$ on $L_j$ for $j=0,2$, 
and for $\zetabold\in\R^2$ satisfying \ref{Ezetarangeeq} 
we write 
$ M[[\zetabold]] := 
M [\, \Lep [m] \,,\,\taubold[\zetabold,m]\,,\,\xiw[\zetabold,m]\,] $ 
(recall \ref{L:heq} and \ref{Dinit}). 
It is easy to modify the definition of $F_{\zetabold}$  
in the proof of \ref{Ttwocir} 
to define for $\zetabold\in\R^2$ satisfying \ref{Ezetarangeeq} 
a smooth diffeomorphism 
$F_{\zetabold} : M[[ \zerobold ]] \to M[[\zetabold]]$ 
covariant under the action of $\groupthree$. 
Note that actually the definition is simpler because $\Lep$ does 
not depend on $\zetabold$ and therefore we can skip the initial steps 
concerning the diffeomorphisms 
$F'_{\xx_1}$. 
The substantial step is to define maps analogous to 
the $\widehat{F}_{\zetabold}$'s which were defined in \ref{EFhat}. 
In analogy we denote (half) the lengths of the corresponding cylinders by 
$\ell_\zerobold[p]$ and $\ell_\zetabold[p]$, where $p\in \Lep$ is mentioned because 
the lengths depend on whether $p$ is a pole or on the equator.

\textit{Step 2: Equivalence of norms under $F_\zetabold$:} 
Using \ref{Elogeq} and \ref{Ecatenoid} 
it is easy to check that 
$$
\ell_\zetabold[p] \, \sim_{1 \, + \, C\,m^{-1/2}\,\log m\, }  \, \ell_\zerobold[p].
$$
Using this and arguing as in the proof of \ref{L:norms} we conclude that 
for $u\in C^{2,\beta}(\,M[[\zetabold]]\,)$ 
and $E\in C^{0,\beta}(\,M[[\zetabold]]\,)$ 
we have 
$$
\| \, u\circ F_\zetabold \, \|_{2,\beta,\gamma;M[[\zerobold]] }
\, \sim_2 \, 
\| \, u\, \|_{2,\beta,\gamma;M[[\zetabold]] } ,
\qquad
\| \, E\circ F_\zetabold \, \|_{0,\beta,\gamma-2;M[[\zerobold]] }
\, \sim_2 \, 
\| \, E\, \|_{0,\beta,\gamma-2;M[[\zetabold]] } .
$$

\textit{Step 3: The map $\Jcal$:}
By applying \ref{L:heq}.v, \ref{Elogeq}, and \ref{LVVpeq}.v we can repeat 
all definitions and estimates in step 3 of of the proof of \ref{Ttwocir}, 
except for using $\cunder_2$ instead of $\cunder_1$ and modifying the 
definition of $\Jcal$ as follows:
$$
\Jcal(v,\zetabold)=\left(u_Q\circ F_\zetabold\, , \, \zetabold + 
\left( \, \frac{ \mutilde_0 (m/2)^{ 1/2} + \mutilde_2 } {2 \, \tau_0[\zetabold,m] } \, , \, 
\frac{ - \mutilde_0 (m/2)^{ 1/2} + \mutilde_2 } {2 \, \tau_0[\zetabold,m] }  \, \right) \, 
\right),
$$
where $\mutilde_0 W_0 + \mutilde_2 W_2 = \xiw + \zw_H + \zw_Q$.

\textit{Step 4: The fixed point argument:}
Using \ref{LVVpeq}.vi and \ref{L:heq}.ii
we can argue in the same way as in the proof of \ref{Ttwocir} to complete the proof.
\end{proof}

\def\baselinestretch{1}

\bibliographystyle{amsplain}
\bibliography{paper}
\end{document}